\documentclass[12pt]{amsart}

\usepackage{geometry}
\newcommand{\calum}[1]{{\textcolor{blue}{[Calum: #1]}}}

\usepackage{amsfonts, adjustbox, amssymb, amscd}
\numberwithin{equation}{section}

\usepackage{bm}
\usepackage{verbatim}
\usepackage{mathrsfs}
\usepackage{graphicx}
\usepackage{tikz-cd}
\usepackage{subcaption}
\usepackage{listings}
\usepackage{subfiles}
\usepackage[toc,page]{appendix}
\usepackage{mathtools}
\usepackage{comment}
\usepackage{enumerate}
\usepackage{enumitem}
\usepackage[all]{xy}

\usepackage{graphicx}
\usepackage{appendix}
\usepackage{hyperref}
\hypersetup{
    colorlinks=true,
    citecolor=red,
    linkcolor=blue,
    filecolor=magenta,      
    urlcolor=red,
}
\newcommand{\bb}{\bm{b}}
\newcommand{\Mm}{{\bf{M}}}
\newcommand{\Nn}{{\bf{N}}}

\newcommand{\Pp}{{\bf{P}}}

\newcommand{\Qq}{\mathbb{Q}}

\newcommand{\Rr}{\mathbb{R}}

\newcommand{\Center}{\operatorname{center}}

\newcommand{\Exc}{\operatorname{Exc}}

\newcommand{\Hh}{\mathbf{H}}

\newcommand{\rk}{\operatorname{rank}}

\newcommand{\ninv}{\operatorname{ninv}}
\newcommand{\inv}{\operatorname{inv}}
\newcommand{\Nklt}{\operatorname{Nklt}}

\newcommand{\Supp}{\operatorname{Supp}}

\newcommand{\Nlc}{\operatorname{Nlc}}

\newcommand{\mult}{\operatorname{mult}}

\newcommand{\cont}{\operatorname{cont}}

\newcommand{\Aa}{{\bf{A}}}

\newcommand{\Bb}{{\bf{B}}}
\newcommand{\Ff}{\mathcal{F}}
\newcommand{\Gg}{\mathcal{G}}

\newcommand{\Ii}{\Gamma}

\newcommand{\Pic}{\mathrm{Pic}}


\newcounter{parentnumber}

\newtheorem{thm}{Theorem}[section]
\newtheorem{conj}[thm]{Conjecture}

\newtheorem{lem}[thm]{Lemma}
\newtheorem{prop}[thm]{Proposition}

\newtheorem{claim}[thm]{Claim}

\theoremstyle{definition}
\newtheorem{defn}[thm]{Definition}
\newtheorem{ques}[thm]{Question}
\theoremstyle{definition}
\newtheorem{rem}[thm]{Remark}

\newtheorem{nota}[thm]{Notation}

\theoremstyle{definition}

\begin{document}

\title{Minimal model program for algebraically integrable adjoint foliated structures}
\author{Paolo Cascini, Jingjun Han, Jihao Liu, Fanjun Meng, Calum Spicer, Roberto Svaldi, and Lingyao Xie}

\date{\today}

\begin{abstract}
For $\Qq$-factorial klt algebraically integrable adjoint foliated structures, we prove the cone theorem, the contraction theorem, and the existence of flips. Therefore, we deduce the existence of the minimal model program for such structures.

We also prove the base-point-freeness theorem for such structures of general type and establish an adjunction formula and the existence of $\Qq$-factorial quasi-dlt modifications for algebraically integrable adjoint foliated structures. 
\end{abstract}

\address{Department of Mathematics, Imperial College London, 180 Queen’s
Gate, London SW7 2AZ, UK}
\email{p.cascini@imperial.ac.uk}

\address{Shanghai Center for Mathematical Sciences, Fudan University, 2005 Songhu Road, Shanghai, 200438, China}
\email{hanjingjun@fudan.edu.cn}

\address{Department of Mathematics, Northwestern University, 2033 Sheridan Road, Evanston, IL 60208, USA}
\email{jliu@northwestern.edu}

\address{Department of Mathematics, Johns Hopkins University, 3400 N. Charles Street, Baltimore, MD 21218, USA}
\email{fmeng3@jhu.edu}

\address{Department of Mathematics, King’s College London, Strand, London WC2R 2LS, UK}
\email{calum.spicer@kcl.ac.uk}

\address{Dipartimento di Matematica ``F. Enriques'', Universit\`a degli Studi di Milano, Via Saldini 50, 20133 Milano (MI), Italy}
\email{roberto.svaldi@unimi.it}

\address{Department of Mathematics, University of California, San Diego, 9500 Gilman Drive \# 0112, La Jolla, CA 92093-0112, USA}
\email{l6xie@ucsd.edu}

\maketitle

\pagestyle{myheadings}\markboth{\hfill P. Cascini, J. Han, J. Liu, F. Meng, C. Spicer, R. Svaldi, and L. Xie \hfill}{\hfill Minimal model program for algebraically integrable adjoint foliated structures\hfill}

\tableofcontents

\section{Introduction}\label{sec:Introduction}
We work over the field of complex numbers $\mathbb{C}$.

The minimal model program for foliations has been extensively developed in recent years, particularly for foliated surfaces \cite{McQ08,Bru15}, foliated threefolds \cite{CS20,Spi20,CS21,SS22}, and algebraically integrable foliations \cite{ACSS21,CS23a,CHLX23,LMX24b}. However, it is also well known that many classical results in birational geometry that hold for algebraic varieties fail when attempting to extend those to foliations, e.g., abundance (cf. \cite[Theorem 3 IV.5.11]{McQ08}), effective birationality (cf. \cite[Theorem 1.3]{Lü24}), and Bertini-type theorems (cf. \cite[Example 3.4]{DLM23}, \cite[3.5.2]{CS23b}).

One natural approach to overcome the failure of these important results is to introduce a more general class of structures which we will call 
\emph{adjoint foliated structures}. 
In simple terms,  rather than  considering the canonical divisor $K_{\Ff}$ of a foliation $\Ff$ on $X$, 
we consider the interpolated canonical divisor
$tK_{\Ff}+(1-t)K_X$ instead, where $0\leq t\leq 1$.

The major goal of this paper is to establish the existence of the minimal model program for algebraically integrable adjoint foliated structures; that is, to establish the cone theorem, the contraction theorem, and the existence of flips for adjoint foliated structures $tK_{\Ff}+(1-t)K_X$ with $\Ff$ algebraically integrable. Our first main theorem is the following:

\begin{thm}\label{thm: simplified main theorem}
Let $\Ff$ be an lc algebraically integrable foliation on a smooth projective variety $X$ and $t\in [0,1]$ a real number. Then we can run a $(tK_{\Ff}+(1-t)K_X)$-MMP.
\end{thm}

Theorem \ref{thm: simplified main theorem} was proved when $t=0$ in \cite{BCHM10} and when $t=1$ in \cite{LMX24b}. For further details on the exact technical statements that are proven in this paper, we refer the reader to the next subsections of the introduction.

The extra condition on the algebraic integrability of the foliation $\mathcal F$ allows us to deal with adjoint foliated structures in dimension higher than two, thus achieving in this context a greater generality than in~\cite{SS23}. 
Indeed, for algebraically integrable foliations the (non-adjoint) foliated MMP has already been established in \cite{ACSS21,CS23a,CHLX23,LMX24b}, whereas for foliations that are not necessarily  algebraically integrable the existence of MMP remains open in dimensions $\geq 4$;
this is one of the main obstacles to fully extending the results of~\cite{SS23} to higher dimension in full generality.


\medskip

Adjoint foliated structures were first considered by the sixth author and Pereira in order to prove effective birationality results for foliated surfaces \cite{PS19}.
These structures were later studied more systematically by the fifth and sixth authors \cite{SS23}: 
for adjoint foliated structures on surfaces, they established the minimal model program when 
$1\geq t>\frac{5}{6}$, \cite[Theorems 1.1 and 1.2]{SS23}, 
moreover, they showed that effective birationality, boundedness in families, lower bound of the canonical volumes, and boundedness of the birational automorphism groups hold when 
$t$ is sufficiently close to $1$, \cite[Theorems 1.3-1.7]{SS23}. These results provide strong evidence that classical results in birational geometry hold for adjoint foliated structures, although many of these results fail for foliations.

Moreover, M\textsuperscript{c}Kernan conjectured that the ``interpolated log canonical threshold" for adjoint foliated structures, that is, the supremum of the real number $t$ so that the adjoint foliated structure associated with $tK_{\Ff}+(1-t)K_X$ has log canonical singularities, should satisfy the ascending chain condition (ACC). See Appendix \ref{sec: acc} for details. M\textsuperscript{c}Kernan also proposed a proof in dimension $2$ (no later than his JAMI conference talk in May 2022 \cite{McK22}). In particular, it is conjectured that if $tK_{\Ff}+(1-t)K_X$ has a log canonical structure for some $0<t\ll 1$, then $\Ff$ is also log canonical. This highlights the importance of considering adjoint foliated structures. Methodologically, results on the $(tK_{\Ff}+(1-t)K_X)$ minimal model program could imply results on the $K_{\Ff}$ minimal model program by either letting $t=1$ or by letting $t$ approach $1$ and applying ACC-type results. In fact, this methodology will be applied in this paper to the minimal model program for algebraically integrable foliations on potentially klt varieties. See Theorem \ref{thm: small improvement LMX24b} for details.

We emphasize that, in this paper, rather than considering $tK_{\Ff}+(1-t)K_X$ only for $t$ sufficiently close to $1$, we take into consideration arbitrary value of $t\in [0,1]$. 
We also expect that similar methods as those of this paper can be applied to prove the existence of MMP for adjoint foliated structures that are not necessarily algebraically integrable in dimensions $\leq 3$. 
This will be addressed in future works.

Our main theorem also holds in the log setting, that is, in the setting where we have boundary divisors. 
Due to the novelty of the notation, before we proceed to state our main theorems in detail, we shall first provide the definition of adjoint foliated structures.

\medskip

\noindent\textbf{Definition of adjoint foliated structures.} 
Rather than considering pairs consisting of a normal variety $X$ and foliation $\Ff$ on $X$, it is more convenient to consider pairs $(X,B)$ and foliated triples $(X,\Ff,B)$. 
Recent developments in birational geometry and foliation theory suggest that it is sometimes even more convenient to consider generalized pairs $(X,B,\Mm)$ \cite{BZ16} and generalized foliated quadruples $(X,\Ff,B,\Mm)$ \cite{LLM23}. When we consider adjoint foliated structures, there is one more object to consider: the parameter $t$. For the reader's convenience, we now provide the definition of adjoint foliated structures.

\begin{defn}[$=$Definition \ref{defn: ads}]\label{defn: ads intro}
An \emph{adjoint foliated structure} $(X,\Ff,B,\Mm,t)/U$ consists of a projective morphism $X\rightarrow U$ from a normal quasi-projective variety to a quasi-projective variety, a foliation $\Ff$ on $X$, an $\Rr$-divisor $B\geq 0$ on $X$,  a nef$/U$ $\bb$-divisor $\Mm$, and a real number $t\in [0,1]$, such that $K_{(X,\Ff,B,\Mm,t)}:=tK_{\Ff}+(1-t)K_X+B+\Mm_X$ is $\Rr$-Cartier. 

If $B=0$, $\Mm=\bm{0}$, $U$ is not important, or $\Ff=T_X$, we may drop $B,\Mm,U,\Ff$ respectively. 
Depending on the objects in the parentheses we may also call an adjoint foliated structure as a quintuple, quadruple, triple, pair, etc. 
\end{defn}

The definition of singularities of an adjoint foliated structure (e.g. klt, lc) is similar to that of usual pairs or foliated triples. See Definition \ref{defn: sing of afs} for details.

\medskip

\noindent\textbf{MMP for adjoint foliated structures.} We establish the cone theorem, the contraction theorem, and the existence of flips for $\Qq$-factorial klt algebraically integrable adjoint foliated structures, which implies the existence of the MMP for such structures. 
We refer the reader to Definition \ref{defn: nklt locus} for the definition of the non-lc locus of an adjoint foliated structure.

\begin{thm}[Cone theorem]\label{thm: cone intro}
      Let $(X,\Ff,B,\Mm,t)/U$ be an algebraically integrable adjoint foliated structure. Let $K:=K_{(X,\Ff,B,\Mm,t)}$ and let $\Nlc:=\Nlc(X,\Ff,B,\Mm,t)$ be the nlc (non-lc) locus of $(X,\Ff,B,\Mm,t)$.
      
      Let $\{R_j\}_{j\in\Lambda}$ be the set of $K$-negative extremal rays in $\overline{NE}(X/U)$ that are rational and are not contained in $\Nlc$. Then:
\begin{enumerate}
    \item We have
    $$\overline{NE}(X/U)=\overline{NE}(X/U)_{K\geq 0}+\overline{NE}(X/U)_{\Nlc}+\sum_{j\in\Lambda}R_j.$$
    In particular, any $K$-negative extremal ray in $\overline{NE}(X/U)$ that is not contained in $\Nlc$ is rational.
    \item Each $R_j$ is spanned by a rational curve $C_j$ such that
    $$0<-K\cdot C_j\leq 2\dim X.$$
    \item For any ample$/U$ $\Rr$-divisor $A$ on $X$,
    $$\Lambda_A:=\{j\in\Lambda\mid R_j\subset\overline{NE}(X/U)_{K+A<0}\}$$
    is a finite set. Moreover, we can write
    $$\overline{NE}(X/U)=\overline{NE}(X/U)_{K+A\geq 0}+\overline{NE}(X/U)_{\Nlc}+\sum_{j\in\Lambda_A}R_j.$$
    In particular, $\{R_j\}_{j\in\Lambda}$ is countable and is a discrete subset of $\overline{NE}(X/U)_{K<0}$. 
    \item Let $F$ be a $K$-negative extremal face that is relatively ample at infinity with respect to $(X,\Ff,B,\Mm,t)$. 
    Then $F$ is a rational extremal face.
\end{enumerate}
\end{thm}

For the definition of a a $K$-negative extremal face that is relatively ample at infinity with respect to an adjoint foliated structure, we refer the reader to Definition \ref{defn: basics of cone theorem}.

\begin{thm}[Contraction theorem]\label{thm: cont intro}
Let $(X,\Ff,B,\Mm,t)/U$ be an lc algebraically integrable adjoint foliated structure associated with a morphism $\pi: X\rightarrow U$, and let $R$ be a $K_{(X,\Ff,B,\Mm,t)}$-negative extremal ray$/U$. Assume that $X$ is klt. Then there exists a contraction$/U$ $\cont_R: X\rightarrow T$ satisfying the following.
    \begin{enumerate}
        \item For any integral curve $C$ such that $\pi(C)$ is a point, $\cont_R(C)$ is a point if and only if $[C]\in R$.
        \item Let $L$ be a line bundle on $X$ such that $L\cdot R=0$. Then there exists a line bundle $L_T$ on $T$ such that $L\cong f^\ast L_T$.
    \end{enumerate}
\end{thm}

\begin{thm}[Existence of flips]\label{thm: eof intro}
 Let $(X,\Ff,B,\Mm,t)/U$ be an lc algebraically integrable adjoint foliated structure such that $X$ is potentially klt (i.e. there exists a klt pair $(X,\Delta)$). Let $f: X\rightarrow T$ be a $K_{(X,\Ff,B,\Mm,t)}$-flipping contraction$/U$. 
 
Then the flip $f^+ \colon X^+\rightarrow T$ of $f$ exists. Moreover, if $X$ is $\Qq$-factorial, then $X^+$ is $\Qq$-factorial and $\rho(X/U)=\rho(X^+/U)$. 
\end{thm}

Theorems \ref{thm: cone intro}, \ref{thm: cont intro}, and \ref{thm: eof intro} imply the existence of the minimal model program for $\Qq$-factorial klt algebraically integrable adjoint foliated structures:

\begin{thm}[Existence of MMP]\label{thm: can run mmp intro}
Let $(X,\Ff,B,\Mm,t)/U$ be a $\Qq$-factorial klt algebraically integrable adjoint foliated structure. Then we may run a $K_{(X,\Ff,B,\Mm,t)}$-MMP$/U$.
\end{thm}

It is clear that Theorem \ref{thm: simplified main theorem} is a special case of Theorem \ref{thm: can run mmp intro}.

\medskip

\noindent\textbf{Base-point-freeness for adjoint foliated structures of general type.} 
Having established the minimal model program settled, a natural question to consider regards the existence of good minimal models for algebraically integrable adjoint foliated structures, of which 
the existence of flips can be viewed as a special case. 
In this paper, we prove the base-point-freeness theorem for algebraically integrable adjoint foliated structures of general type.

\begin{thm}[Base-point-freeness]\label{thm: bpf intro}
Let $(X,\Ff,B,\Mm,t)/U$ be a klt algebraically integrable adjoint foliated structure such that $K:=K_{(X,\Ff,B,\Mm,t)}$ is nef$/U$ and big$/U$. Then:
\begin{enumerate}
\item $K$ is semi-ample$/U$.
\item Let $m$ be a positive integer such that $mK$ is Cartier. Then for any integer $n\gg 0$ divisible by $m$, $\mathcal{O}_X(nK)$ is globally generated$/U$.
\end{enumerate}
\end{thm}

For the general case, i.e. when $K$ is nef/$U$ but not necessarily big, we expect that the techniques will be different from those used in this paper; 
we shall address this problem in future work.

\medskip

\noindent\textbf{An adjunction formula.} An adjunction formula for foliations along non-invariant divisors was established by the first and fifth authors in \cite[Theorems 1.1 and 3.16]{CS23b}. 
However, for arbitrary foliations, adjunction to invariant divisors may not be well-behaved, cf.~\cite[Example 3.20]{CS20}. 
Nevertheless, we are able to prove a precise adjunction formula for invariant divisors in the case of algebraically integrable foliations \cite[Proposition 3.2]{ACSS21}, \cite[Theorem 2.4.2]{CHLX23}. 
In this paper, we establish the following adjunction formula for algebraically integrable foliated structures, which will be crucial to establish the cone theorem, see Theorem~\ref{thm: cone intro}.

\begin{thm}[Adjunction]\label{thm: adj intro}
Let $(X,\Ff,B,\Mm,t)/U$ be an algebraically integrable adjoint foliated structure.
Let $S$ be an irreducible component of $B$ such that $S$ is an lc place of $(X,\Ff,B,\Mm,t)/U$, 
i.e. $\mult_{S}B=1$ if $S$ is not $\Ff$-invariant and $\mult_{S}B=1-t$ if $S$ is $\Ff$-invariant. 
Let $\nu \colon S^{\nu} \to S$ be the normalization of $S$ and let $\Mm^{S^{\nu}}:=\nu^\ast \Mm$  be the restricted $\bb$-divisor. 
Then there exists a canonically defined algebraically integrable adjoint foliated structure $(S^{\nu},\Ff_{S^{\nu}},B_{S^{\nu}},\Mm^{S^{\nu}},t)/U$ satisfying the following:
\begin{enumerate}
    \item $K_{(S^{\nu},\Ff_{S^{\nu}},B_{S^{\nu}},\Mm^{S^{\nu}},t)}=
    K_{(X,\Ff,B,\Mm,t)}|_{S^\nu}$. 
    \item If $(X,\Ff,B,\Mm,t)$ is lc, then $(S^{\nu},\Ff_{S^{\nu}},B_{S^{\nu}},\Mm^{S^{\nu}},t)$ 
    is lc. 
\end{enumerate}
\end{thm}

\smallskip

\noindent\textbf{Existence of quasi-dlt modifications.} Another result that we obtain is the existence of $\Qq$-factorial quasi-dlt (qdlt) modifications for algebraically integrable adjoint foliated structures. These modifications can be viewed as an analogue of dlt modifications for usual pairs and as an analogue of $(\ast )$-modifications \cite{ACSS21,CS23a} or $\Qq$-factorial ACSS modifications \cite{CHLX23,LMX24b} for algebraically integrable foliations.

\begin{thm}[Existence of $\Qq$-factorial qdlt modifications]\label{thm: qdlt model intro}
    Let $(X,\Ff,B,\Mm,t)/U$ be an algebraically integrable adjoint foliated structure such that $t<1$. Let $$B_t:=B\wedge(\Supp B^{\ninv}+(1-t)\Supp B^{\inv}),$$
    where $B^{\ninv}$ and  $B^{\inv}$ are the non-$\Ff$-invariant and $\Ff$-invariant part of $B$ respectively. Then there exists a projective birational morphism $h: X'\rightarrow X$ satisfying the following properties:
    \\
    Let $\Ff':=h^{-1}\Ff$ and $$B':=h^{-1}_\ast B_t+\Exc(h)^{\ninv}+(1-t)\Exc(h)^{\inv},$$
    where $\Exc(h)^{\ninv}$ and  $\Exc(h)^{\inv}$ are the non-$\Ff'$-invariant and $\Ff'$-invariant part of $\Exc(h)$ respectively. Let $B_{X'}:=B'^{\ninv}+\frac{1}{1-t}B'^{\inv}$, where $B'^{\ninv}$ and  $B'^{\inv}$ are the non-$\Ff'$-invariant and $\Ff'$-invariant part of $B'$ respectively. Then the following hold.
    \begin{enumerate}
    \item $E$ is an nklt place of $(X,\Ff,B,\Mm,t)$ for any $h$-exceptional prime divisor $E$.
    \item $(X',\Ff',B',\Mm,t)$ is $\Qq$-factorial lc. In particular, if $(X,\Ff,B,\Mm,t)$ is lc, then
    $$K_{(X',\Ff',B',\Mm,t)}=h^\ast K_{(X,\Ff,B,\Mm,t)}.$$
    \item $(X',B_{X'},\Mm)$ is qdlt. In particular, $X'$ is klt.
    \item Any lc place of $(X',\Ff',B',\Mm,t)$ is an lc place of $(X',B_{X'},\Mm)$.
    \item If $t>0$, then there exists a positive real number $\epsilon_0$ such that 
    $$\left(X',\Ff',B'+\epsilon\{B_{X'}\},(1+\epsilon)\Mm,t+\epsilon\right)$$
    is lc for any $\epsilon\in [0,\epsilon_0]$.
    \end{enumerate}
\end{thm}

\medskip

\noindent\textbf{A characterization of the singularities of the ambient variety.} We prove the following theorem, which indicates that, for any algebraically integrable adjoint foliated structure $(X,\Ff,B,\Mm,t)/U$ with parameter $t<1$ and mild singularities, its ambient variety $X$ also admits mild singularities.

\begin{thm}\label{thm: lc afs ambient has nice singularities intro}
    Let $(X,\Ff,B,\Mm,t)/U$ be an algebraically integrable adjoint foliated structure. Then:
    \begin{enumerate}
        \item  if $(X,\Ff,B,\Mm,t)$ is klt, then $X$ is potentially klt; and,
        \item if $(X,\Ff,B,\Mm,t)$ is lc, $t<1$, and $K_X$ is $\Qq$-Cartier, then $X$ is lc.
    \end{enumerate}
\end{thm}

Note that if $t=1$, then (2) above fails. For example, let $X \to U$ be a locally stable family over a normal variety with a normal generic fiber. This induces an lc algebraically integrable foliation $\Ff$ on $X$. We can choose $U$ such that it is $\Qq$-Gorenstein and not lc. In this case, $K_X$ is $\Qq$-Cartier, and $X$ is not lc.

\medskip

\noindent\textbf{Applications to usual foliated MMP.} It is worth mentioning that most of our main theorems also hold for the case when $t=1$, that is, for algebraically integrable generalized foliated quadruples $(X,\Ff,B,\Mm)/U$ and foliated triples $(X,\Ff,B)/U$. In particular, we obtain the following result, which improves \cite[Theorems 1.2 and 1.3]{LMX24b}.

\begin{thm}\label{thm: small improvement LMX24b}
    Let $(X,\Ff,B)/U$ be an lc algebraically integrable foliated triple such that $X$ is potentially klt. Assume that one of the following conditions holds:
    \begin{enumerate}
        \item $K_{\Ff}+B$ is pseudo-effective$/U$;
        \item $(X,\Delta)$ is lc for some $\Delta$ such that $B^{\ninv}\geq\Delta^{\ninv}$, where $B^{\ninv}$ and $\Delta^{\ninv}$ are the non-$\Ff$-invariant parts of $B$ and $\Delta$ respectively.
    \end{enumerate}
    Then the cone theorem, the contraction theorem, and the existence of flips hold for $K_{\Ff}+B$, and we can run a $(K_{\Ff}+B)$-MMP. 
    
    Moreover, let $A$ and $H$ be two ample$/U$ $\Rr$-divisors on $X$. Then we may run a $(K_{\Ff}+B+A)$-MMP$/U$ with scaling of $H$, and any such MMP terminates with a good minimal model or a Mori fiber space of $(X,\Ff,B+A)/U$.
\end{thm}

We remark that, when $X$ is not potentially klt, we still expect Theorem \ref{thm: small improvement LMX24b} to be true. However, the proof is expected to be very different and much more difficult, and we only know the case when $\Ff$ is induced by a locally stable family by \cite{MZ23}.

\medskip

\noindent\textbf{Sketch of the proof of the main theorems.} 
For simplicity, throughout this subsection we will assume that $B=0$, $\Mm=\bm{0}$, and $U=\{pt\}$.
Before discussing the technical details, we will first explain the key idea of this paper, which is the observation of the following phenomenon: 
given $tK_{\Ff}+(1-t)K_X$, we can first pass to a birational model $W \rightarrow X$ such that the pull-back $\Ff_W$ of $\Ff$ on $W$ has nice singularities. 
Let $E \subset W$ be the reduced divisor which is exceptional over $X$, and let $E^{\ninv}$ and $E^{\inv}$ be the non-$\Ff_W$-invariant and $\Ff_W$-invariant parts of $E$, respectively. 
Assume that $K_{\Ff_W}+E^{\ninv}$ is pseudo-effective. We first run a $(K_{\Ff_W}+E^{\ninv})$-MMP to produce a minimal model $Y$. 
By construction, denoting by $\Ff_Y$ (resp. $E_Y$)
the push-forward of $\Ff_W$ 
(resp. $E$)
to $Y$, 
$(Y,E_Y,\Nn:=\overline{\frac{t}{1-t}(K_{\Ff_Y}+E_Y^{\ninv})})$ is a $\Qq$-factorial generalized pair with generalized log canonical singularities. 
Thus, thanks to results of~\cite{BZ16,HL23}, we can run a $(K_Y+E_Y+\Nn_Y)$-MMP terminating with a model $V$. 
Let us observe that the MMP just run  is also a run of the 
$(tK_{\Ff_Y}+(1-t)K_Y+E_Y^{\ninv}+(1-t)E_Y^{\inv})$-MMP.
In general, $V$ may not be a minimal model of $tK_{\Ff}+(1-t)K_X$, but it is under some extra assumptions on the augmented base locus of $tK_{\Ff}+(1-t)K_X$. This methodology will be applied twice later.

Another methodology is the use of generalized pairs. We have already used it once for the $(K_Y+E_Y+\Nn_Y)$-MMP above. Note that the use of a generalized pair to run this MMP is necessary: $(Y,E_Y)$ is only guaranteed to be $\Qq$-factorial qdlt, and since abundance fails for foliations, we may only expect $\Nn_Y=K_{\Ff_Y}+E_Y^{\ninv}$ to be nef rather than semi-ample. Therefore, in general, we cannot obtain an lc structure $(Y,\Delta_Y)$ such that $\Delta_Y\sim_{\mathbb{R}}E_Y+\Nn_Y$. Moreover, since Bertini-type theorems fail for foliations (see, e.g., \cite[Example 3.4]{DLM23}), they also do not hold in general for adjoint foliated structures (although some weaker Bertini-type theorems hold when $t<1$). Therefore, when considering $+A$ for some ample $\Rr$-divisor $A$, we usually put $A$ in the moduli part (the $\Mm$ part) of the adjoint foliated structure to preserve the singularities.

In the remainder of this section we will discuss the technical details of the proof.

\smallskip

\noindent\textit{Cone theorem and qdlt modifications.} 
Given the (relative) cone theorem for algebraically integrable foliations was established in \cite{ACSS21,CHLX23}, our strategy will mimic the ones used to prove that result.
To do this, there are three key ingredients: bend and break (cf. \cite[Corollary 2.28]{Spi20}), adjunction, and the existence of an analogue of dlt modifications ($(\ast )$-modifications in \cite{ACSS21} and $\Qq$-factorial ACSS modifications in \cite{CHLX23}).

The bend and break technique, \cite[Corollary 2.28]{Spi20}, still works, and in the relative setting, \cite[Theorem 8.1.1]{CHLX23} can be applied with minor changes (cf. Proposition \ref{prop: cone non-big case}). The adjunction formula for algebraically integrable foliations holds by \cite[Proposition 3.2]{ACSS21} and \cite[Theorem 2.4.1]{CHLX23}. Thus, we may decompose the adjoint canonical divisor of the adjoint foliated structure into the $K_{\Ff}$ part and the $K_X$ part, apply adjunction formulae separately, and then combine them to obtain the adjunction formula for adjoint foliated structures (see Theorem \ref{thm: afs adj}). It is worth mentioning that, if $X$ is not $\Qq$-factorial, then the decomposition is not straightforward, but it can be bypassed after we prove the existence of an analogue of dlt modifications (Theorem \ref{thm: qdlt model intro}).

The difficult part is the existence of an analogue of dlt modifications (Theorem \ref{thm: qdlt model intro}), as we cannot follow the proofs in \cite{ACSS21,CHLX23} due to the fact that we cannot reduce lower-dimensional adjoint foliated MMP to an MMP for usual pairs as in \cite[Proof of Lemma 3.11]{ACSS21}. To overcome this difficulty, we apply the methodology mentioned earlier. The idea is to take a $\Qq$-factorial ACSS modification $Y \rightarrow X$ of $\Ff$ so that $K_{\Ff_Y}+E^{\ninv}$ is ample$/X$, where $E$ is the reduced exceptional$/X$ divisor and $E^{\ninv}$ is its non-invariant part. Now, by treating $K_{\Ff_Y}+E^{\ninv}$ as a nef$/X$ $\Qq$-divisor, we can run a $(K_Y+E+\frac{t}{1-t}(K_{\Ff_Y}+E^{\ninv}))$-MMP$/X$, and the output of this MMP will satisfy our requirements. In practice, since we do not have the termination of flips and do not know whether $K_{\Ff}$ is $\Qq$-Cartier, instead of starting with $Y$, we need to start with a foliated log resolution $W \rightarrow X$ of $\Ff$ and run a $(K_{\Ff_W}+E_W^{\ninv})$-MMP$/X$ to obtain the model $Y$ first. We need to use an auxiliary ample divisor to guarantee the termination of the MMPs with scaling so that we can obtain the model $Y$. To obtain the model $X'$, we further apply the general negativity lemma \cite[Lemma 3.3]{Bir12}. See Theorem \ref{thm: eoqdlt model} for details. Theorem \ref{thm: qdlt model intro} is a direct consequence of Theorem \ref{thm: eoqdlt model} and an easy perturbation lemma (Lemma \ref{lem: perturb qdlt afs}).

Theorem \ref{thm: qdlt model intro} shows that $X' \rightarrow X$ is a satisfactory analogue of dlt modifications. We call it a \emph{$\Qq$-factorial qdlt modification} as the ambient variety has a natural qdlt structure. See Definition \ref{defn: qdlt afs} for details.

\smallskip

\noindent\textit{Contraction and base-point-freeness theorems.} The next goal is the contraction theorem (Theorem \ref{thm: cont intro}). The existence of Fano contractions is straightforward (at least for the $\Qq$-factorial case) because the bend and break theorem implies that the variety is covered by rational curves spanning the extremal ray. Thus, any $(tK_{\Ff}+(1-t)K_X)$-negative extremal ray is also $K_X$-negative, and the contraction theorem follows from the usual contraction theorem for klt varieties (see Theorem \ref{thm: contraction not big}). Therefore, we only need to be concerned with the existence of flipping contractions and divisorial contractions.

In either case, we have an ample $\Rr$-divisor $A$ on $X$ such that $H:=tK_{\Ff}+(1-t)K_X+A$ is nef and big and is the supporting function of a $(tK_{\Ff}+(1-t)K_X)$-negative extremal ray. Proving the existence of a flipping contraction is equivalent to proving that $tK_{\Ff}+(1-t)K_X+A$ has a good minimal model. As mentioned earlier, we consider the $\bb$-divisor $\Aa:=\overline{A}$ instead of $A$ due to the failure of Bertini-type theorems. Now we apply our methodology again: take a foliated log resolution $W \rightarrow X$ with reduced exceptional divisor $E$, and run a $(K_{\Ff_W}+E^{\ninv}+\Aa_W)$-MMP. After obtaining a minimal model $Y$, we treat $P_Y:=K_{\Ff_Y}+E_Y^{\ninv}+\Aa_Y$ as a nef $\Rr$-divisor and run an MMP for the generalized pair $(Y,E_Y,\overline{\frac{t}{1-t}P_Y})$ to obtain a good minimal model $V$.

The problem is that $tK_{\Ff_V}+(1-t)K_V+E_V^{\ninv}+(1-t)E_V^{\inv}+A_V$ is usually not a minimal model of $tK_{\Ff}+(1-t)K_X+A$, even in the sense of Birkar-Shokurov. In fact, we cannot exclude the possibility that $W \dashrightarrow V$ contracts some divisors that are not in the augmented base locus of $tK_{\Ff}+(1-t)K_X+A$. For example, if $t$ is very small, the singularities of $\Ff$ can be very bad. However, the entire process $X \leftarrow W \dashrightarrow Y$ only concerns $\Ff$ and completely ignores the structure of $X$. It is very possible that some divisors not in $\Bb_+(H)$ are contracted by the map $W \dashrightarrow Y$. However, if we contract any such divisor, we would never obtain a minimal model by a simple positivity computation (see Lemma \ref{lem: nz keep under pullback}(3)).

Therefore, we cannot generally obtain the existence of good minimal models by simply considering $V$. However, there is a straightforward way to resolve this issue: instead of considering $tK_{\Ff}+(1-t)K_X+A$, we consider $tK_{\Ff}+(1-t)K_X+A+lH$ with $l \gg 0$. The corresponding minimal model program is $H$-trivial by the length of extremal rays, which is part of the cone theorem (Theorem \ref{thm: cone intro}(2)). Therefore, when $H$ is associated with a flipping contraction, $\Bb_+(H)$ is small, so the divisorial part contracted is contained in the exceptional divisors of $W \rightarrow X$. This leads to the proof of the existence of flipping contractions (see Theorem \ref{thm: eo flipping and divisorial contraction}). We also note that when we consider $\Rr$-divisors, we need to prove that $H$ is NQC (i.e., it is a positive sum of nef Cartier divisors). This follows from the finiteness of $(tK_{\Ff}+(1-t)K_X+A)$-negative extremal rays, which is also part of the cone theorem (Theorem \ref{thm: cone intro}(3)).

For the existence of divisorial contractions, the proof is more complicated as we need to have more explicit control of the divisors that are contracted by $W \dashrightarrow V$. However, $H$ is big, and $\Bb_+(H)$ contains finitely many divisors. We want to contract, and only contract, these divisors for $X \dashrightarrow V$. Therefore, we can take $W$ to be a sufficiently high resolution and increase the coefficients of all boundary divisors that are contained in $\Bb_+(H_W)$ to the maximum. Now we can use the fact that the structures on $W$ and $V$ are both lc to deduce that $V$ is a good minimal model of the new structure on $W$. Moreover, $X$ is a minimal model of the new structure on $W$ by some easy Nakayama-Zariski decomposition arguments. This implies that $H$ is actually semi-ample, and we are done. See Theorem \ref{thm: eo flipping and divisorial contraction} for details. The proof of the general type base-point-freeness theorem (Theorem \ref{thm: bpf intro}) follows the same lines as the proof of the contraction theorem.

We also remark that there is an alternative way to consider the MMP we run in the proof of the contraction theorem: suppose that the contraction $X\rightarrow \mathcal{T}$ defined by $H$ already exists; then running the MMP by adding a sufficiently large multiple of $H$ is essentially equivalent to running an MMP$/\mathcal{T}$. In particular, if $\mathcal{O}_H(-H|_H)$ is ample, then the contraction $X \rightarrow \mathcal{T}$ exists in the category of algebraic spaces \cite[Theorem 6.2]{Art70}, and with the MMP for algebraic spaces, we can show that $X \rightarrow \mathcal{T}$ is indeed a contraction in the category of projective varieties. This is a common strategy in previous works on foliations \cite{CS20,CS21,Spi20}. However, since we are not in dimension $3$, it is not easy to show that $\mathcal{O}_H(-H|_H)$ is ample. Moreover, we want to work with normal quasi-projective varieties associated with projective morphisms rather than normal projective varieties only. Thus, although it is an inspiring strategy, we do not argue in this way.

\smallskip

\noindent\textit{Existence of flips.} Next, we prove the existence of flips (Theorem \ref{thm: eof intro}). The key idea of the proof is as follows: in the proof of divisorial contraction and flipping contractions, we obtain a birational map $X \dashrightarrow V$ that is $H$-trivial. Now the crepant transform $H_V$ of $H$ on $V$ has a ``$\Qq$-factorial qdlt generalized pair $+$ ample" structure, which induces the contractions $V \rightarrow T$, hence $H$ also induces a contraction $\pi: X \rightarrow T$. Following this construction, we know that $T$ is potentially klt. In particular, there exists a small $\Qq$-factorialization $T' \rightarrow T$. Since $T'$ is of Fano type over $T$, $tK_{\Ff_{T'}}+(1-t)K_{T'}$ has a good minimal model$/T$, hence $tK_{\Ff_{T'}}+(1-t)K_{T'}$ has an ample model$/T$, denoted by $X^+$. It is easy to check that the induced birational map $X \dashrightarrow X^+$ is the $(tK_{\Ff}+(1-t)K_X)$-flip we want.

We also note that in the non-$\Qq$-factorial minimal model program, a step of an MMP may be of the form $X \rightarrow T \leftarrow X^+$, where $X \rightarrow T$ is a divisorial contraction and $X^+ \rightarrow T$ is a small modification. The arguments above could also imply that the small modification $X^+ \rightarrow T$ exists if $X \rightarrow T$ is a divisorial contraction of a $(tK_{\Ff}+(1-t)K_X)$-negative extremal ray.

\smallskip

\noindent\textit{Existence of the minimal model program and Theorem \ref{thm: lc afs ambient has nice singularities intro}.} With the cone theorem, the contraction theorem, and the existence of flips, and after proving that some auxiliary properties (e.g., $\Qq$-factorial, potentially klt) are preserved under the minimal model program, we can deduce that we can run an MMP for algebraically integrable adjoint foliated structures (Theorem \ref{thm: can run mmp intro}). Theorem \ref{thm: simplified main theorem} and the main part of Theorem \ref{thm: small improvement LMX24b} are corollaries of the aforementioned theorems, while the ``moreover" part of Theorem \ref{thm: small improvement LMX24b} follows from \cite[Theorems 1.10 and 1.11]{LMX24b}.

We are only left to prove Theorem \ref{thm: lc afs ambient has nice singularities intro}. When $X$ is $\Qq$-factorial, Theorem \ref{thm: lc afs ambient has nice singularities intro}(2) can be proved by taking a proper $\Qq$-factorial ACSS modification $h: (X',\Ff',B_{\Ff'},\Mm;G)/Z \rightarrow (X,\Ff,B^{\ninv},\Mm)$ such that $G$ contains all the $h$-exceptional prime divisors and the strict transform of $B^{\inv}$, and using the fact that $(X',B_{\Ff'}+G,\Mm)$ is lc. When $X$ is not $\Qq$-factorial, we need to choose the $\Qq$-factorial ACSS model for another structure $(X,\Ff,B_{\Ff},\Pp)$, and the proof follows from similar arguments. See Proposition \ref{prop: adjoint lc implies X lc} for details.

The proof of Theorem \ref{thm: lc afs ambient has nice singularities intro}(1) is more complicated. The idea is to first use the existence of $\Qq$-factorial qdlt modifications (Theorem \ref{thm: qdlt model intro}) to obtain a model $h: X' \rightarrow X$. It is easy to show that $h$ is small. Now a key observation is that, following the proof of the contraction theorem, one can actually show that the contraction $X' \rightarrow X$ is the contraction of a $(K'+A')$-trivial extremal face, where $K'$ is associated with a klt adjoint foliated structure and $A'$ is ample$/X$. Following the proof of the contraction theorem, this implies that $X$ is potentially klt, which is Theorem \ref{thm: lc afs ambient has nice singularities intro}(1).

\medskip

\noindent\textbf{Structure of the paper.} In Section \ref{sec: prelimaries} we introduce some preliminary results. In Section \ref{sec: qdlt modification} we define qdlt singularities for algebraically integrable adjoint foliated structures, prove the existence of $\Qq$-factorial qdlt modifications (Theorem \ref{thm: qdlt model intro}) and also prove Theorem \ref{thm: lc afs ambient has nice singularities intro}(2) (see Proposition \ref{prop: adjoint lc implies X lc}).  In Section \ref{sec: an adjunction formula} we prove the adjunction formula (Theorem \ref{thm: adj intro}). In Section \ref{sec: cone} we prove the cone theorem (Theorem \ref{thm: cone intro}). In Section \ref{sec: contraction} we prove the contraction theorem (Theorem \ref{thm: cont intro}) and the general type base-point-freeness theorem (Theorem \ref{thm: bpf intro}). In Section \ref{sec: eof} we prove the existence of flips (Theorem \ref{thm: eof intro}).  In Section \ref{sec: run mmp} we prove the existence of the minimal model program (Theorems \ref{thm: can run mmp intro}, \ref{thm: simplified main theorem} and \ref{thm: small improvement LMX24b}) and Theorem \ref{thm: lc afs ambient has nice singularities intro}(1) (see Proposition \ref{prop: klt afs imply potentially klt}), hence also conclude the proof of Theorem \ref{thm: lc afs ambient has nice singularities intro}.

\medskip

\noindent\textbf{Acknowledgements.} The authors would like to thank Caucher Birkar, Guodu Chen, Christopher D. Hacon, Dongchen Jiao, Junpeng Jiao, James M\textsuperscript{c}Kernan, Vyacheslav V. Shokurov, 
Pascale Voegtli and Ziquan Zhuang for useful discussions. The first and the fifth authors are partially funded by EPSRC. The second author is supported by NSFC for Excellent Young Scientists (\#12322102), and the National Key Research and Development Program of China (\#2023YFA1010600, \#2020YFA0713200). The second author is a member of LMNS, Fudan University. 
The fourth author is partially supported by an AMS-Simons Travel Grant. 
The fifth author is supported by the 
``Programma per giovani ricercatori Rita Levi Montalcini''
of the Italian Ministry of University and Research
and by 
PSR 2022 -- Linea 4 of the 
University of Milan. 
He is a member of the GNSAGA group of INDAM.
The last author was partially supported by NSF research grants no. DMS-1801851 and DMS-1952522, as well as a grant from the Simons Foundation (Award Number: 256202).

\section{Preliminaries}\label{sec: prelimaries}

We will adopt the standard notation and definitions on MMP in \cite{KM98,BCHM10} and use them freely. For foliations, foliated triples, generalized foliated quadruples, we adopt the notation and definitions in \cite{LLM23,CHLX23} which generally align with \cite{CS20, ACSS21, CS21} (for foliations and foliated triples) and \cite{BZ16,HL23} (for generalized pairs and $\bb$-divisors), possibly with minor differences. 

\subsection{Special notation}

\begin{defn}
    A \emph{contraction} is a projective morphism of varieties
    $f \colon X \to Y$
    such that 
    $f_\ast \mathcal{O}_X=\mathcal{O}_Y$.
\end{defn}

When $X, Y$ are normal, the condition 
$f_\ast \mathcal{O}_X=\mathcal{O}_Y$
is equivalent to the fibers of $f$ being connected.

\begin{nota}
    Let $h \colon  X\dashrightarrow X'$ be a birational map between normal varieties. We denote by $\Exc(h)$ the reduced divisor supported on the codimension one part of the exceptional locus of $h$. 
\end{nota}

\begin{nota}
    Let $X$ be a normal variety and $D,D'$ two $\Rr$-divisors on $X$. We define 
    $D\wedge D':=\sum_P\min\{\mult_PD,\mult_PD'\}P$ and  $D\vee D':=\sum_P\max\{\mult_PD,\mult_PD'\}P$ where the sum runs through all the prime divisors $P$ on $X$. We denote by $\Supp D$ the reduced divisor supported on $D$. 
\end{nota}

\subsection{Foliations}

\begin{defn}[Foliations, {cf. \cite{ACSS21,CS21}}]\label{defn: foliation}
Let $X$ be a normal variety. A \emph{foliation} on $X$ is a coherent sheaf $\Ff\subset T_X$ such that
\begin{enumerate}
    \item $\Ff$ is saturated in $T_X$, i.e. $T_X/\Ff$ is torsion free, and
    \item $\Ff$ is closed under the Lie bracket.
\end{enumerate}

The \emph{rank} of a foliation $\Ff$ on a variety $X$ is the rank of $\Ff$ as a sheaf and is denoted by $\rk\Ff$. 
The \emph{co-rank} of $\Ff$ is $\dim X-\rk\Ff$. The \emph{canonical divisor} of $\Ff$ is a divisor $K_\Ff$ such that $\mathcal{O}_X(-K_{\mathcal{F}})\cong\mathrm{det}(\Ff)$. If $\Ff=0$, then we say that $\Ff$ is a \emph{foliation by points}.

Given any dominant map 
$h: Y\dashrightarrow X$ and a foliation $\mathcal F$ on $X$, we denote by $h^{-1}\Ff$ the \emph{pullback} of $\Ff$ on $Y$ as constructed in \cite[3.2]{Dru21} and say that $h^{-1}\Ff$ is \emph{induced by} $\Ff$. Given any birational map $g: X\dashrightarrow X'$, we denote by $g_\ast \Ff:=(g^{-1})^{-1}\Ff$ the \emph{pushforward} of $\Ff$ on $X'$ and also say that $g_\ast \Ff$ is \emph{induced by} $\Ff$. We say that $\Ff$ is an \emph{algebraically integrable foliation} if there exists a dominant map $f: X\dashrightarrow Z$ such that $\Ff=f^{-1}\Ff_Z$, where $\Ff_Z$ is the foliation by points on $Z$, and we say that $\Ff$ is \emph{induced by} $f$.

A subvariety $S\subset X$ is called \emph{$\Ff$-invariant} if for any open subset $U\subset X$ and any section $\partial\in H^0(U,\Ff)$, we have $\partial(\mathcal{I}_{S\cap U})\subset \mathcal{I}_{S\cap U}$,  where $\mathcal{I}_{S\cap U}$ denotes the ideal sheaf of $S\cap U$ in $U$.  
For any prime divisor $P$ on $X$, we define $\epsilon_{\Ff}(P):=1$ if $P$ is not $\Ff$-invariant and $\epsilon_{\Ff}(P):=0$ if $P$ is $\Ff$-invariant. For any prime divisor $E$ over $X$, we define $\epsilon_{\Ff}(E):=\epsilon_{\Ff_Y}(E)$ where $h: Y\dashrightarrow X$ is a birational map such that $E$ is on $Y$ and $\Ff_Y:=h^{-1}\Ff$.

Suppose that the foliation structure $\Ff$ on $X$ is clear in the context. Then, given an $\mathbb R$-divisor $D = \sum_{i = 1}^k a_iD_i$ where each $D_i$ is a prime Weil divisor,
we denote by $D^{{\rm ninv}} \coloneqq \sum \epsilon_{\mathcal F}(D_i)a_iD_i$ the \emph{non-$\Ff$-invariant part} of $D$ and $D^{{\rm inv}} \coloneqq D-D^{{\rm ninv}}$ the \emph{$\Ff$-invariant part} of $D$.

\end{defn}

\begin{defn}[Tangent, {cf. \cite[Section 3.4]{ACSS21}}]\label{defn: tangent to foliation}
 Let $X$ be a normal variety, $\Ff$ be  a foliation on $X$, and $V\subset X$ be a subvariety. 
 Suppose that $\Ff$ is a foliation induced by a dominant rational map $X\dashrightarrow Z$. 
 We say that $V$ is \emph{tangent} to $\Ff$ if there exists a birational morphism $\mu: X'\rightarrow X$, an equidimensional contraction $f': X'\rightarrow Z$, and a subvariety $V'\subset X'$, such that
    \begin{enumerate}
    \item $\mu^{-1}\Ff$ is induced by $f'$, and
        \item $V'$ is contained in a fiber of $f'$ and $\mu(V')=V$.
    \end{enumerate}
Note that we may also define tangency when the foliation is not algebraically integrable but we do not need it in this paper.
\end{defn}

\begin{defn}[$=$Definition \ref{defn: ads intro}]\label{defn: ads}
An \emph{adjoint foliated structure}, $(X,\Ff,B,\Mm,t)/U$ is the datum of a normal quasi-projective variety $X$ and a projective morphism $X\rightarrow U$, a foliation $\Ff$ on $X$, an $\Rr$-divisor $B\geq 0$ on $X$, a $\bb$-divisor $\Mm$ nef$/U$, and a real number $t\in [0,1]$ such that $K_{(X,\Ff,B,\Mm,t)}:=tK_{\Ff}+(1-t)K_X+B+\Mm_X$ is $\Rr$-Cartier. 

When $t=0$ or $\Ff=T_X$, we call $(X,B,\Mm)/U$ a \emph{generalized pair}, and in addition, if $\Mm=\bm{0}$, then we call $(X,B)/U$ a log pair. 
When $t=1$, we call $(X,\Ff,B,\Mm)/U$ a \emph{generalized foliated quadruple}, and in addition, if $\Mm=\bm{0}$, then we call $(X,\Ff,B)/U$ a \emph{foliated triple}.

If $B=0$, or if $\Mm=\bm{0}$, or if $U$ is not important, then we may drop $B,\Mm,U$ respectively. If we allow $B$ to have negative coefficients, then we shall add the prefix ``sub-". If $B$ is a $\Qq$-divisor and $\Mm$ is a $\Qq$-$\bb$-divisor, then we shall add the prefix ``$\Qq$-".
\end{defn}

\begin{defn}\label{defn: sing of afs}
Let $(X,\Ff,B,\Mm,t)/U$ be an adjoint foliated structure. Let $h: X'\rightarrow X$ be a birational morphism, $\Ff':=h^{-1}\Ff$, and $B'$ the unique $\Rr$-divisor such that
$$K_{(X',\Ff',B',\Mm,t)}=h^\ast K_{(X,\Ff,B,\Mm,t)}.$$
 For any prime divisor $E$ on $X'$, we denote by
$$a(E,\Ff,B,\Mm,t):=-\mult_{E}B'$$
the \emph{discrepancy} of $E$ with respect to $(X,\Ff,B,\Mm,t)$. We say that $(X,\Ff,B,\Mm,t)$ is \emph{lc} (resp. \emph{klt}) if $a(E,\Ff,B,\Mm,t)\geq -t\epsilon_{\Ff}(E)-(1-t)$ (resp. $>-t\epsilon_{\Ff}(E)-(1-t)$) for any prime divisor $E$ over $X$. If we allow $B$ to have negative coefficients, then we shall add the prefix ``sub-" for the types of singularities above.
\end{defn}

The following lemma provides an equivalent definition for an adjoint foliated structure to be lc.

\begin{lem}
 Let $(X,\Ff,B,\Mm,t)/U$ be a sub-adjoint foliated structure
 such that 
 $$a(E,\Ff,B,\Mm,t)\geq -1$$
  for any prime divisor $E$ over $X$. Then $(X,\Ff,B,\Mm,t)$ is sub-lc.
\end{lem}
\begin{proof}
    Suppose not. Then there exists a prime divisor $E$ over $X$ such that 
     $$a(E,\Ff,B,\Mm,t)<-1+t(1-\epsilon_{\Ff}(E)).$$
Since $a(E,\Ff,B,\Mm,t)\geq -1$, $E$ is $\Ff$-invariant.
Possibly replacing $X$ with a resolution we may assume that $X$ is smooth and $E$ is on $X$. Now we may repeatedly blow-up points on $E$ as in \cite[Remark 2.3]{CS21} and obtain a divisor with discrepancy $<-1$ with respect to $(X,\Ff,B,\Mm,t)$. This contradicts our assumption.
\end{proof}

\begin{defn}
    A generalized (sub-)pair $(X,B,\Mm)/U$ is called \emph{(sub-)qdlt} if for any lc center $V$ of $(X,B,\Mm)$ with generic point $\eta_V$, $\Mm$ descends near $\eta_V$ and there exist components $S_1,\dots,S_{\dim X-\dim V}$ of $\lfloor B\rfloor$ that contain $\eta_V$ and are $\Qq$-Cartier near $\eta_V$, i.e. $(X,B)$ is log toroidal near $\eta_V$.
\end{defn}

\begin{defn}[Nklt locus]\label{defn: nklt locus}
    Let $(X,\Ff,B,\Mm,t)/U$ be a sub-adjoint foliated structure such that $t<1$. 
    \begin{enumerate}
        \item  An \emph{nklt place} (resp. \emph{lc place}, \emph{nlc place}) of $(X,\Ff,B,\Mm,t)$ is a prime divisor $E$ over $X$ such that 
        $$a(E,X,\Ff,B,\Mm,t)\leq\text{(resp. }=, <\text{)} -t\epsilon_{\Ff}(E)-(1-t).$$
        \item   An \emph{nklt center} (resp. \emph{nlc center} of $(X,\Ff,B,\Mm,t)$ is the center of an nklt place (resp. nlc place) of $(X,\Ff,B,\Mm,t)$ on $X$.
        \item The \emph{nklt locus (resp. \emph{nlc locus}}) of $(X,\Ff,B,\Mm,t)$ is the union of all its nklt centers (resp. nlc centers) associated with a reduced scheme structure.
        \item  An \emph{lc center} $(X,\Ff,B,\Mm,t)$ is an nklt center of $(X,\Ff,B,\Mm,t)$ that is not an nlc center of $(X,\Ff,B,\Mm,t)$ 
    \end{enumerate}
    When $t=1$ the corresponding definitions are slightly different. We refer the reader to \cite[Definition 3.4.5]{CHLX23} for the corresponding definitions.
\end{defn}

\begin{defn}[Potentially klt]\label{defn: potentially klt}
Let $X$ be a normal quasi-projective variety. We say that $X$ is \emph{potentially klt} if $(X,\Delta)$ is klt for some $\Rr$-divisor $\Delta\geq 0$. 
\end{defn}

\subsection{Special algebraically integrable adjoint foliated structures}

\begin{defn}[{cf. \cite[3.2 Log canonical foliated pairs]{ACSS21}, \cite[Definition 6.2.1]{CHLX23}}]\label{defn: foliated log smooth}
Let $(X,\Ff,B,\Mm,t)/U$ be an algebraically integrable adjoint foliated structure. We say that $(X,\Ff,B,\Mm,t)$ is \emph{foliated log smooth} if there exists a contraction $f: X\rightarrow Z$ satisfying the following.
\begin{enumerate}
  \item $X$ has at most quotient toric singularities.
  \item $\Ff$ is induced by $f$.
  \item $(X,\Sigma_X)$ is toroidal for some reduced divisor $\Sigma_X$ such that $\Supp B\subset\Sigma_X$.  In particular, $(X,\Supp B)$ is toroidal, and $X$ is $\Qq$-factorial klt.
  \item There exists a log smooth pair $(Z,\Sigma_Z)$ such that $$f: (X,\Sigma_X)\rightarrow (Z,\Sigma_Z)$$ is an equidimensional toroidal contraction.
  \item $\Mm$ descends to $X$.
\end{enumerate}
We say that $f: (X,\Sigma_X)\rightarrow (Z,\Sigma_Z)$ is \emph{associated with} $(X,\Ff,B,\Mm,t)$. It is important to remark that $f$ may not be a contraction$/U$. In particular, $\Mm$ may not be nef$/Z$.

Note that the definition of foliated log smooth has nothing to do with $t$. In other words, as long as $(X,\Ff,B,\Mm)$ is foliated log smooth, $(X,\Ff,B,\Mm,t)$ is foliated log smooth.
\end{defn}

\begin{defn}[Foliated log resolutions]\label{defn: log resolution}
Let $X$ be a normal quasi-projective variety, $B$ an $\Rr$-divisor on $X$, $\Ff$ an algebraically integrable foliation on $X$, $t\in [0,1]$ a real number, $X\rightarrow U$ a projective morphism, and $\Mm$ a nef$/U$ $\bb$-divisor on $X$. A \emph{foliated log resolution} of $(X,\Ff,B,\Mm,t)$ is a birational morphism $h: X'\rightarrow X$ such that 
$$(X',\Ff':=h^{-1}\Ff,B':=h^{-1}_\ast B+\Exc(h),\Mm,t)$$ 
is foliated log smooth. By \cite[Lemma 6.2.4]{CHLX23}, foliated log resolution for $(X,\Ff,B,\Mm,t)$ always exists.
\end{defn}

\begin{lem}\label{lem: foliated log smooth imply lc}
Let $(X,\Ff,B,\Mm,t)/U$ be a sub-adjoint foliated structure such that $(X,\Ff,B,\Mm)$ is foliated log smooth. Then $(X,\Ff,\Supp B^{\ninv}+(1-t)\Supp B^{\inv},\Mm,t)$ is lc.
\end{lem}
\begin{proof}
There exists an equidimensional toroidal contraction $f: (X,\Sigma_X)\rightarrow (Z,\Sigma_Z)$ such that $\Supp B\subset\Sigma_X$. Moreover, $\Mm$ descends to $X$. By definition, $(X,\Sigma_X)$ is lc, so $(X,\Sigma_X,\Mm)$ is lc, and so $(X,\Supp B^{\ninv}+\Supp B^{\inv},\Mm)$ is lc. By \cite[Lemma 6.2.2]{CHLX23}, $(X,\Ff,\Supp B^{\ninv},\Mm)$ is lc. Thus $(X,\Ff,\Supp B^{\ninv}+(1-t)\Supp B^{\inv},\Mm,t)$ is lc.
\end{proof}

\begin{defn}[Property $(\ast )$ foliations, {\cite[Definition 3.8]{ACSS21}, \cite[Definition 7.2.2]{CHLX23}}]\label{defn: foliation property *}
Let $(X,\Ff,B,\Mm)/U$ be a generalized foliated quadruple, $G\geq 0$ be a reduced divisor on $X$, and $f: X\rightarrow Z$ a contraction. We say that $(X,\Ff,B,\Mm;G)/Z$ satisfies \emph{Property $(\ast )$} if the following conditions hold.
\begin{enumerate}
\item $\Ff$ is induced by $f$ and $G$ is an $\Ff$-invariant divisor.
\item $f(G)$ is of pure codimension $1$, $(Z,f(G))$ is log smooth, and $G=f^{-1}(f(G))$.
\item For any closed point $z\in Z$ and any reduced divisor  $\Sigma\ge f(G)$ on $Z$ such that  $(Z,\Sigma)$ is log smooth near $z$, $(X,B+G+f^\ast (\Sigma-f(G)),\Mm)$ is lc over a neighborhood of $z$.
\end{enumerate}
We say that $f$, $Z$, and $G$ are \emph{associated} with $(X,\Ff,B,\Mm)$. 

Note that Property $(\ast )$ implies that the foliation is algebraically integrable. 
\end{defn}

\begin{prop}[{\cite[Proposition 3.6]{ACSS21}, \cite[Proposition 7.3.6]{CHLX23}}]\label{prop: weak cbf gfq}
Let $(X,\Ff,B,\Mm)$ be a generalized foliated quadruple. Let  $G\geq 0$ be a reduced divisor on $X$ and $f: X\rightarrow Z$ an equidimensional contraction, such that $(X,\Ff,B,\Mm;G)/Z$ satisfies Property $(\ast )$ and $B$ is horizontal$/Z$. Then
$$K_{\Ff}+B\sim_{Z}K_X+B+G.$$
\end{prop}

\begin{defn}[ACSS, {cf. \cite[Definitions 5.4.2, 7.2.2, 7.2.3]{CHLX23}}]\label{defn: ACSS f-triple}
Let $(X,\Ff,B,\Mm)$ be an lc foliated triple, $G\geq 0$ a reduced divisor on $X$, and $f: X\rightarrow Z$ a contraction. We say that $(X,\Ff,B,\Mm;G)/Z$ is \emph{ACSS} if the following conditions hold:
\begin{enumerate}    
\item $(X,\Ff,B,\Mm;G)/Z$ satisfies Property $(\ast )$.
\item $f$ is equidimensional.
\item There exists an $\Rr$-Cartier $\Rr$-divisor $D\geq 0$ on $X$ and a nef$/X$ $\bb$-divisor $\Nn$ on $X$, such that  $\Supp\{B\}\subset\Supp D$, $\Nn-\alpha\Mm$ is nef$/X$ for some $\alpha>1$, and for any reduced divisor $\Sigma\geq f(G)$ such that $(Z,\Sigma)$ is log smooth, $$(X,B+D+G+f^\ast (\Sigma-f(G)),\Nn)$$ 
      is qdlt (cf. \cite[Definition 7.1.1]{CHLX23}, \cite[Definition 35]{dFKX17}).
\item For any lc center of $(X,\Ff,B,\Mm)$ with generic point $\eta$, over a neighborhood of $\eta$,
    \begin{enumerate}
      \item $\eta$ is the generic point of an lc center of $(X,\Ff,\lfloor B\rfloor)$, and
       \item $f: (X,B+G)\rightarrow (Z,f(G))$ is a toroidal morphism.
    \end{enumerate}
\end{enumerate}
If $(X,\Ff,B,\Mm;G)/Z$ is ACSS, then we say that $(X,\Ff,B,\Mm)/Z$ and $(X,\Ff,B,\Mm)$ are ACSS.
\end{defn}

\subsection{ACSS modifications}

\begin{defn}[ACSS modifications]\label{defn: simple model}
    Let $(X,\Ff,B,\Mm)$ and $(X',\Ff',B',\Mm)$ be two algebraically integrable generalized foliated quadruples and $h: X'\rightarrow X$ a birational morphism. Let $f: X'\rightarrow Z$ be a contraction and $G$ a reduced divisor on $X'$. We say that $h: (X',\Ff',B',\Mm;G)/Z\rightarrow (X,\Ff,B,\Mm)$ is an \emph{ACSS modification} if the following conditions hold.
    \begin{enumerate}
        \item $\Ff':=h^{-1}\Ff$ and $B':=h^{-1}_\ast (B^{\ninv}\wedge\Supp B^{\ninv})+\Supp\Exc(h)^{\ninv}$.
        \item $(X',\Ff',B',\Mm;G)/Z$ is ACSS.
        \item  $a(E,\Ff,B,\Mm)\leq-\epsilon_{\Ff}(E)$ for any $h$-exceptional prime divisor $E$.
        \item If $K_{\Ff}+B^{\ninv}+\Mm_X$ is $\Rr$-Cartier, then $a(E,\Ff,B^{\ninv},\Mm)\leq-\epsilon_{\Ff}(E)$ for any $h$-exceptional prime divisor $E$.
    \end{enumerate}
 We say that $(X',\Ff',B',\Mm;G)/Z$, $(X',\Ff',B',\Mm)/Z$, and $(X',\Ff',B',\Mm)$ are ACSS models of $(X,\Ff,B)$.     We say that $h$ is an ACSS modification of $(X,\Ff,B)$. Moreover, we say that $h$ is
\begin{itemize}
\item \emph{$\Qq$-factorial} if $X'$ is $\Qq$-factorial, and
\item \emph{proper} if $\Supp\Exc(h)^{\inv}\cup\Supp h^{-1}_\ast B^{\inv}$ is contained in $\Supp G$.
\end{itemize}
Note that this definition here has slight differences with \cite{CHLX23,LMX24b} because the concept of ``proper" in \cite{CHLX23,LMX24b} only requires $\Supp\Exc(h)^{\inv}$ to be contained in $\Supp G$ but do not have any requirement on $\Supp h^{-1}_\ast B^{\inv}$.
\end{defn}

\begin{thm}[{cf. \cite[Theorem 8.2.2]{CHLX23}, \cite[Theorem 3.10]{ACSS21}}]\label{thm: eo acss model}
    Let $(X,\Ff,B,\Mm)$ be an algebraically integrable generalized foliated quadruple. Then there exists an ACSS modification $h: (X',\Ff',B',\Mm;G)/Z\rightarrow (X,\Ff,B,\Mm)$ that is $\Qq$-factorial and proper.
\end{thm}
\begin{proof}
    Let $g: W\rightarrow X$ be a foliated log resolution of $(X,\Ff,\Supp B)$, $\Ff_W:=g^{-1}\Ff$, and $B_W:=g^{-1}_\ast (B^{\ninv}\wedge\Supp B^{\ninv})+\Supp\Exc(h)^{\ninv}$. Then there exists an equidimensional toroidal contraction $f: (W,\Sigma_W)\rightarrow (Z,\Sigma_Z)$ such that $\Supp g^{-1}_\ast B\cup\Supp\Exc(g)\subset\Supp\Sigma_W$. Let $G_W$ be the vertical$/Z$ part of $\Sigma_W$. 

 By \cite[Lemma 7.3.3]{CHLX23}, $(W,\Ff_W,B_W,\Mm;G_W)/Z$ is $\Qq$-factorial ACSS. Moreover, there are unique $\Rr$-divisors $E^+\geq 0,E^{-}\geq 0$ on $W$, such that $E^+\wedge E^-=0$ and
 $$K_{\Ff_W}+B_W+\Mm_W=h^\ast (K_{\Ff}+B+\Mm_X)+E^+-E^-.$$
 Moreover, if $K_{\Ff}+B^{\ninv}+\Mm_X$ is $\Rr$-Cartier, then there are unique $\Rr$-divisors $F^+\geq 0,F^{-}\geq 0$ on $W$, such that $F^+\wedge F^-=0$ and
 $$K_{\Ff_W}+B_W+\Mm_W=h^\ast (K_{\Ff}+B^{\ninv}+\Mm_X)+F^+-F^-.$$
By our construction, $E^+$ is exceptional$/W$, and $a(D,\Ff,B,\Mm)\leq -\epsilon_{\Ff}(D)$ for any component $D$ of $E^-$.  Moreover, if $K_{\Ff}+B^{\ninv}+\Mm_X$ is $\Rr$-Cartier, then  $F^+$ is exceptional$/W$, and $a(D,\Ff,B^{\ninv},\Mm)\leq -\epsilon_{\Ff}(D)$ for any component $D$ of $E^-$. By \cite[Theorem 9.4.1]{CHLX23}, we may run a $(K_{\Ff_W}+B_W+\Mm_W)$-MMP$/X$ with scaling of an ample divisor, and after finitely many steps, we reach a model $X'$ such that $E^+$ is contracted by the induced birational map $\phi: W\dashrightarrow X'$, and if $K_{\Ff}+B^{\ninv}+\Mm_X$ is $\Rr$-Cartier, then $F^+$ is contracted by $\phi$.

Let $\Ff':=\phi_\ast \Ff_W,B':=\phi_\ast B$, $G:=\phi_\ast G_W$, and $h: X'\rightarrow X$ the induced birational morphism.  By \cite[Lemma 9.1.4]{CHLX23}, $(X',\Ff',B',\Mm;G)/Z$ is $\Qq$-factorial ACSS. Since $\Supp G_W$ contains any $\Ff_W$-invariant component of $\Exc(g)$ and any component of $\Supp g^{-1}_\ast B^{\inv}$, $\Supp G$ contains $\Supp\Exc(h)^{\inv}\cup\Supp h^{-1}_\ast B^{\inv}.$
For any prime divisor $E$ that is extracted by $h$, $\Center_WE$ is a component of $E^-$, so $a(E,\Ff,B,\Mm)\leq -\epsilon_{\Ff}(E)$, and if $K_{\Ff}+B^{\ninv}+\Mm_X$ is $\Rr$-Cartier, then  $a(E,\Ff,B^{\ninv},\Mm)\leq -\epsilon_{\Ff}(E)$. Therefore,  $h: (X',\Ff',B',\Mm;G)/Z\rightarrow (X,\Ff,B,\Mm)$ is a $\Qq$-factorial proper ACSS modification.
\end{proof}

\subsection{Supporting functions}

\begin{defn}[{\cite[Definition 5.3]{Amb03}, \cite[Definition 6.7.2]{Fuj11}}]\label{defn: basics of cone theorem}
Let $(X,\Ff,B,\Mm,t)/U$ be an adjoint foliated structure and let $F$ be an extremal face of $\overline{NE}(X/U)$.
\begin{enumerate}
    \item A \emph{supporting function} of $F$ is a  $\pi$-nef $\Rr$-divisor $H$ such that $F=\overline{NE}(X/U)\cap H^{\bot}$. If $H$ is a $\Qq$-divisor, we say that $H$ is a \emph{rational supporting function}. Since $F$ is an extremal face of $\overline{NE}(X/U)$, $F$ always has a supporting function.
    \item We say that $F$ is \emph{rational} if $F$ has a rational supporting function.
    \item For any $\Rr$-Cartier $\Rr$-divisor $D$ on $X$, we say that $F$ is $D$-\emph{negative} if $$F\cap\overline{NE}(X/U)_{D\geq 0}=\{0\}.$$
    \item We say that $F$ is \emph{relatively ample at infinity with respect to} $(X,\Ff,B,\Mm,t)$  if $$F\cap\overline{NE}(X/U)_{\Nlc(X,\Ff,B,\Mm,t)}=\{0\}.$$
\end{enumerate}
\end{defn}

\subsection{Birational maps in MMP}

\begin{defn}
    Let $X\rightarrow U$ be a projective morphism from a normal quasi-projective variety to a variety.  Let $D$ be an $\Rr$-Cartier $\Rr$-divisor on $X$ and $\phi: X\dashrightarrow X'$ a birational map$/U$. Then we say that $X'$ is a \emph{birational model} of $X$. We say that $\phi$ is $D$-non-positive (resp. $D$-negative, $D$-trivial, $D$-non-negative, $D$-positive) if the following conditions hold:
    \begin{enumerate}
    \item $\phi$ does not extract any divisor.
    \item $D':=\phi_\ast D$ is $\Rr$-Cartier.
    \item There exists a resolution of indeterminacy $p: W\rightarrow X$ and $q: W\rightarrow X'$, such that
    $$p^\ast D=q^\ast D'+F$$
    where $F\geq 0$ (resp. $F\geq 0$ and $\Supp p_\ast F=\Exc(\phi)$, $F=0$, $0\geq F$, $0\geq F$ and $\Supp p_\ast F=\Exc(\phi)$).
    \end{enumerate}
\end{defn}

\begin{defn} Let $X\rightarrow U$ be a projective morphism from a normal quasi-projective variety to a variety. Let $D$ be an $\Rr$-Cartier $\Rr$-divisor on $X$, $\phi: X\dashrightarrow X'$ a $D$-negative map$/U$ and $D':=\phi_\ast D$. We say that $X'$ is a \emph{minimal model}$/U$ (resp. \emph{good minimal model}$/U$) of $D$ if $D'$ is nef$/U$ (resp. semi-ample$/U$).
\end{defn}

\begin{defn}[Minimal models]\label{defn: minimal model}
    Let $(X,\Ff,B,\Mm,t)/U$ and $(X',\Ff',B',\Mm,t)/U$  be two adjoint foliated structures associated a birational map$/U$ $\phi: (X,\Ff,B,\Mm,t)\dashrightarrow (X',\Ff',B',\Mm,t)$. Let $K:=K_{(X,\Ff,B,\Mm,t)}$ and $K':=K_{(X',\Ff',B',\Mm,t)}$. We say that $(X',\Ff',B',\Mm,t)/U$ is a \emph{weak lc model} (resp. \emph{minimal model}) of $(X,\Ff,B,\Mm,t)/U$ if $\phi$ is $K$-non-positive (resp. $K$-negative) and $K'$ is nef$/U$. We say that  $(X',\Ff',B',\Mm,t)/U$ is a \emph{semi-ample model} (resp. \emph{good minimal model}) of $(X,\Ff,B,\Mm,t)/U$ if $(X',\Ff',B',\Mm,t)/U$ is a weak lc model (resp. minimal model) of $(X,\Ff,B,\Mm,t)/U$ and $K'$ is semi-ample$/U$.
\end{defn}

\subsection{Relative Nakayama-Zariski decompositions}

\begin{defn}
    Let $\pi: X\rightarrow U$ be a projective morphism from a normal quasi-projective variety to a quasi-projective variety, $D$ a pseudo-effective$/U$ $\Rr$-Cartier $\Rr$-divisor on $X$, and $P$ a prime divisor on $X$. We define $\sigma_{P}(X/U,D)$ as in \cite[Definition 3.1]{LX23} by considering $\sigma_{P}(X/U,D)$ as a number in  $[0,+\infty)\cup\{+\infty\}$. We define $N_{\sigma}(X/U,D)=\sum_Q\sigma_Q(X/U,D)Q$
    where the sum runs through all prime divisors on $X$ and consider it as a formal sum of divisors with coefficients in $[0,+\infty)\cup\{+\infty\}$. We say that $D$ is \emph{movable$/U$} if $N_{\sigma}(X/U,D)=0$.
\end{defn}

\begin{lem}[{\cite[Lemma 3.5]{LX23}}]\label{lem: finiteness of sigmap}
Let $\pi: X\rightarrow U$ be a projective morphism from a normal quasi-projective variety to a quasi-projective variety and let $D$ be a pseudo-effective$/U$ $\Rr$-Cartier $\Rr$-divisor on $X$. Then $N_{\sigma}(X/U,D)$ has finitely many irreducible components. 
\end{lem}

\begin{lem}[{\cite[Lemma 3.4(2)(3), Lemma 3.7(4)]{LX23}}]\label{lem: nz keep under pullback}
Let $\pi: X\rightarrow U$ be a projective morphism from a normal quasi-projective variety to a quasi-projective variety and $D$ a pseudo-effective$/U$ $\Rr$-Cartier $\Rr$-divisor on $X$. Let $f: Y\rightarrow X$ be a projective birational morphism. Then:
\begin{enumerate}
    \item For any exceptional$/X$ $\Rr$-Cartier $\Rr$-divisor $E\ge0$ and any prime divisor $P$ on $Y$, we have $$\sigma_P(Y/U,f^\ast D+E)=\sigma_P(Y/U,f^\ast D)+\mult_PE.$$
    \item For any exceptional$/X$ $\Rr$-Cartier $\Rr$-divisor $E\ge0$ on $Y$, we have 
    $$N_{\sigma}(X/U,D)=f_\ast N_{\sigma}(Y/U,f^\ast D+E).$$
    \item $\Supp N_{\sigma}(X/U,D)$ coincides with the divisorial part of $\Bb_-(X/U,D)$.
\end{enumerate}
\end{lem}

\begin{lem}[{\cite[Lemma 2.25]{LMX24b}}]\label{lem: nz for lc divisor}
Let $X\rightarrow U$ be a projective morphism from a normal quasi-projective variety to a variety and $\phi: X\dashrightarrow X'$ a birational map$/U$. Let $D$ be an $\Rr$-Cartier $\Rr$-divisor on $X$ such that $\phi$ is $D$-negative and $D':=\phi_\ast D$. Then:
\begin{enumerate}
    \item The divisors contracted by $\phi$ are contained in $\Supp N_{\sigma}(X/U,D)$.
    \item If $D'$ is movable$/U$, then $\Supp N_{\sigma}(X/U,D)$ is the set of all $\phi$-exceptional divisors.
    \end{enumerate}
 \end{lem}

\begin{lem}\label{lem: sbl of movable big divisor}
Let $X\rightarrow U$ be a projective morphism from a $\Qq$-factorial normal quasi-projective variety to a quasi-projective variety. Let $D$ be a big$/U$ and movable$/U$ $\Rr$-divisor on $X$ and $E\geq 0$ and $\Rr$-divisor such that $\Supp E\subset\Bb_+(D/U)$. Then
$$N_{\sigma}(X/U,D+E)=E.$$
\end{lem}
\begin{proof}
This is essentially \cite[Lemma 3.5]{Nak04}. However, \cite[Lemma 3.5]{Nak04} is stated for smooth projective varieties with $U=\{pt\}$ and is for the $\nu$-decomposition rather than the Nakayama-Zariski decomposition ($\sigma$-decomposition). Therefore, we provide a full proof here.

Since $D$ is movable$/U$, $E\geq N_{\sigma}(X/U,D+E)\geq 0$. 

Suppose that $N_{\sigma}(X/U,D+E)\not=E$. Possibly replacing $E$ with $E-N_{\sigma}(X/U,D+E)$, we may assume that $N_{\sigma}(X/U,D+E)=0$ and $E\not=0$. Since $E$ has finitely many components, possibly replacing $E$, we may assume that for any $E'\geq 0$ such that $\Supp E'\subsetneq\Supp E$ and $E'\not=0$, $N_{\sigma}(X/U,D+E')\not=0$.

Since $D$ is big$/U$, we may write $D=A+F$ for some $\Rr$-divisor $F\geq 0$ and ample$/U$ $\Rr$-divisor $A$. Since $\Supp E\subset\Bb_+(D/U)$, we have $\Supp E\subset\Supp F$. Let $E_1,\dots,E_m$ be the irreducible components of $E$ and let
$$\epsilon:=\min_{1\leq i\leq m}\left\{\frac{\mult_{E_i}E}{\mult_{E_i}F}\right\}.$$
Then we may write
$$F+E=(1+\epsilon)F+F_1-F_2$$
such that $F_1\geq 0,F_2\geq 0$, $F_1\wedge F_2=0$, $E\wedge F_2=0$, and $\Supp F_1\subsetneq\Supp E$.

If $F_1=0$, then 
$$N_{\sigma}(X/U,(1+\epsilon)D-\epsilon A)=N_{\sigma}(X/U,D+E+F_2)\leq F_2.$$
Thus 
$$\Supp E\subset\Bb_+(D/U)=\Bb_+((1+\epsilon)D/U)\subset\Bb_-(((1+\epsilon)D-\epsilon A)/U)\subset\Supp F_2,$$
which is not possible as $E\wedge F_2=0$. Thus $F_1\not=0$, so 
$$\Supp F_1\supset N_{\sigma}(X/U,(1+\epsilon)D+F_1)=N_{\sigma}(X/U,\epsilon A+(D+E)+F_2)\subset\Supp F_2.$$
Since $F_1\wedge F_2=0$,
$$N_{\sigma}(X/U,(1+\epsilon)D+F_1)=N_{\sigma}(X/U,\epsilon A+(D+E)+F_2)=0,$$
so 
$$N_{\sigma}\left(X/U,D+\frac{1}{1+\epsilon}F_1\right)=0.$$
This contradicts our assumption as $\Supp F_1\subsetneq\Supp E$ and $F_1\not=0$. 
\end{proof}

\begin{lem}\label{lem: sbl under pullback}
 Let $X\rightarrow U$ be a projective morphism from a normal quasi-projective variety to a quasi-projective variety, and $h: Y\rightarrow X$ a birational morphism. Assume that $X$ and $Y$ are $\Qq$-factorial. Let $D$ be a big$/U$ and movable$/U$ $\Rr$-divisor on $X$. Then
 $$\Supp\Bb_+(h^\ast D/U)=h^{-1}_\ast \Supp\Bb_+(D/U)\cup\Exc(h).$$
\end{lem}
\begin{proof}
    Let $A$ be an ample$/U$ $\Rr$-divisor on $X$ such that
    $$\Bb_-((D-A)/U)=\Bb_+(D/U).$$
    Then
    $$\Bb_-((D-\epsilon A)/U)=\Bb_+(D/U)$$
    for any real number $\epsilon\in (0,1]$. Let $E\geq 0$ be an anti-ample$/X$ $\Rr$-divisor such that $h^\ast A-E$ is ample$/U$ and $\Supp E=\Exc(h)$. Then
    $$\Supp\Bb_+(h^\ast D/U)=\Supp\Bb_-((h^\ast D-\epsilon(h^\ast A-E))/U)=\Supp\Bb_-((h^\ast (D-\epsilon A)+\epsilon E)/U)$$
    for some real number $\epsilon>0$. By Lemma \ref{lem: nz keep under pullback}, we have
\begin{align*}
 &\Supp\Bb_-((h^\ast (D-\epsilon A)+\epsilon E)/U)=\Supp N_{\sigma}(X/U,h^\ast (D-\epsilon A)+\epsilon E)\\
 =&\Supp(h^\ast N_{\sigma}(X/U,D-\epsilon A)+\epsilon E)=h^{-1}_\ast \Supp N_{\sigma}(X/U,D-\epsilon A)\cup\Exc(h)\\
 =&h^{-1}_\ast \Supp\Bb_-((D-\epsilon A)/U)\cup\Exc(h)=h^{-1}_\ast \Supp\Bb_+(D/U)\cup\Exc(h).
\end{align*}
The lemma follows.
\end{proof}

 \section{Qdlt algebraically integrable adjoint foliated structures}\label{sec: qdlt modification}

In this section we introduce the concept of qdlt singularities for algebraically integrable adjoint foliated structures and study its properties. 

\subsection{Definition of qdlt}

\begin{defn}[Qdlt]\label{defn: qdlt afs}
    We say that an adjoint foliated structure $(X,\Ff,B,\Mm,t)$ is \emph{(sub-)qdlt} if the following conditions hold. Note that this definition does not require the foliation to be algebraically integrable.
    \begin{enumerate}
        \item $t<1$.
        \item $(X,\Ff,B,\Mm,t)$ is (sub-)lc.
        \item $(X,B_X:=B^{\ninv}+\frac{1}{1-t}B^{\inv},\Mm)$ is (sub-)qdlt.
        \item Any lc place of $(X,\Ff,B,\Mm,t)$ is an lc place of $(X,B_X,\Mm)$.
    \end{enumerate}
     In particular, $K_X+B_X+\Mm_X$ is $\Rr$-Cartier.
\end{defn}

There are two important properties of qdlt algebraically integrable adjoint foliated structures. The first property is that qdlt is preserved under small perturbations:

\begin{lem}\label{lem: perturb qdlt afs}
Let $(X,\Ff,B,\Mm,t)/X$ be a sub-qdlt algebraically integrable adjoint foliated structure such that $0<t<1$,
and let $B_X:=B^{\ninv}+\frac{1}{1-t}B^{\inv}$. Let $D$ be an $\Rr$-divisor, $\Nn$ a $\bb$-divisor, and $\lambda$ a real number, such that 
\begin{enumerate}
    \item $D+\Nn_X+\lambda(K_{\Ff}-K_X)$ is $\Rr$-Cartier, 
    \item $D=0$ near the generic point of any strata of $\lfloor B_X\rfloor$, 
    \item $\Nn$ descends to $X$  near the generic point of any strata of $\lfloor B_X\rfloor$, and
    \item $\Mm+\epsilon_1\Nn$ is nef$/X$ for some $\epsilon_1>0$.
\end{enumerate}
Then there exists a positive real number $\epsilon_0$, such that
$$(X,\Ff,B+\epsilon D,\Mm+\epsilon\Nn,t+\lambda\epsilon)$$
is sub-qdlt for any $\epsilon\in [0,\epsilon_0]$.
\end{lem}
\begin{proof}
Let $h: X'\rightarrow X$ be a foliated log resolution of $(X,\Ff,B+D,\Mm+\Nn)$ associated with an equidimensional toroidal contraction $f: (X',\Sigma_{X'})\rightarrow (Z,\Sigma_{Z})$, and let $F_i$ be the irreducible components of $\Sigma_{X'}$. Then for each $i$, 
$$a_i: x\rightarrow a(F_i,X,\Ff,B+xD,\Mm+x\Nn,t+\lambda x)$$
is a continuous function for any $0\leq x\leq\min\{1-t,\epsilon_1\}$, and $a_i(0)\geq-t\epsilon_{\Ff}(F_i)-(1-t)$ for each $i$. Therefore, there exists $\epsilon_2>0$, such that $\epsilon_1>\epsilon_2$, and for any $i$ such that $a_i(0)>-t\epsilon_{\Ff}(F_i)-(1-t)$, we have $a_i(x)>-(t+\lambda x)\epsilon_{\Ff}(F_i)-(1-t-\lambda x)$ for any $x\in [0,\epsilon_2]$. Moreover, if $a_i(0)=-t\epsilon_{\Ff}(F_i)-(1-t)$, then $F_i$ is an lc place of $(X,\Ff,B,\Mm,t)$, so $F_i$ is an lc place of $(X,B_{X},\Mm)$, hence an lc place of $(X,\Ff,B^{\ninv},\Mm)$.
Since $(X,B_{X},\Mm)$ is sub-qdlt, 
 $\Center_XF_i$ is a stratum of $\lfloor B_X\rfloor$. Therefore, near the generic point of $\Center_{X}F_i$, $D=0$ and $\Nn$ descends to $X$. We have
\begin{align*}
   a_i(x)&=a(F_i,X,\Ff,B,\Mm,t+\lambda x)\\
   &=(t+\lambda x)a(F_i,\Ff,B^{\ninv},\Mm)+(1-t-\lambda x)a(F_i,X,B_{X},\Mm)\\
   &=-(t+\lambda x)\epsilon_{\Ff}(F_i)-(1-t-\lambda x). 
\end{align*}
Therefore, $a_i(x)\geq -(t+\lambda x)\epsilon_{\Ff}(F_i)-(1-t-\lambda x)$ for any $i$ and any $x\in [0,\epsilon_2]$. By Lemma \ref{lem: foliated log smooth imply lc},  $$(X,\Ff,B+xD,\Mm+x\Nn,t+\lambda x)$$
is sub-lc for any $x\in [0,\epsilon_2]$. 

Since $(X,B_X,\Mm)$ is sub-qdlt, by conditions (2-4), there exists a positive real number $\epsilon_0<\frac{\epsilon_2}{2}$ such that for any $x\in [0,2\epsilon_0]$, $$\left(X,B_x:=B_X+x\left(D^{\ninv}+\frac{1}{1-t}D^{\inv}\right),\Mm+x\Nn\right)$$
is qdlt and any lc place of $(X,B_x,\Mm+x\Nn)$ is an lc place of $(X,B_X,\Mm)$. Now for any $\epsilon\in [0,\epsilon_0]$, we have the following.
\begin{itemize}
    \item Since $t+2\lambda\epsilon\leq 1$ and $t<1$, $t+\lambda\epsilon<1$,
    \item $(X,B_\epsilon,\Mm+\epsilon\Nn)$ is sub-qdlt, and $B_{\epsilon}=(B+\epsilon D)^{\ninv}+\frac{1}{1-t}(B+\epsilon D)^{\inv}$.
    \item Since $(X,\Ff,B,\Mm,t)$ and $(X,\Ff,B+2\epsilon D,\Mm+2\epsilon\Nn,t+2\lambda\epsilon)$ are sub-lc, $(X,\Ff,B+\epsilon D,\Mm+\epsilon\Nn,t+\lambda\epsilon)$ is sub-lc. Moreover, any lc place of $(X,\Ff,B+\epsilon D,\Mm+\epsilon\Nn,t+\lambda\epsilon)$ is an lc place of $(X,\Ff,B,\Mm,t)$, hence an lc place of $(X,B_X,\Mm)$, and hence an lc place of $(X,B_{\epsilon},\Mm+\epsilon\Nn)$.
\end{itemize}
Thus $\epsilon_0$ satisfies our requirements.
\end{proof}

The second property is that qdlt is preserved under the minimal model program. We first prove the following important proposition which characterizes the singularities of the ambient variety of any adjoint foliated structure with $t<1$:

\begin{prop}\label{prop: adjoint lc implies X lc}
     Let $(X,\Ff,B,\Mm,t)/X$ be an algebraically integrable adjoint foliated structure such that $t<1$. Let $B_X$ be an $\Rr$-divisor on $X$ and let $\Nn$ be a nef$/X$ $\bb$-divisor on $X$ such that
     \begin{enumerate}
     \item $B^{\ninv}\geq B_X^{\ninv}$ and $B^{\inv}\geq (1-t)B_X^{\inv}$,
     \item $\Mm-\Nn$ is nef$/X$, and
     \item $K_X+B_X+\Nn_X$ is $\Rr$-Cartier.
     \end{enumerate}
     Then
$$\Nlc(X,B_X,\Nn)\subset\Nlc(X,\Ff,B,\Mm,t)$$
and
$$\Nklt(X,B_X,\Nn)\subset\Nklt(X,\Ff,B,\Mm,t).$$
     In particular, if $(X,\Ff,B,\Mm,t)$ is lc (resp. klt), then $(X,B_X,\Nn)$ is sub-lc (resp. sub-klt).
\end{prop}
\begin{proof}
We may assume that $t>0$, otherwise the proposition is trivial. 

Let $B_{\Ff}$ be the unique $\Rr$-divisor on $X$ such that $tB_{\Ff}+(1-t)B_X=B$. Let $\Pp$ be the unique $\bb$-divisor such that $t\Pp+(1-t)\Nn=\Mm$. 
Then $\Pp=\Nn+\frac{1}{t}(\Mm-\Nn)$ is nef$/X$, and $\Pp-\Nn$ is nef$/X$. Since
$$t(K_{\Ff}+B_{\Ff}+\Pp_X)+(1-t)(K_X+B_X+\Nn_X)=K_{(X,\Ff,B,\Mm,t)},$$
$K_{\Ff}+B_{\Ff}+\Pp_X$ is $\Rr$-Cartier. Since
$tB_{\Ff}^{\inv}=B^{\inv}-(1-t)B_X^{\inv}\geq 0,$
 $B_{\Ff}^{\inv}\geq 0$. Since
$tB_{\Ff}^{\ninv}=B^{\ninv}-(1-t)B_X^{\ninv}\geq tB^{\ninv},$
$B_{\Ff}^{\ninv}\geq B^{\ninv}$. 

Since lc and klt are local properties, we may assume that $(X,\Ff,B,\Mm,t)$ is lc (resp. klt). Let $h: (X',\Ff',B_{\Ff'},\Pp;G)/Z\rightarrow (X,\Ff,B_{\Ff},\Pp)$ be a $\Qq$-factorial proper ACSS modification. By definition, 
$$B_{\Ff'}=h^{-1}_\ast (B_{\Ff}^{\ninv}\wedge\Supp B_{\Ff}^{\ninv})+\Supp\Exc(h)^{\ninv},$$
and for any $h$-exceptional prime divisor $E$, 
$$a(E,\Ff,B_{\Ff},\Pp)\leq -\epsilon_{\Ff}(E).$$ Since $(X,\Ff,B,\Mm,t)$ is lc (resp. klt), 
$$a(E,X,\Ff,B,\Mm,t)\geq\text{( resp. }>\text{)} -t\epsilon_{\Ff}(E)-(1-t).$$ 
Since
$$a(E,X,\Ff,B,\Mm,t)=ta(E,\Ff,B_{\Ff},\Pp)+(1-t)a(E,X,B_X,\Nn),$$
we have $a(E,X,B_X,\Nn)\geq -1$ (resp. $>-1$).
Let $B_{X'}:=h^{-1}_*B_X+\Exc(h)$. Then 
$$K_{X'}+B_{X'}+\Nn_{X'}\geq K_{X'}+\bar B_{X'}+\Nn_{X'}:=h^*(K_X+B_X+\Nn_X).$$
\begin{claim}
$B_{\Ff'}+G\geq B_{X'}$, and if $(X,\Ff,B,\Mm,t)$ is klt, then $\lfloor \bar B_{X'}\rfloor\leq 0$. 
\end{claim}
\begin{proof}
    For any prime divisor $D$ that is a component of $B_{X'}$, there are four cases:  

\medskip

\noindent\textbf{Case 1}. $D$ is a component of $\Exc(h)^{\inv}$. Then $D$ is a component of $G$, so
$$\mult_D(B_{\Ff'}+G)=1=\mult_D\Supp\Exc(h)^{\inv}=\mult_DB_{X'}.$$
If $(X,\Ff,B,\Mm,t)$ is klt, then $1=\mult_DB_{X'}>\mult_D\bar B_{X'}$, so $\mult_D\lfloor \bar B_{X'}\rfloor\leq 0$.

\medskip

\noindent\textbf{Case 2}. $D$ is a component of $\Exc(h)^{\ninv}$. In this case,
$$\mult_D(B_{\Ff'}+G)=1=\mult_D\Supp\Exc(h)^{\ninv}=\mult_DB_{X'}\geq\mult_D\bar B_{X'}$$
by definitions of $B_{\Ff'}$ and $B_{X'}$ directly. If $(X,\Ff,B,\Mm,t)$ is klt, then $$1=\mult_DB_{X'}>\mult_D\bar B_{X'},$$ 
so $\mult_D\lfloor \bar B_{X'}\rfloor\leq 0$.

\medskip

\noindent\textbf{Case 3}. $D$ is a component of $h^{-1}_*B_X$ that is $\Ff'$-invariant. Then $D$ is a component of $G$, so
$$\mult_D(B_{\Ff'}+G)=1.$$
Since $h_*D$ is an $\Ff$-invariant prime divisor,
$$a(D,X,\Ff,B,\Mm,t)\geq -1+t.$$
and strict inequality holds if $(X,\Ff,B,\Mm,t)$ is klt. Since $B_{\Ff}^{\inv}\geq 0$,
$$a(D,X,\Ff,B_{\Ff},\Pp)\leq 0.$$
Thus $a(D,X,B_X,\Nn)\geq -1$ and strict inequality holds if $(X,\Ff,B,\Mm,t)$ is klt. So $\mult_DB_{X'}\leq 1$ and strict inequality holds if $(X,\Ff,B,\Mm,t)$ is klt. Therefore,
$$\mult_D(B_{\Ff'}+G)=1\geq\mult_DB_{X'}.$$
and strict inequality holds if $(X,\Ff,B,\Mm,t)$ is klt, in which case $\mult_D\lfloor \bar B_{X'}\rfloor\leq \mult_D\lfloor B_{X'}\rfloor\leq0$.

\medskip

\noindent\textbf{Case 4}. $D$ is a component of $h^{-1}_*B_X$ that is not $\Ff'$-invariant. In this case, 
$$\mult_D(B_{\Ff'}+G)=\mult_DB_{\Ff'}=\min\{\mult_{h_*D}B_{\Ff}^{\ninv},1\}\geq\mult_{h_*D}B_X^{\ninv}=\mult_DB_{X'}.$$
Moreover, if $(X,\Ff,B,\Mm,t)$ is klt, then 
$$\mult_DB_{X'}=\mult_{h_*D}B_{X}^{\ninv}\leq\mult_{h_*D}B^{\ninv}<1,$$
 in which case $\mult_D\lfloor \bar B_{X'}\rfloor\leq \mult_D\lfloor B_{X'}\rfloor\leq0$.
\end{proof}

\noindent\textit{Proof of Proposition \ref{prop: adjoint lc implies X lc} continued.} Since $B_{\Ff'}+G\geq B_{X'}$, $\Pp-\Nn$ is nef$/X$, and $(X',B_{\Ff'}+G,\Pp)$ is qdlt, we have that $(X',B_{X'},\Nn)$ is sub-qdlt. Thus $(X',\bar B_{X'},\Nn)$ is sub-qdlt, and $\lfloor \bar B_{X'}\rfloor\leq 0$ if $(X,\Ff,B,\Mm,t)$ is klt. Thus $(X,B_X,\Nn)$ is sub-lc and is sub-klt if $(X,\Ff,B,\Mm,t)$ is klt.
\end{proof}

\begin{prop}\label{prop: qdlt preserved under mmp}
    Let $(X,\Ff,B,\Mm,t)/U$ be a qdlt adjoint foliated structure, $K:=K_{(X,\Ff,B,\Mm,t)}$, and let $\phi: (X,\Ff,B,\Mm,t)\dashrightarrow (X',\Ff',B',\Mm,t)$ be a step of a $K$-MMP. Assume that $K_{X'}+B'^{\ninv}+\frac{1}{1-t}B'^{\inv}+\Mm_{X'}$ is $\Rr$-Cartier. Then $(X',\Ff',B',\Mm,t)$ is qdlt.
\end{prop}
\begin{proof}
    We let $B_X:=B^{\ninv}+\frac{1}{1-t}B^{\inv}$, and let $B_{X'}:=B'^{\ninv}+\frac{1}{1-t}B'^{\inv}$.
    
    We only need to check that all conditions of Definition \ref{defn: qdlt afs} holds for $(X',\Ff',B',\Mm,t)$. (1) of Definition \ref{defn: qdlt afs} is obvious while (2) of Definition \ref{defn: qdlt afs} follows from the fact that the minimal model program does not decrease discrepancies. 

    Let $E$ be an lc place of $(X',\Ff',B',\Mm,t)$. Then $E$ is also an lc place of $(X,\Ff,B,\Mm,t)$. Therefore, $\phi^{-1}$ is an isomorphism near the generic point of $\Center_{X'}E$. In particular, near the generic point of $\Center_{X'}E$, $\phi^{-1}:(X',B_{X'},\Mm)\rightarrow (X,B_X,\Mm)$ is an isomorphism. Thus $E$ is an lc place of $(X',B_{X'},\Mm)$. This implies (4) of Definition \ref{defn: qdlt afs}. 
    
    Let $W$ be any lc center of $(X',B_{X'},\Mm)$. By Proposition \ref{prop: adjoint lc implies X lc}, $W$ is an nklt center of $(X',\Ff',B',\Mm,t)$, so $\phi^{-1}:(X',B_{X'},\Mm)\rightarrow (X,B_X,\Mm)$ is an isomorphism near the generic point of $W$. Since $(X,B_X,\Mm)$ is qdlt, $(X',B_{X'},\Mm)$ is qdlt near $W$. This implies (3) of Definition \ref{defn: qdlt afs}. 
\end{proof}

\subsection{Existence of qdlt models}

\begin{thm}[Existence of $\Qq$-factorial qdlt models]\label{thm: eoqdlt model}
    Let $(X,\Ff,B,\Mm,t)/U$ be an algebraically integrable adjoint foliated structure such that $t\in [0,1)$.
    Let $B_t:=B\wedge(\Supp B^{\ninv}+(1-t)\Supp B^{\inv})$. Then there exists a projective birational morphism $h: X'\rightarrow X$ satisfying the following. Let $\Ff':=h^{-1}\Ff$, $$B':=h^{-1}_*B_t+\Exc(h)^{\ninv}+(1-t)\Exc(h)^{\inv}.$$
    Then the following hold.
    \begin{enumerate}
    \item $(X',\Ff',B',\Mm,t)$ is $\Qq$-factorial qdlt. In particular, $X'$ is klt.
    \item $h$ only extracts nklt places of $(X,\Ff,B,\Mm,t)$. In particular, if $(X,\Ff,B,\Mm,t)$ is lc, then
    $$K_{(X',\Ff',B',\Mm,t)}=h^*K_{(X,\Ff,B,\Mm,t)}.$$
    \end{enumerate}
\end{thm}
\begin{proof}
We let $g: W\rightarrow X$ be a foliated log resolution of $(X,\Ff,B,\Mm)$ associated with an equidimensional toroidal contraction $f: (X,\Sigma)\rightarrow (Z,\Sigma_Z)$, $\Ff_W:= g^{-1}\Ff$, and 
$$B_W:=g^{-1}_*(B\wedge\Supp B^{\ninv})+\Exc(g)^{\ninv}$$
Let $G_W$ be the $\Ff_W$-invariant part of $\Sigma$. Since $g$ only extract finitely many divisors, there exists a positive real number $s$ satisfying the following: for any $g$-exceptional prime divisor $D$, if 
$$a(D,X,\Ff,B,\Mm,t)>-t\epsilon_{\Ff}(D)-(1-t)$$
then 
$$a(D,X,\Ff,B,\Mm,t)>-t\epsilon_{\Ff}(D)-(1-t)+s.$$

Since $g$ is birational, there exists an anti-ample$/X$ divisor $A_W\geq 0$ such that $\Supp\Exc(g)\subset\Supp A_W$. Possibly rescaling $A_W$, we may assume that
$$t\mult_DA_W<s$$
for any component $D$ of $A_W$.

Let $\Nn:=\overline{-A_W}$. Since $(W,\Ff_W,B_W+G_W,\Mm)$ is foliated log smooth, $(W,\Ff_W,B_W+G_W,\Mm+\Nn)/X$ is foliated log smooth. In particular, $(W,\Ff_W,B_W,\Mm+\Nn;G)/Z$ is $\Qq$-factorial ACSS (cf. \cite[Lemma 7.3.3]{CHLX23}).  

 Since $-A_W$ is ample$/X$, by \cite[Theoerm 16.1.4]{CHLX23}, we may run a 
$$(K_{\Ff_W}+B_W+\Mm_W+\Nn_W)\text{-MMP}/X$$
with scaling of an ample$/X$ divisor which terminates with a good minimal model $(Y,\Ff_Y,B_Y,\Mm+\Nn)/X$ of $(W,\Ff_W,B_W,\Mm+\Nn)/X$. We let $A_Y$ and $G_Y$ be the images of $A_W$ and $G_W$ on $Y$ respectively and let $p: Y\rightarrow X$ be the induced morphism. Then $(Y,\Ff_Y,B_Y,\Mm+\Nn;G_Y)/Z$ is ACSS, so $(Y,B_Y+G_Y,\Mm+\Nn)/Z$ is qdlt. In particular, 
$$(Y,\Ff_Y,B_Y+(1-t)G_Y,\Mm+t\Nn;t)$$ is lc. Let 
$$B_Y':=p^{-1}_*B_t+\Exc(p)^{\ninv}+(1-t)\Exc(p)^{\inv}.$$ We claim that  $B_Y+(1-t)G_Y\geq B_Y'$ and $B_Y+(1-t)G_Y-B_Y'$ is $\Ff_Y$-invariant. To see this: let $D_Y$ be a prime divisor on $Y$ that is a component of $B_Y'$. Let $D=p_*D_Y$ and let $D_W:=\Center_WD_Y$. Then:
\begin{itemize}
    \item If $D_Y$ 
    is not $\Ff_Y$-invariant and is not exceptional$/X$, then 
    \begin{align*}
\mult_{D_Y}(B_Y+(1-t)G_Y)&=\mult_{D_Y}B_Y=\mult_{D_W}B_W=\mult_D(B\wedge\Supp B^{\ninv})\\
    &=\mult_DB_t=\mult_{D_Y}p^{-1}_*B_t=\mult_{D_Y}B_Y'. 
    \end{align*}
    \item If $D_Y$ is not $\Ff_Y$-invariant and is exceptional$/X$, then 
    \begin{align*}
\mult_{D_Y}(B_Y+(1-t)G_Y)&=\mult_{D_Y}B_Y=\mult_{D_W}B_W=\mult_{D_W}\Exc(g)^{\ninv}\\
    &=1=\mult_{D_Y}\Exc(p)^{\ninv}=\mult_{D_Y}B_Y'. 
    \end{align*}
    \item If $D_Y$ is $\Ff_Y$-invariant and is not exceptional$/X$, then 
    \begin{align*}
\mult_{D_Y}(B_Y+(1-t)G_Y)&=(1-t)\mult_{D_Y}G_Y=1-t\\
    &\geq\mult_DB_t=\mult_{D_Y}p^{-1}_*B_t=\mult_{D_Y}B_Y'. 
    \end{align*}
    \item If $D_Y$ is $\Ff_Y$-invariant and is exceptional$/X$, then 
    \begin{align*}
\mult_{D_Y}(B_Y+(1-t)G_Y)&=(1-t)\mult_{D_Y}G_Y\\
&=1-t=(1-t)\mult_{D_Y}\Exc(p)^{\inv}=\mult_{D_Y}B_Y'. 
    \end{align*}
\end{itemize}
Therefore, $B_Y+(1-t)G_Y\geq B_Y'$ and $B_Y+(1-t)G_Y-B_Y'$ is $\Ff_Y$-invariant. Let $B_Y'':=\frac{1}{1-t}(B_Y'-tB_Y)$. Then $B_Y+G_Y\geq B_Y''$ and 
$$B_Y''=B_Y'^{\ninv}+\frac{1}{1-t}B_Y'^{\inv}\geq 0.$$ In particular, $(Y,B_Y'',\Mm)$ is qdlt. Thus
\begin{align*}
  &tK_{\Ff_Y}+(1-t)K_Y+B_Y'+\Mm_Y+t\Nn_Y\\
  =&t(K_{\Ff_Y}+B_Y+\Mm_Y+\Nn_Y)+(1-t)(K_Y+B_Y''+\Mm_Y)\\
  =&(1-t)\left((K_Y+B_Y''+\Mm_Y)+\frac{t}{1-t}(K_{\Ff_Y}+B_Y+\Mm_Y+\Nn_Y)\right),
\end{align*}
where $(Y,B_Y'',\Mm)$ is a qdlt generalized pair and $K_{\Ff_Y}+B_Y+\Mm_Y+\Nn_Y$ is nef$/X$, so
$$\left(Y,B_Y'',\Pp:=\Mm+\overline{\frac{t}{1-t}(K_{\Ff_Y}+B_Y+\Mm_Y+\Nn_Y)}\right)/X$$
is a qdlt generalized pair. By \cite[Lemma 4.4]{BZ16}, we may run a 
$$(tK_{\Ff_Y}+(1-t)K_Y+B_Y'+\Mm_Y+t\Nn_Y)\text{-MMP}/X$$
with scaling of an ample divisor.

Let $E_i$ be the $p$-exceptional prime divisors and let
$$tK_{\Ff_Y}+(1-t)K_Y+p^{-1}_*B+\sum e_iE_i+\Mm_Y:=p^*(tK_{\Ff}+(1-t)K_X+B+\Mm_X).$$
By our construction, we may write $-\Nn_Y=\sum n_iE_i+L_Y$ such that $tn_i<s$ for each $i$, $L_Y\geq 0$, and no component of $L_Y$ is exceptional over $X$.
Then
\begin{align*}
&tK_{\Ff_Y}+(1-t)K_Y+B_Y'+\Mm_Y+t\Nn_Y\\
\sim_{\mathbb R,X}&B_Y'-p^{-1}_*B-\sum e_iE_i+t\Nn_Y\\
=&p^{-1}_*(B_t-B)+\sum (t\epsilon_{\Ff}(E_i)+(1-t)-e_i-tn_i)E_i-tL_Y\\
:=&E^+-E^-,
\end{align*}
where $E^+\geq 0,E^{-}\geq 0$ have no common component. Then $E^+$ is exceptional$/X$, and $\Supp E^+$ consists of exactly all $E_i$ such that 
$$t\epsilon_{\Ff}(E_i)+(1-t)-e_i-tn_i>0.$$
Since $e_i=-a(E,X,\Ff,B,\Mm,t)$, $\Supp E^+$ consists of exactly all $E_i$ such that
$$a(E,X,\Ff,B,\Mm,t)>-t\epsilon_{\Ff}(E_i)-(1-t)+tn_i.$$
Since $tn_i<s$ for each $i$, 
by our construction of $s$, $\Supp E^+$ consists of exactly all $E_i$ such that
$$a(E,X,\Ff,B,\Mm,t)>-t\epsilon_{\Ff}(E_i)-(1-t).$$
By \cite[Proposition 3.9]{HL22}, the 
$$(tK_{\Ff_Y}+(1-t)K_Y+B_Y'+\Mm_Y+t\Nn_Y)\text{-MMP}/X$$
contracts $E^+$ after finitely many steps and achieves a model $X'$. We let $h: X'\rightarrow X$ be the induced birational morphism, $\psi: Y\dashrightarrow X'$ the induced birational map, $\Ff':=h^{-1}_*\Ff$, $$B':=h^{-1}_*B_t+\Exc(h)^{\ninv}+(1-t)\Exc(h)^{\inv},$$
and $B_{X'}:=B'^{\ninv}+\frac{1}{1-t}B'^{\inv}$. We show that $h$ satisfies our requirement. 

(1) The $\Qq$-factoriality is clear as $\psi$ is a sequence of steps of an MMP of a  $\Qq$-factorial lc generalized pair. Next we check the conditions of Definition \ref{defn: qdlt afs}.

Definition \ref{defn: qdlt afs}(1): it is in the assumption.

Definition \ref{defn: qdlt afs}(2): $B'$ is the image of $B_Y'$ on $X'$. Since $B_Y+(1-t)G_Y\geq B_Y'$,
$$(Y,\Ff_Y,B_Y',\Mm+t\Nn;t)$$ is lc. Since $\psi$ is a sequence of steps of a 
$$(tK_{\Ff_Y}+(1-t)K_Y+B_Y'+\Mm_Y+t\Nn_Y)\text{-MMP}/X,$$
$(X',\Ff',B',\Mm+t\Nn;t)$
is lc, so $(X',\Ff',B',\Mm;t)$ is lc.

Definition \ref{defn: qdlt afs}(3): Since 
$$B_Y''=B_Y'^{\ninv}+\frac{1}{1-t}B_Y'^{\inv},$$
$B_{X'}$ is the image of $B_Y''$ on $X'$. Since $(Y,B''_{Y},\Pp)$ is qdlt and $\psi$ is a sequence of steps of a $(K_{Y}+B''_{Y}+\Pp_Y)$-MMP, $(X',B_{X'},\Pp)$ is qdlt. Since $\Pp-\Mm$ is nef, $(X',B_{X'},\Mm)$ is qdlt. 

Definition \ref{defn: qdlt afs}(4): Let $E$ be an lc place of $(X',\Ff',B',\Mm,t)$. Then
\begin{align*}
-t\epsilon_{\Ff}(E)-(1-t)&=a(E,X',\Ff',B',\Mm,t)\\
&\geq a(E,X',\Ff',B',\Mm+t\Nn,t)\\
&\geq a(E,Y,\Ff_Y,B_Y',\Mm+t\Nn,t)\\
&=ta(E,\Ff_Y,B_Y,\Mm+\Nn)+(1-t)a(E,Y,B_Y'',\Mm)\\
&=ta(E,\Ff_Y,B_Y,\Mm+\Nn)+(1-t)a(E,Y,B_Y'',\Pp)\\
&\geq -t\epsilon_{\Ff}(E)-(1-t).
\end{align*}
Therefore, $a(E,Y,B_Y'',\Pp)=-1$ and 
$$a(E,X',\Ff',B',\Mm+t\Nn,t)=a(E,Y,\Ff_Y,B_Y',\Mm+t\Nn,t),$$
so $E$ is an lc place of $(Y,B_Y'',\Pp)$ and $\psi^{-1}$ is an isomorphism near the generic point of $\Center_{X'}E$. In particular,
$$a(E,X',B_{X'},\Pp)=a(E,Y,B_{Y}'',\Pp)=-1.$$
Since $(X',B_{X'},\Pp)$ is qdlt, $\Pp$ descends to $X'$ near the generic point of $\Center_{X'}E$. Therefore, 
$$a(E,X',B',\Mm)=a(E,X',B_{X'},\Pp)=-1.$$
Thus $E$ is an lc place of $(X',B',\Mm)$.

Finally, the in particular part: since $(X',B_{X'},\Mm)$ is $\Qq$-factorial qdlt, $X'$ is klt.

(2) It follows from the fact that any $p$-exceptional divisor that is not an nklt place of $(X,\Ff,B,\Mm,t)$ is contained in $\Supp E^+$ and is contracted by $\psi$. 
\end{proof}

\begin{defn}
    Notation as in Theorem \ref{thm: eoqdlt model}. For any such $h$, we call
    $$h: (X',\Ff',B',\Mm,t)\rightarrow (X,\Ff,B,\Mm,t)$$
    a \emph{$\Qq$-factorial qdlt modification}, and say that $(X',\Ff',B',\Mm,t)$ is a \emph{$\Qq$-factorial qdlt model} of $(X,\Ff,B,\Mm,t)$. We also say that $h$ is a \emph{$\Qq$-factorial qdlt modification} of $(X,\Ff,B,\Mm,t)$.
\end{defn}

\begin{proof}[Proof of Theorem \ref{thm: qdlt model intro}]
    The Theorem follows from Theorem \ref{thm: eoqdlt model} and Lemma \ref{lem: perturb qdlt afs}.
\end{proof}

Sometimes we need the following strengthened version of  Theorem \ref{thm: qdlt model intro} which indicates that we can extract lc places via $\Qq$-factorial qdlt modifications.

\begin{prop}\label{prop: qdlt model extract certain divisors}
Let $(X,\Ff,B,\Mm,t)/U$ be an algebraically integrable adjoint foliated structure such that $t<1$. Let $E_1,\dots,E_m$ be lc places of $(X,\Ff,B,\Mm,t)$ such that $(X,\Ff,B,\Mm,t)$ is lc near the generic point of $\Center_{X_i}E_i$ for each $i$. Then there exists a $\Qq$-factorial qdlt modification $h: (X',\Ff',B',\Mm,t)\rightarrow (X,\Ff,B,\Mm,t)$ such that $E_1,\dots,E_m$ are on $X'$.
\end{prop}
\begin{proof}
First we take a $\Qq$-factorial qdlt modification $g: (X'',\Ff'',B'',\Mm,t)\rightarrow (X,\Ff,B,\Mm,t)$ whose existence is guaranteed by Theorem \ref{thm: eoqdlt model}. Since $(X,\Ff,B,\Mm,t)$ is lc near the generic point of $\Center_{X_i}E_i$ for each $i$, over the generic point of $\Center_XE_i$ for each $i$, $g$ only extract lc places of $(X,\Ff,B,\Mm,t)$. Therefore, over the generic point of $\Center_XE_i$ for each $i$, $K_{(X'',\Ff'',B'',\Mm,t)}=g^*K_{(X,\Ff,B,\Mm,t)}$. Thus $E_1,\dots,E_m$ are lc places of $(X'',\Ff'',B'',\Mm,t)$. By Theorem \ref{thm: eoqdlt model}(4), $E_1,\dots,E_m$ are lc places of $(X'',B_{X''},\Mm)$ where $B_{X''}=B''^{\ninv}+\frac{1}{1-t}B''^{\inv}$. Since $X''$ is $\Qq$-factorial klt, by \cite[Corollary 1.4.3]{BCHM10}, there exists a birational morphism $p: Y\rightarrow X''$ which exactly extracts $E_1,\dots,E_m$. Let 
$$K_{(Y,\Ff_Y,B_Y,\Mm,t)}=p^*K_{(X'',\Ff'',B'',\Mm,t)},$$
then $B_Y\geq 0$ and $p$ only extracts lc places of $(X'',\Ff'',B'',\Mm,t)$. Let $q: (X',\Ff',B',\Mm,t)\rightarrow (Y,\Ff_Y,B_Y,\Mm,t)$ be a $\Qq$-factorial qdlt modification and let $h:=g\circ p\circ q$. 

Then $h: (X',\Ff',B',\Mm,t)\rightarrow (X,\Ff,B,\Mm,t)$ satisfies our requirements. More precisely, since $p$ extracts $E_1,\dots,E_m$, $E_1,\dots,E_m$ are also on $X'$. The fact  $q: (X',\Ff',B',\Mm,t)\rightarrow (Y,\Ff_Y,B_Y,\Mm,t)$ is a $\Qq$-factorial qdlt modification guarantees (2-5) of Theorem \ref{thm: eoqdlt model} to hold for $h$. Since $g$ only extracts nklt places of $(X,\Ff,B,\Mm,t)$, $p\circ q$ only extracts lc places of $(X'',\Ff'',B'',\Mm,t)$, and any nklt place of $(X'',\Ff'',B'',\Mm,t)$ is also an nklt place of $(X,\Ff,B,\Mm,t)$, (1) of Theorem \ref{thm: eoqdlt model}  holds for $h$. The proposition follows.
\end{proof}

\section{Adjunction formula}\label{sec: an adjunction formula}

\begin{thm}[$=$Theorem \ref{thm: adj intro}]\label{thm: afs adj}
    Let $(X,\Ff,B,\Mm,t)/U$ be an algebraically integrable adjoint foliated structure and $\tilde S$ a prime divisor on $X$ that is an lc place of $(X,\Ff,B,\Mm,t)$. Let $S$ be the normalization of $\tilde S$, and $\Mm^S:=\Mm|_S$. Then there exists a canonically defined restricted foliation $\Ff_S$ on $S$ and a canonically defined $\Rr$-divisor $B_S\geq 0$ on $S$ such that
    $$tK_{\Ff_S}+(1-t)K_X+B_S+\Mm^S_S=(tK_{\Ff}+(1-t)K_X+B+\Mm_X)|_S.$$
Moreover, if $(X,\Ff,B,\Mm,t)/U$ is lc, then $(S,\Ff_S,B_S,\Mm,t)/U$ is lc.
\end{thm}
\begin{proof}
We may assume that $t<1$ otherwise we may apply \cite[Theorem 2.4.2]{CHLX23} directly. 
First we construct $\mathcal F_S$ and $B_S$.
Let $g: W\rightarrow X$ be a foliated log resolution of $(X,\Ff,B,\Mm)$, $S_W:=h^{-1}\tilde S$, $\Ff_W:=h^{-1}\Ff$, and 
$$tK_{\Ff_W}+(1-t)K_{W}+B_W+\Mm_{W}:=g^*(tK_{\Ff}+(1-t)K_X+B+\Mm_X).$$
Let $\Mm^S:=\overline{\Mm_{W}|_{S_W}}$. Since $(W,\Ff_W,B_W,\Mm)$ is foliated log smooth, $\mult_{S_W}B_W^{\ninv}=\epsilon_{\Ff_W}(S_W)$, and $\mult_{S_W}(B_W^{\ninv}+\frac{1}{1-t}B_W^{\inv})=1$. By the adjunction formula for varieties we have
$$K_{S_W}+B_{S_W,W}+\Mm^S_{S_W}:=\left(K_{W}+B_W^{\ninv}+\frac{1}{1-t}B_W^{\inv}+\Mm_{W}\right)\Bigg|_{S_W}$$
and by the adjunction formula for foliations, \cite[Proposition-Definition 3.7]{CS23b}, we have
$$K_{\Ff_{S_W}}+B_{S_W,\Ff_W}+\Mm^S_{S_W}:=\left(K_{\Ff_W}+B_W^{\ninv}+\Mm_{W}\right)|_{S_W}$$
where $\Ff_{S_W}$ is a canonically defined restriction foliation and $B_{S_W,W}$ and $B_{S_W,\Ff_W}$ are canonically defined $\Rr$-divisors. We let $B_{S_W}:=tB_{S_W,W}+(1-t)B_{S_W,\Ff_W}$ and let $B_S:=(g_{S})_*B_{S_W}$, where $g_{S_W}: S_W\rightarrow S$ is the induced birational morphism. It is clear that $B_S$ is a canonically defined $\Rr$-divisor and satisfies our requirements.

Next we prove that, if $(X,\Ff,B,\Mm,t)$ is lc, then $(S,\Ff_S,B_S,\Mm,t)$ is sub-lc. We let $\bar B_W:=g^{-1}_*B+\Exc(g)^{
\ninv}+(1-t)\Exc(g)^{\inv}$. Since $(X,\Ff,B,\Mm,t)$ is lc, $\bar B_W\geq B_W$. We let
$$K_{S_W}+\bar B_{S_W,W}+\Mm^S_{S_W}:=\left(K_{W}+\bar B_W^{\ninv}+\frac{1}{1-t}\bar B_W^{\inv}+\Mm_{W}\right)\Bigg|_{S_W}$$
and
$$K_{S_W}+\bar B_{S_W,\Ff_W}+\Mm^S_{S_W}:=\left(K_{\Ff_W}+\bar B_W^{\ninv}+\Mm_{W}\right)|_{S_W}.$$
Then $\bar B_{S_W,W}\geq B_{S_W,W}$ and $\bar B_{S_W,\Ff_W}\geq B_{S_W,\Ff_W}$. By \cite[Theorem 2.4.2]{CHLX23}, $(W,\bar B_{S_W,W},\Mm^S)$ and $(W,\Ff_W,\bar B_{S_W,\Ff_W},\Mm^S)$ are lc, so $(S_W,B_{S_W,W},\Mm^S)$ and $(S_W,\Ff_W,B_{S_W,\Ff_W},\Mm^S)$ are sub-lc, and so $(S,\Ff,B_S,\Mm^S,t)$ is sub-lc.

Finally, we show that $B_S\geq 0$. By Theorem \ref{thm: qdlt model intro}, we may let $h: (X',\Ff',B',\Mm,t)\rightarrow (X,\Ff,B,\Mm,t)$ be a $\Qq$-factorial qdlt modification of $(X,\Ff,B,\Mm,t)$ and let $S':=h^{-1}_*\tilde S$. Then $S'$ is normal and 
$$K_{(X',\Ff',B',\Mm,t)}\leq h^*K_{(X,\Ff,B,\Mm,t)}.$$
Let
$$K_{(S',\Ff_{S'},B_{S'},\Mm^S,t)}:=K_{(X',\Ff',B',\Mm,t)}|_{S'}$$
and let $h_{S}: S'\rightarrow S$ be the induced birational morphism. Then $(h_S)_*B_{S'}=B_S$. Thus if $B_{S'}\geq 0$, then so is $B_S$. Possibly replacing $(X,\Ff,B,\Mm,t)$ with $(X,\Ff',B',\Mm,t)$, we may assume that $(X,\Ff,B,\Mm,t)$ is $\Qq$-factorial qdlt.

By \cite[Proposition-Definition 3.7]{CS23b}
we have 
\begin{align*}
(K_{\mathcal F}+B^{\ninv})\vert_S = K_{\mathcal F_S}+D_S,
\end{align*}
where $D_S \ge 0$.
By classical adjunction, it then follows that  
\begin{align*}
(K_X+B)\vert_S = K_S+E_S, 
\end{align*}
where $E_S \ge 0$.
Hence,
$B_S = tD_S+(1-t)E_S \ge 0$
as required.

\end{proof}

\section{Cone theorem}\label{sec: cone}

The goal of this section is to prove the cone theorem, Theorem \ref{thm: cone intro}, for algebraically integrable foliated structures. We first state a weaker version of Theorem \ref{thm: cone intro} which is more convenient for us to apply induction on.

\begin{thm}\label{thm: cone}
      Let $(X,\Ff,B,\Mm,t)/U$ be an algebraically integrable adjoint foliated structure. For simplicity, in the following, we denote $K:=K_{(X,\Ff,B,\Mm,t)}$ and $\Nlc:=\Nlc(X,\Ff,B,\Mm,t)$. 
      
      Let $\{R_j\}_{j\in\Lambda}$ be the set of $K$-negative extremal rays in $\overline{NE}(X/U)$ that are not contained in $\overline{NE}(X/U)_{\Nlc}$.
      Then
    $$\overline{NE}(X/U)=\overline{NE}(X/U)_{K\geq 0}+\overline{NE}(X/U)_{\Nlc}+\sum_{j\in\Lambda}R_j$$
    and each $R_j$ is spanned by a rational curve $C_j$ such that
    $$0<-K\cdot C_j\leq 2\dim X.$$
\end{thm}
Compared with Theorem \ref{thm: cone intro}, we do not require that the curves $C_j$ are rational.

\subsection{Length of exposed rays}

Assume Theorem \ref{thm: cone} in low dimensions. In this subsection, we prove that each $K$-negative exposed ray is spanned by a rational curve with bounded length.

\begin{prop}[Length of exposed rays, non-big case]\label{prop: cone non-big case}
Let $(X,\Ff,B,\Mm,t)/U$ be an algebraically integrable adjoint foliated structure such that $t<1$. For simplicity, in the following, we denote $K:=K_{(X,\Ff,B,\Mm,t)}$. Let $R$ be a $K$-negative exposed ray in $\overline{NE}(X/U)$ such that the supporting function $H_R$ of $R$ is not big$/U$. Then $R$ is spanned by a rational curve $C$ such that  $$0<-K\cdot C\leq 2\dim X.$$
\end{prop}
\begin{proof}
Let $d:=\dim X$. Possibly passing to a Stein factorization, we may assume that the induced morphism $\pi: X\rightarrow U$ is a contraction. By \cite[Lemma 8.4.1]{CHLX23} we may write $H_R=K+A$ for some ample$/U$ $\Rr$-divisor $A$. Let $F$ be a general fiber of $\pi$, $H_F:=H_R|_F$, and $A_F:=A|_F$. Then $H_F$ is nef but not big, and $A_F$ is ample. 

Let $q:=\dim F$. Then there exists an integer $1\leq k\leq q-1$ such that
$$H_F^k\cdot A_F^{q-k}>H_F^{k+1}\cdot A_F^{q-k-1}=0.$$
Let $D_i:=H_R$ if $1\leq i\leq k+1$, and let $D_i:=A$ if $k+2\leq i\leq q$. For simplicity, we let $J$ be the $\Rr$-$1$-cycle $(D_2|_F)\cdot\dots\cdot(D_q|_F)$. Then
$$(D_1|_F)\cdot J=H_F^{k+1}\cdot A_F^{q-k-1}=0$$
and
$$s:=-K|_F\cdot J=(A_F-H_F)\cdot H_F^{k}\cdot A_F^{q-k-1}>0.$$
Since $B+\Mm_X$ is pseudo-effective$/U$,
$$-(tK_{\Ff}+(1-t)K_X)|_F\cdot J\geq (A_F-H_F)\cdot H_F^{k}\cdot A_F^{q-k-1}>0.$$
Let $M:=H_R+A$. Then $M$ is ample$/U$. Since $H_F^{k+1}\cdot A_F^{q-k-1}=0$,
$$M|_F\cdot J=-K|_F\cdot J.$$
Since $H_R$ is the supporting function of $R$,
$$M\cdot R=-K\cdot R.$$
There are three possibilities:

\medskip

\noindent\textbf{Case 1}. $-K_X|_F\cdot J\leq 0$. In this case, $-K_{\Ff}|_F\cdot J>0$. By \cite[Theorem 8.1.1]{CHLX23}, for any general closed point $x\in X$, there exists a rational curve $C_x$, such that $x\in C_x$, $\pi(C_x)$ is a closed point, $H_R\cdot C_x=D_1\cdot C_x=0$, and
$$M\cdot C_x\leq 2d\frac{M|_F\cdot J}{-K_{\Ff}|_F\cdot J}=2d\frac{-K|_F\cdot J}{-K_{\Ff}|_F\cdot J}\leq 2d\frac{-tK_{\Ff}|_F\cdot J-(1-t)K_X|_F\cdot J}{-K_{\Ff}|_F\cdot J}\leq 2dt\leq 2d.$$
Since $H_R\cdot C_x=0$ and $\pi(C_x)$ is a closed point, $[C_x]\in R$, so $M\cdot C_x=-K\cdot C_x>0$. Thus
$$0<-K\cdot C_x\leq 2d$$
and we may take $C=C_x$.

\medskip

\noindent\textbf{Case 2}. $-K_{\Ff}|_F\cdot J\leq 0$. In this case, $-K_X|_F\cdot J>0$. By \cite[Theorem 8.1.1]{CHLX23}, for any general closed point $x\in X$, there exists a rational curve $C_x'$, such that $x\in C_x'$, $\pi(C_x')$ is a closed point, $H_R\cdot C_x'=D_1\cdot C_x'=0$, and 
$$M\cdot C_x'\leq 2d\frac{M|_F\cdot J}{-K_{X}|_F\cdot J}=2d\frac{-K|_F\cdot J}{-K_{X}|_F\cdot J}\leq 2d\frac{-tK_{\Ff}|_F\cdot J-(1-t)K_X|_F\cdot J}{-K_{X}|_F\cdot J}\leq 2d(1-t)\leq 2d.$$
Since $H_R\cdot C_x'=0$ and $\pi(C_x')$ is a closed point, $[C_x']\in R$, so $M\cdot C_x'=-K\cdot C_x'>0$. Thus
$$0<-K\cdot C_x'\leq 2d$$
and we may take $C=C_x'$.

\medskip

\noindent\textbf{Case 3}. $-K_{\Ff}|_F\cdot J<0$ and $-K_{X}|_F\cdot J<0$. By \cite[Theorem 8.1.1]{CHLX23}, for any general closed point $x\in X$, there are two rational curves (possibly identical) $C_x$, $C_x'$, such that $x\in C_x$, $x\in C_x'$, $\pi(C_x)$, $\pi(C_x')$ are closed points, $H_R\cdot C_x=D_1\cdot C_x=0$, $H_R\cdot C_x'=D_1\cdot C_x'=0$,
$$M\cdot C_x\leq 2d\frac{M|_F\cdot J}{-K_{\Ff}|_F\cdot J}=2d\frac{-K|_F\cdot J}{-K_{\Ff}|_F\cdot J}\leq 2d\frac{-tK_{\Ff}|_F\cdot J-(1-t)K_X|_F\cdot J}{-K_{\Ff}|_F\cdot J},$$
and
$$M\cdot C_x'\leq 2d\frac{M|_F\cdot J}{-K_{X}|_F\cdot J}=2d\frac{-K|_F\cdot J}{-K_{X}|_F\cdot J}\leq 2d\frac{-tK_{\Ff}|_F\cdot J-(1-t)K_X|_F\cdot J}{-K_{X}|_F\cdot J}.$$
We let $C:=C_x$ if $-K_X|_F\cdot J\leq -K_{\Ff}|_F\cdot J$, and let $C:=C_x'$ if $-K_X|_F\cdot J>-K_{\Ff}|_F\cdot J$. By our construction, 
$$M\cdot C\leq 2d.$$
Since $H_R\cdot C=0$ and $\pi(C)$ is a closed point, $[C]\in R$, so $M\cdot C=-K\cdot C>0$. Thus
$$0<-K\cdot C\leq 2d$$
and we are done.
\end{proof}

\begin{prop}[Length of exposed rays, big case]\label{prop: cone big case}
Let $d$ be a positive integer. Assume Theorem \ref{thm: cone} in dimension $\leq d-1$.

Let $(X,\Ff,B,\Mm,t)/U$ be an algebraically integrable adjoint foliated structure of dimension $d$ such that $t<1$. For simplicity, in the following, we denote $K:=K_{(X,\Ff,B,\Mm,t)}$ and $\Nlc:=\Nlc(X,\Ff,B,\Mm,t)$. Let $R$ be a $K$-negative exposed ray in $\overline{NE}(X/U)$ that is not contained in $\overline{NE}(X/U)_{\Nlc}$, such that such that the supporting function $H_R$ of $R$ is big$/U$. Then $R$ is spanned by a rational curve $C$ such that
$$0<-K\cdot C\leq 2(d-1).$$
\end{prop}
\begin{proof}
\noindent\textbf{Step 1}. In this step we construct a set $\Ii=\{(W,\lambda)\}$ where $W$ are subvarieties of $X$ and $\lambda$ is a non-negative real number, and find a minimal element of $\Ii$.

By \cite[Lemma 8.4.1]{CHLX23} we have $H_R=K+A$ for some ample$/U$ $\Rr$-divisor $A$. In particular, $H_R$ is nef, $H_R\cdot R=0$, and $H_R\cdot R'>0$ for any $R'\in\overline{NE}(X/U)\backslash R$. Since $H_R$ is big$/U$, $H_R\sim_{\mathbb R,U}A'+P$ for some ample$/U$ $\Rr$-divisor $A'$ and $\Rr$-divisor $P\geq 0$. In particular, $P$ is $\Rr$-Cartier and $P\cdot R<0$. Let $\Ii$ be the set of all $(W,\lambda)$ such that
\begin{enumerate}
    \item $\lambda$ is a non-negative real number,
    \item $W$ is an nklt center of $(X,\Ff,B+\lambda P,\Mm,t)$ with normalization $W^\nu$, and
    \item $R$ is contained in the image of $\overline{NE}(W^\nu/U)\rightarrow\overline{NE}(X/U)$ induced by the composition $W^\nu\rightarrow W\rightarrow X$, where $W^\nu$ is the normalization of $W$. In particular, $(X,\Ff,B,\Mm,t)$ is lc near the generic point of $W$. 
\end{enumerate}
\begin{claim}\label{claim: existence of minimal (W,lambda)}
There exists $(W_0,\lambda_0)\in\Ii$, such that for any $(W,\lambda)\in\Ii$, one of the following holds.
\begin{itemize}
\item $\lambda>\lambda_0$.
\item $\lambda=\lambda_0$ and $W\not\subset W_0$.
\item $(W,\lambda)=(W_0,\lambda_0)$.
\end{itemize}
In particular, $(X,\Ff,B+\lambda_0P,\Mm,t)$ is lc near the generic point of $W_0$.
\end{claim}
\begin{proof}
    First we show that $\Ii\not=\emptyset$. Let $S$ be the normalization of $\Supp P$, then $R$ is contained in the image of $\overline{NE}(S/U)\rightarrow\overline{NE}(X/U)$ induced by the natural inclusion $\iota:S\rightarrow \Supp P\rightarrow X.$ Then there exists a component $T$ of $S$ such that $R$ is contained in the image of $\overline{NE}(T/U)\rightarrow\overline{NE}(X/U)$. Hence $(\iota(T),1)\in\Ii$.

    Next, suppose that there exists $(W,0)\in\Ii$ for some $W$. The we may choose a subvariety $W_0\subset W$ such that $(W_0,0)\in\Ii$ and
    $$\dim W_0=\min\{\dim W'\mid (W',0)\in\Ii, W'\subset W\}.$$
    This is because there are only finitely many choices on the dimension. Then $(W_0,\lambda_0:=0)$ satisfies our requirements. In the following, we may assume that $\lambda\not=0$ for any $(W,\lambda)\in\Ii$.

    Therefore, we may assume that $\lambda\not=0$ for any $(W,\lambda)\in\Ii$. Let $h: X'\rightarrow X$ be a foliated log resolution of $(X,\Ff,\Supp B\cup\Supp P,\Mm)$. Then there exists a toroidal morphism $f': (X',\Sigma_X)\rightarrow (Z,\Sigma_Z)$ such that $h^{-1}_\ast (\Supp B\cup\Supp P)\cup\Supp\Exc(h)$ is contained in $\Sigma_X$. Let $\Ff':=h^{-1}\Ff$. For any $(W,\lambda)\in\Ii$, let $$\lambda_W:=\sup\{s\geq 0\mid (X,\Ff,B+sP,\Mm,t)\text{ is lc near the generic point of }W\},$$
    then 
\begin{align*}
\lambda_W:=\sup\left\{s\geq 0\Biggm|
    \begin{array}{r@{}l}
        \mult_D(tK_{\Ff'}+(1-t)K_{X'}-h^\ast (K+sP+\Mm_X))\\ \geq -t\epsilon_{\Ff}(D)-(1-t)\\
        \text{for any prime divisor }D\subset\Sigma_{X}\text{ such that }W\subset h(D)
    \end{array}\right\}.
    \end{align*}
    Since there are only finitely many prime divisor contained in $\Sigma_X$, there are only finitely many values of $\lambda_W$, so
    $$\lambda_W=\max\{s\geq 0\mid (X,\Ff,B+sP,\Mm,t)\text{ is lc near the generic point of }W\},$$
    and there exists an lc center $W'$ of $(X,\Ff,B+\lambda_WP,\Mm,t)$ such that $W\subset W'$. In particular, there exists $(W_1,\lambda_1)\in\Ii$ such that 
    $$\lambda_{W_1}=\min\{\lambda_W\mid (W,\lambda)\in\Ii\},$$
and an lc center $W_1'$ of $(X,\Ff,B+\lambda_{W_1}P,\Mm,t)$ such that $W_1\subset W_1'$. Let $\Ii'$ be the set of subvarieties $V\subset X$ such that
\begin{itemize}
    \item $V$ is an nklt center of $(X,\Ff,B+\lambda_{W_1}P,\Mm,t)$,
    \item $V\subset W_1'$, and
    \item  $R$ is contained in the image $\overline{NE}(V^\nu/U)\rightarrow\overline{NE}(X/U)$ induced by the natural inclusion $V^\nu\rightarrow V\rightarrow X$, where $V^\nu$ is the normalization of $V$.
\end{itemize}
Since $W_1'\in\Ii'$, $\Ii'\not=\emptyset$. Now let $W_0\in\Ii'$ be a subvariety of $X$ such that
$$\dim W_0=\min\{\dim V\mid V\in\Ii'\}.$$
We show that $(W_0,\lambda_0:=\lambda_{W_1})$ satisfies our requirements. Since $W_0\in\Ii'$, $W_0$ is an nklt center of $(X,\Ff,B+\lambda_{0}P,\Mm,t)$ and $R$ is contained in the image $\overline{NE}(W_0^\nu/U)\rightarrow\overline{NE}(X/U)$ induced by the composition $W_0^\nu\rightarrow W_0\rightarrow X$, where $W_0^\nu$ is the normalization of $W_0$. Thus $(W_0,\lambda_0)\in\Ii$. By our assumption,
$$\lambda_{W_0}\leq\lambda_0=\lambda_{W_1}=\min\{\lambda_W\mid (W,\lambda)\in\Ii\}\leq\lambda_{W_0},$$
so $\lambda_0=\lambda_{W_0}$, and  $(X,\Ff,B+\lambda_{0}P,\Mm,t)$ is lc near the generic point of $W_0$. Moreover, for any $(W,\lambda)\in\Ii$, $\lambda\geq\lambda_W\geq\lambda_0$, and if $\lambda=\lambda_0$ and  $W\subset W_0$, then $W\in\Ii'$, so by the minimality of the dimension of $W_0$, $W=W_0$. 

The proof of the claim is concluded.
\end{proof}

\noindent\textbf{Step 2}. In this step we construct a $\Qq$-factorial qdlt modification and a ray $R'$.

Let $\bar B:=B+\lambda_0P$ and let $E$ be an lc place of $(X,\Ff,\bar B,\Mm,t)$ such that $\Center_XE=W_0$. By Proposition \ref{prop: qdlt model extract certain divisors}, there exists a $\Qq$-factorial qdlt modification 
$$h: (X',\Ff',\bar B',\Mm,t)\rightarrow (X,\Ff,\bar B,\Mm,t)$$
such that $E$ is on $X'$. Let $K':=(X',\Ff',\bar B',\Mm,t)$. Then
$$K'+F=h^\ast K$$
for some $F\geq 0$. Let $V:=h(\Supp F)$, then $V\subset\Nlc(X,\Ff,\bar B,\Mm,t)$ is a reduced subscheme of $X$. By \cite[Lemma 8.2.3]{CHLX23}, there exists an extremal ray $R'$ in $X'$ such that $h(R')=R$. Then there exist $C_i'\in NE(X'/U)$ such that $R'=[\lim_{i\rightarrow+\infty} C_i']$. We let $C_i:=h(C_i')$, then $R=[\lim_{i\rightarrow+\infty} C_{i}]$. By the projection formula,
$$\lim (K'+F)\cdot C_{i}'=\lim K\cdot C_i,$$
so
$$(K+F)\cdot R'=K\cdot R<0.$$
Thus $R'$ is a $(K+F)$-negative extremal ray. 

If $F\cdot R'<0$, then $R'$ is contained in the image of $\overline{NE}(\Supp F/U)\rightarrow\overline{NE}(X'/U)$. Then $R=h(R')$ is contained in the image of $\overline{NE}(V/U)\rightarrow\overline{NE}(X/U)$. Thus there exists an irreducible component $V_0$ of $V$ such that $R$ is contained in the image of $\overline{NE}(V_0/U)\rightarrow\overline{NE}(X/U)$.  Since $R$ is not contained in $\overline{NE}(X/U)_{\Nlc}$, $V_0$ is not an lc center of $(X,\Ff,B,\Mm,t)$. Since $V_0\subset V=h(F)\subset\Nlc(X,\Ff,\bar B,\Mm,t)$ and $\bar B=B+\lambda_0P$, there exists a real number $0<\lambda_1<\lambda_0$ such that $V_0$ is an lc center of $\Nlc(X,\Ff,B+\lambda_1P,\Mm,t)$. This contradicts the minimality of $(W_0,\lambda_0)$ as $(V_0,\lambda_1)\in\Ii$ and $\lambda_1<\lambda_0$. Therefore, $F\cdot R'\geq 0$, so $R'$ is a  $K'$-negative extremal ray.

\medskip

\noindent\textbf{Step 3}. In this step we prove the proposition under the additional condition that $X$ is $\Qq$-factorial. 

Assume that $X$ is $\Qq$-factorial. By \cite[Lemma 3.6.2]{BCHM10}, $\Exc(h)$ is a divisor, so $h^{-1}(W_0)$ is a divisor. Since $R$ is contained in the image of $\overline{NE}(W_0/U)\rightarrow\overline{NE}(X/U)$ and $h(R')=R$, there exists a divisor $E_0$ on $X'$ such that $R'$ is contained in the image of $\overline{NE}(E_0/U)\rightarrow\overline{NE}(X'/U)$. Since $h$ is a $\Qq$-factorial qdlt modification of $(X,\Ff,\bar B,\Mm,t)$, $E_0$ is an lc place of $(X,\Ff,\bar B,\Mm,t)$ and an lc place of $(X',\Ff',\bar B',\Mm,t)$. 

Let $T$ be the normalization of $E_0$, $\Ff_{T}:=\Ff_Y|_{T}$ be the restricted foliation, $\Mm^{T}:=\Mm|_{T}$, and
$$\bar K_T:=K_{(T,\Ff_T,B_T,\Mm^T,t)}:=K'|_T.$$
Since $R'$ is contained in the image of $\overline{NE}(E_0/U)\rightarrow\overline{NE}(X'/U)$,  $R'$ is contained in the image of 
$$\iota: \overline{NE}(T/U)\rightarrow \overline{NE}(E_0/U)\rightarrow\overline{NE}(X'/U).$$ 
By \cite[Lemma 8.2.3]{CHLX23}, there exists a extremal ray $R_T\in\overline{NE}(T/U)$ such that $\iota(R_T)=R'$. Then $R_T$ is a $\bar K_T$-negative extremal ray$/U$. By Theorem \ref{thm: afs adj}, $(T,\Ff_T,B_T,\Mm^T,t)$ is lc. By Theorem \ref{thm: cone} in dimension $d-1$, $R_T$ is spanned by a rational curve $C_T$ such that 
$$0<-\bar K_T\cdot C_T\leq 2(d-1).$$
Let $C'$ be the image of $C_T$ in $Y$, then $C'$ spans $R'$, and
$$0<-\bar K_T\cdot C_T=-K'\cdot C'\leq 2(d-1),$$
Let $C:=h(C')$. By \textbf{Step 2}, $F\cdot C'\geq 0$, so
$$2(d-1)\geq -K'\cdot C'\geq -(K'+F)\cdot C'=-K\cdot C>0.$$
We are done for the case when $X$ is $\Qq$-factorial.

\medskip

\noindent\textbf{Step 4}. In this step we complete the proof of the proposition. Since $X'$ is $\Qq$-factorial and $R'$ is a $K'$-negative extremal ray, by the $\Qq$-factorial case proved in \textbf{Step 3}, $R'$ is spanned by a rational curve $C'$ such that
$$0<-K'\cdot C'\leq 2(d-1).$$
Let $C:=h(C')$. Since $F\cdot C'\geq 0$,
$$2(d-1)\geq -K'\cdot C'\geq -(K'+F)\cdot C'=-K\cdot C>0$$
and we are done.
\end{proof}

\subsection{Finiteness of rays}

Assume Theorem \ref{thm: cone} in low dimensions. In this subsection, we prove that there are finitely many $(K+A)$-negative extremal rays that are not contained in $\overline{NE}(X/U)_{\Nlc}$.

\begin{prop}\label{prop: cone finiteness rays}
Let $d$ be a positive integer. Assume that Theorem \ref{thm: cone} holds in dimension $\leq d-1$. Let $(X,\Ff,B,\Mm,t)/U$ be an algebraically integrable adjoint foliated structure of dimension $d$ such that $t<1$. For simplicity, in the following, we denote $K:=K_{(X,\Ff,B,\Mm,t)}$ and $\Nlc:=\Nlc(X,\Ff,B,\Mm,t)$. Let $A$ be an ample$/U$ $\Rr$-divisor on $X$. Then there are finitely many $(K+A)$-negative 
extremal rays$/U$ that are not contained in $\overline{NE}(X/U)_{\Nlc}$.
\end{prop}
\begin{proof}
The proof is almost verbatim to the proof of \cite[Proposition 8.4.3]{CHLX23} with only minor differences. Let $\rho:=\rho(X/U)$ and let $A_1,\dots,A_{\rho-1}$ be ample$/U$ Cartier divisors on $X$, such that $K,A_1,\dots,A_{\rho-1}$ form a basis of $N^1_{\mathbb R}(X/U)$. Let $0<\epsilon\ll 1$ be a rational number such that $A-\epsilon\sum_{i=1}^{\rho-1}A_i$ is ample$/U$. Then we only need to show that there are finitely many $(K+\epsilon\sum_{i=1}^{\rho-1}A_i)$-negative extremal rays$/U$ that are not contained in $\overline{NE}(X/U)_{\Nlc}$. Possibly replacing $A$, we may assume that $A=\epsilon\sum_{i=1}^{\rho-1}A_i$. We lose the condition that  $A-\epsilon\sum_{i=1}^{\rho-1}A_i$ is ample$/U$ but we do not need it anymore.

Suppose that the proposition does not hold. Then there exist an infinite set $\Lambda$ and an infinite set $\{R_j\}_{j\in\Lambda}$ of $(K+A)$-negative extremal rays$/U$ that are not contained in $\overline{NE}(X/U)_{\Nlc}$. By \cite[Definition-Lemma 3.1.8]{CHLX23} (\cite[Corollary 18.7.1]{Roc97}, \cite[Lemma 6.2]{Spi20}), possibly replacing $\Lambda$ with a smaller infinite subset, we may assume that each $R_j$ is a  $(K+A)$-negative exposed ray$/U$. By Propositions \ref{prop: cone non-big case} and \ref{prop: cone big case}, for any $j\in\Lambda$, there exists a rational curve $C_j$ on $X$  such that $R_j=\mathbb R_+[C_j]$ and 
$$-2d\leq K\cdot C_j<0.$$ 
For each $j\in\Lambda$, by \cite[Lemma 8.4.1]{CHLX23}, there exists an ample$/U$ $\Rr$-divisor $L_j$ and a nef$/U$ $\Rr$-divisor $H_j$, such that
$$H_j=L_j+(K+A)=L_j+\epsilon\sum_{i=1}^{\rho-1}A_i+K$$
and $H_j$ is the supporting function of $R_j$. We have
$$0=H_j\cdot C_j=L_j\cdot C_j+\epsilon\sum_{i=1}^{\rho-1}A_i\cdot C_j+K\cdot C_j\geq-2d+\epsilon\sum_{i=1}^{\rho-1}A_i\cdot C_j.$$
Therefore, $A_i\cdot C_j\leq\frac{2d}{\epsilon}$ for any $i,j$. Since $\Lambda$ is infinite, there are two curves $C_{j_1},C_{j_2}$ such that $A_i\cdot C_{j_1}=A_i\cdot C_{j_2}$, so $C_{j_1}\equiv_U C_{j_2}$. This is not possible.
\end{proof}

\subsection{Proof of the cone theorem}

We first prove the weak version of the cone theorem, Theorem \ref{thm: cone}, by induction on dimension. Then we shall prove the rationality of the extremal faces and conclude the proof of Theorem \ref{thm: cone intro}. Note that the rationality of extremal rays means that their supporting functions are rational, and is not equivalent to that they are spanned by rational curves.

\begin{proof}[Proof of Theorem \ref{thm: cone}]
We apply induction on dimension. It is clear that Theorem \ref{thm: cone} holds in dimension $1$. We let $K:=K_{(X,\Ff,B,\Mm,t)}$ and $\Nlc:=\Nlc(X,\Ff,B,\Mm,t)$. Assume that $\dim X=d\geq 2$ and  Theorem \ref{thm: cone} holds in dimension $\leq d-1$.

First we show that $\overline{NE}(X/U)=V$, where
$$V:=\overline{NE}(X/U)_{K\geq 0}+\overline{NE}(X/U)_{\Nlc}+\sum_{j\in\Lambda}R_j.$$
 By \cite[Definition-Lemma 3.1.8]{CHLX23} (\cite[Corollary 18.7.1]{Roc97}, \cite[Lemma 6.2]{Spi20}), $\overline{NE}(X/U)=\overline{V}$. Suppose that $V\not=\overline{V}$, then there exists an extremal ray $R$ in $\overline{NE}(X/U)$ such that $R\not\in V$. Since $\overline{NE}(X/U)_{K\geq 0}$ and $\overline{NE}(X/U)_{\Nlc}$ are closed, $R\not\in \overline{NE}(X/U)_{K\geq 0}$ and $R\not\in \overline{NE}(X/U)_{\Nlc}$, so $R$ is a $K$-negative extremal ray that is not contained in $\overline{NE}(X/U)_{\Nlc}$. Thus $R=R_j$ for some $j$, a contradiction. Therefore, $\overline{NE}(X/U)=V$.

Next we show that any each $R_j$ is exposed. For any fixed $j$, There exists an ample$/U$ $\Rr$-divisor $A$ such that $R_j$ is a $(K+A)$-negative extremal ray$/U$. Suppose that $R_j$ is not exposed. By \cite[Definition-Lemma 3.1.8]{CHLX23} (\cite[Corollary 18.7.1]{Roc97}, \cite[Lemma 6.2]{Spi20}), $R_j=\lim_{i\rightarrow+\infty}R_{j,i}$ for some exposed rays $R_{j,i}\in\overline{NE}(X/U)$. Since $(K+A)\cdot R_j<0$, possibly passing to a subsequence, we have $(K+A)\cdot R_{j,i}<0$ for any $i$.  By Proposition \ref{prop: cone finiteness rays}, there are only finitely many $(K+A)$-negative extremal rays that are not contained in $\overline{NE}(X/U)_{\Nlc}$, so for any $i\gg 0$, $R_{j,i}$ is contained in $\overline{NE}(X/U)_{\Nlc}$. Since $\overline{NE}(X/U)_{\Nlc}$ is a closed sub-cone of $\overline{NE}(X/U)$, $R_j$ is contained in $\overline{NE}(X/U)_{\Nlc}$, a contradiction.

By Propositions \ref{prop: cone non-big case} and \ref{prop: cone big case}, for any $j\in\Lambda$, $R_j$ is spanned by a rational curve $C_j$ such that
$$0<-K\cdot C_j\leq 2d.$$
\end{proof}

\begin{proof}[Proof of Theorem \ref{thm: cone intro}]
    (1) follows from Theorem \ref{thm: cone} and the fact that the numerical classes of curves are rational in $\overline{NE}(X/U)$. (2) follows from Theorem \ref{thm: cone}. (3) follows from Proposition \ref{prop: cone finiteness rays} and that 
    $$\Lambda=\cup_{n=1}^{+\infty}\Lambda_{\frac{1}{n}A}$$
    for any ample$/U$ $\Rr$-divisor $A$. We left to prove (4).

     For any $K$-negative extremal face $F$ in $\overline{NE}(X/U)$ that is relatively ample at infinity with respect to $(X,\Ff,B,\Mm,t)$, $F$ is also a $(K+A)$-negative extremal face for some ample$/U$ $\Rr$-divisor $A$ on $X$. Let $V:=F^\bot\subset N^1(X/U)$. By (1), $F$ is spanned by a subset of $\{R_j\}_{j\in\Lambda_A}$ and $R_j$ is rational, so $V$ is defined over $\Qq$. We let
$$W_F:=\overline{NE}(X/U)_{K+A\geq 0}+\overline{NE}(X/U)_{\Nlc}+\sum_{j\mid j\in\Lambda_A,R_j\not\subset F}R_j.$$
Then $W_F$ is a closed cone, $\overline{NE}(X/U)=W_F+F$, and $W_F\cap F=\{0\}$. The supporting functions of $F$ are the elements in $V$ that are positive on $W_F\backslash\{0\}$, which is a non-empty open subset of $V$, and hence contains a rational element $H$. In particular, $F=H^\bot\cap \overline{NE}(X/U)$, hence $F$ is rational, and we get (4). This concludes the proof of Theorem \ref{thm: cone intro}.
\end{proof}

The cone theorem immediately implies the following result:

\begin{prop}
    Let $(X,\Ff,B,\Mm,t)/U$ be an lc algebraically integrable adjoint foliated structure, $A$ an ample$/U$ $\Rr$-divisor on $X$, and $D$ a nef$/U$ $\Rr$-divisor on $X$. Let $K:=K_{(X,\Ff,B,\Mm,t)}$. Assume that $K+A$ is not nef$/U$. Let
    $$\lambda:=\sup\{s\geq 0\mid D+s(K+A)\text{ is nef}/U\}.$$
    Then there exists a $K$-negative extremal ray $R$ in $\overline{NE}(X/U)$ such that
    $$(D+\lambda(K+A))\cdot R=0.$$
\end{prop}
\begin{proof}
If $\lambda=0$ then $D+s(K+A)$ is not nef$/U$ for any $s>0$, so for any $s>0$, there exists a  $(D+s(K+A))$-negative extremal ray $R_s$. Then $R_s$ is $(K+A)$-negative for any $s$. By Theorem \ref{thm: cone intro}, there are only finitely many $(K+A)$-negative extremal ray$/U$, so there exists a $(K+A)$-negative extremal ray $R$ such that $(D+s(K+A))\cdot R<0$ for any $s>0$. Thus $D\cdot R=0$ and the proposition follows. In the following, we may assume that $\lambda\not=0$. 
    Since $A+\frac{1}{2\lambda}D$ is ample$/U$, by Theorem \ref{thm: cone intro}, there are only finitely many $D+2\lambda(K+A)$-negative extremal rays $R_1,\dots,R_l$ in $\overline{NE}(X/U)$ and each $R_i$ is spanned by a rational curve $C_i$. Therefore, we have
    $$\lambda=\min_{1\leq i\leq l}\left\{-\frac{D\cdot R_i}{(K+A)\cdot R_i}\right\}.$$
    and there exists $i$ such that $\lambda=-\frac{D\cdot R_i}{(K+A)\cdot R_i}$. We may take $R=R_i$.
\end{proof}

\section{Contraction theorem}\label{sec: contraction}

In this section we prove Theorem \ref{thm: cont intro}.

\subsection{Existence of flipping and divisorial contractions}

\begin{defn}[NQC]
    Let $X\rightarrow U$ be a projective morphism from a normal quasi-projective variety to a variety. Let $D$ be a nef $\Rr$-Cartier $\Rr$-divisor on $X$ and $\Mm$ a nef $\bb$-divisor on $X$.

    We say that $D$ is NQC$/U$ if $D=\sum d_iD_i$, where each $d_i\geq 0$ and each $D_i$ is a nef$/U$ Cartier divisor. We say that $\Mm$ is NQC$/U$ if $\Mm=\sum \mu_i\Mm_i$, where each $\mu_i\geq 0$ and each $\Mm_i$ is a nef$/U$ Cartier $\bb$-divisor.
\end{defn}

\begin{lem}\label{lem: nqc of K+A}
Let $(X,\Ff,B,\Mm,t)/U$ be an lc algebraically integrable adjoint foliated structure and $A$ an ample$/U$ $\Rr$-divisor on $X$. Let $K:=K_{(X,\Ff,B,\Mm,t)}$. Let $D$ be a nef$/U$ $\Rr$-divisor on $X$ such that $D-K$ is ample$/U$. Then $D$ is NQC$/U$.
\end{lem}
\begin{proof}
We write $D=\sum_{i=1}^cr_iD_i$, where $r_1,\dots,r_c$ are linearly independent over $\mathbb Q$ and each $D_i$ is a Cartier divisor. We define $D(\bm{v}):=\sum_{i=1}^cv_iD_i$ for any $\bm{v}=(v_1,\dots,v_c)\in\mathbb R^c$. By \cite[Lemma 5.3]{HLS19}, $D(\bm{v})$ is $\Rr$-Cartier for any $\bm{v}\in\mathbb R^c$. 

Let $L:=D-K$. Since ample$/U$ is an open condition, there exists an open set $V\ni\bm{r}$ in $\mathbb R^c$, such that 
$\frac{1}{2}L+D(\bm{v})-D$ is ample$/U$ for any $\bm{v}\in V$.

By Theorem \ref{thm: cone intro}, there exist finitely many $\left(K+\frac{1}{2}L\right)$-negative extremal rays$/U$ $R_1,\dots,R_l$, and $R_j=\mathbb R_+[C_j]$ for some curve $C_j$. Since $D$ is nef, $D\cdot C_j\geq 0$ for each $j$. Thus possibly shrinking $V$, we may assume that for any $\bm{v}\in V$, we have that $D(\bm{v})\cdot C_j>0$ for any $j$ such that $D\cdot C_j>0$. Since $r_1,\dots,r_c$ are linearly independent over $\mathbb Q$, for any $j$ such that $D\cdot C_j=0$, we have $D(\bm{v})\cdot C_j=0$ for any $\bm{v}\in\mathbb R^c$. Therefore,  $D(\bm{v})\cdot C_j\geq 0$ for any $j$ and any $\bm{v}\in V$.

By Theorem \ref{thm: cone intro}, for any curve $C$ on $X$, we may write
$[C]=\eta+\sum_{i=1}^l a_i[C_i]$
where $a_1,\dots,a_l\geq 0$ and $\eta\in\overline{NE}(X/U)_{K+\frac{1}{2}L\geq 0}$. For any $\bm{v}\in V$, since
$$D(\bm{v})\cdot \eta=\left(K+\frac{1}{2}L\right)\cdot \eta+\left(\frac{1}{2}L+D(\bm{v})-D\right)\cdot \eta\geq 0,$$
$D(\bm{v})\cdot C\geq 0$. Therefore, $D(\bm{v})$ is nef$/U$ for any $\bm{v}\in V$. We let $\bm{v}_1,\dots,\bm{v}_{c+1}\in V\cap\mathbb Q^c$ be rational points such that $\bm{r}$ is in the interior of the convex hull of $\bm{v}_1,\dots,\bm{v}_{c+1}$. Then there exist positive real numbers $a_1,\dots,a_{c+1}$ such that $\sum_{i=1}^{c+1}a_i=1$ and $\sum_{i=1}^{c+1}a_i\bm{v}_i=\bm{r}$. Since $D(\bm{v}_i)$ is a nef$/U$ $\Qq$-divisor for each $i$ and $D=\sum_{i=1}^{c+1}a_iD(\bm{v}_i)$, $D$ is NQC$/U$.
\end{proof}

\begin{thm}[Existence of flipping and divisorial contractions]\label{thm: eo flipping and divisorial contraction}
    Let $(X,\Ff,B,\Mm,t)/U$ be an lc algebraically integrable adjoint foliated structure. Let $K:=K_{(X,\Ff,B,\Mm,t)}$. Let $A$ be an ample$/U$ $\Rr$-divisor on $X$. Assume that $X$ is potentially klt, and $K+A$ is nef and big$/U$. Then $K+A$ is semi-ample$/U$. Moreover, the following hold.
    \begin{enumerate}
        \item If $K+A$ is Cartier, then $\mathcal{O}_X(n(K+A))$ is globally generated$/U$ for any integer $n\gg 0$.
        \item Let $c: X\rightarrow T$ be the contraction$/U$ associated to $K+A$. Then $T$ is potentially klt.
        \item Suppose that $K+A$ is the supporting function of a $K$-negative extremal ray $R$.
\begin{enumerate}
    \item For any integral curve $C$ such that $\pi(C)$ is a point, where $\pi: X\rightarrow U$ is the associated morphism, $c(C)$ is a point if and only if $[C]\in R$.
    \item Let $L$ be a line bundle on $X$ such that $L\cdot R=0$. Then there exists a line bundle $L_T$ on $T$ such that $L\cong c^\ast L_T$.
    \end{enumerate}
    \end{enumerate}
\end{thm}
\begin{proof}

\noindent\textbf{Step 1}. We reduce to the case when $(X,\Ff,B,\Mm,t)$ is $\Qq$-factorial klt and $\Supp B$ contains the divisorial part of $\Bb_+((K+A)/U)$.

Let $(X,\Delta)$ be a klt pair. Possibly replacing $(X,\Ff,B,\Mm,t)$ with $(X,\Ff,(1-\delta)B+\delta\Delta,(1-\delta)\Mm,(1-\delta)t)$ and replacing $A$ with $A+\delta K-\delta(K_X+\Delta)$ for some $0<\delta\ll 1$, we may assume that $(X,\Ff,B,\Mm,t)/U$ is klt and $t<1$. Since $X$ is potentially klt, there exists a small $\Qq$-factorialization $h: (X',\Ff',B',\Mm,t)\rightarrow (X,\Ff,B,\Mm,t)$. 
Let $E\geq 0$ be an $h$-anti-ample $\Rr$-divisor on $X'$. Then $h^\ast A-eE$ is ample for any $0<e\ll 1$. By Lemma \ref{lem: perturb qdlt afs}, $(X,\Ff,B+eE,\Mm,t)$ is klt for any $0<e\ll 1$. Since $K_{(X',\Ff',B'+eE,\Mm,t)}+h^\ast A-eE$ and $K+A$ are crepant, we may replace $(X,\Ff,B,\Mm,t)$ with $(X',\Ff',B'+eE,\Mm,t)$ and $A$ with $h^\ast A-eE$, and assume that $(X,\Ff,B,\Mm,t)$ is $\Qq$-factorial klt.

Since $K+A$ is big and nef$/U$, the divisorial part of $\Bb_+((K+A)/U)$ consists of finitely many components. We let $S$ to be the support of the divisorial part of  $\Bb_+((K+A)/U)$. Then possibly replacing $B$ with $B+eS$ and $A$ with $A-eS$ for $0<e\ll 1$, we may assume that $\Supp B$ contains the divisorial part of $\Bb_+((K+A)/U)$.

\medskip

\noindent\textbf{Step 2}. In this step we take a foliated log resolution of $(X,\Ff,B,\Mm)$ and we construct auxiliary divisors satisfying certain properties.

Let $g: W\rightarrow X$ be a foliated log resolution of $(X,\Ff,B,\Mm,t)$ associated with an equidimensional toroidal contraction $f: (W,\Sigma_W)\rightarrow (Z,\Sigma_Z)$. 
Let $\Ff_W:=g^{-1}\Ff$, $S_W:=g^{-1}_\ast S$, and let $G_W$ be the vertical$/Z$ part of $\Sigma_W$.
Since any component of $S$ is a component of $B$, $S_W^{\inv}\subset\Sigma_W$.

Let $E_{1},\dots,E_{m}$ be $g$-exceptional prime divisors. Let $e_i:=-a(E_i,X,\Ff,B,\Mm,t)$, $A_W:=g^\ast A$, $\Ff_W:=g^{-1}\Ff$, and $B_W:=g^{-1}_\ast B$. 
Then we have
$$tK_{\Ff_W}+(1-t)K_W+B_W+\sum e_iE_i+\Mm_W=g^\ast K.$$
Since $(X,\Ff,B,\Mm,t)$ is klt, there exists a positive real number $s$, such that for any irreducible component $D$ of $B_W+e_iE_i$, we have
$$\mult_D(B_W+e_iE_i)<t\epsilon_{\Ff}(D)+(1-t)-s.$$
In particular, $e_i<t\epsilon_{\Ff}(E_i)+(1-t)-s$ for any $i$. Let
$$\widehat B_W:=g^{-1}_\ast (B^{\ninv}\vee S^{\ninv}+B^{\inv}\vee (1-t)S^{\inv}).$$
Then
$$\Supp(\widehat B_W-B_W)=S_W.$$

Since $X$ is $\Qq$-factorial, there exists a $g$-exceptional $\Qq$-divisor $Q\geq 0$ that is anti-ample$/X$ and $\lfloor Q\rfloor=0$. In particular, there exists a real number $0<\tau<\frac{s(1-t)}{2}$, such that $A_W-\tau Q$ is ample$/U$. Let $\Aa:=\overline{A_W-\tau Q}$.

\medskip

\noindent\textbf{Step 3}. In this step we run a foliated MMP $\phi: W\dashrightarrow Y$.

Let $H:=K+A$ and $\Hh:=\overline{H}$. By Lemma \ref{lem: nqc of K+A}, $\Hh$ is NQC$/U$. We may write $\Hh=\sum a_i\Hh_i$ such that each $a_i>0$ and each $\Hh_i$ is nef$/U$ and $\bb$-Cartier. Let $\epsilon_0:=\min\{a_i,1\}$.

Since $$\left(W,\Ff_W,\widehat B_W+\sum E_i,\Mm+\Aa\right)$$ is foliated log smooth, 
$$\left(W,\Ff_W,\widehat{B}_W^{\ninv}+\sum \epsilon_{\Ff}(E_i)E_i,\Mm+\Aa,G_W\right)/Z$$
is $\Qq$-factorial ACSS. Let $l\gg\frac{2\dim X}{\epsilon_0}$ be an integer. Since $\Hh_W$ is big and nef, by \cite[Theorem 16.1.4]{CHLX23}, we may run a 
$$\left(K_{\Ff_W}+\widehat{B}_W^{\ninv}+\sum \epsilon_{\Ff}(E_i)E_i+\Mm_W+\Aa_W+l\Hh_W\right)\text{-MMP}/U$$ with scaling of an ample divisor which terminates with a good minimal model$/U$. Moreover, by \cite[Lemma B.6]{LMX24b}, this MMP is $\Hh_W$-trivial.
  
Let $\phi: W\dashrightarrow Y$ be this MMP, $\Ff_Y:=\phi_\ast \Ff_W$, and $B_Y,\widehat{B}_Y,E_{i,Y},A_Y,Q_Y,G_Y,S_Y$ the images of $B_W,\widehat{B}_W,E_{i},A_W,Q_W,G_W,S_W$ on $Y$ respectively. Then
$$P_Y:=K_{\Ff_Y}+\widehat{B}_Y^{\ninv}+\sum\epsilon_{\Ff}(E_i)E_{i,Y}+\Mm_Y+\Aa_Y+l\Hh_Y$$
is semi-ample$/U$. Moreover, by \cite[Lemma 16.1.1]{CHLX23}, there exists an ample $\Rr$-divisor $A'_Y$ on $Y$ and a nef$/U$ $\bb$-divisor $\Mm'$, such that $\Mm'_Y+A'_Y=\Mm_Y+\Aa_Y$, $\Mm'-\Mm$ is nef$/U$, and $$\left(Y,\Ff_Y,\widehat{B}_Y^{\ninv}+\sum\epsilon_{\Ff}(E_i)E_{i,Y},\Mm'+\Aa'+l\Hh;G_Y\right)/Z$$ 
is $\Qq$-factorial ACSS, where $\Aa':=\overline{A'_Y}$. Then
$$\left(Y,\widehat{B}_Y^{\ninv}+\sum\epsilon_{\Ff}(E_i)E_{i,Y}+G_Y,\Mm'+\Aa'+l\Hh\right)$$
is $\Qq$-factorial qdlt, so
$$\left(Y,\Ff_Y,\widehat{B}_Y^{\ninv}+\sum\epsilon_{\Ff}(E_i)E_{i,Y}+(1-t)G_Y,\Mm'+\Aa'+l\Hh,t\right)$$
is lc.

\medskip

\noindent\textbf{Step 4}. In this step we show that there exists a natural lc generalized pair structure on $Y$, and run a generalized pair MMP  $\psi: Y\dashrightarrow V$.

Let $\Pp:=\overline{P_Y}$. Let $\tau_i:=t\epsilon_{\Ff}(E_i)+(1-t)$ for each $i$. Then
\begin{align*}
    &tK_{\Ff_Y}+(1-t)K_Y+\widehat{B}_Y+\sum \tau_iE_i+\Mm_Y+\Aa_Y+l\Hh_Y\\
        =&tK_{\Ff_Y}+(1-t)K_Y+\widehat{B}_Y+\sum \tau_iE_i+\Mm'_Y+\Aa'_Y+l\Hh_Y\\
=&tP_Y+(1-t)\bigg(K_Y+\widehat{B}_Y^{\ninv}+\frac{1}{1-t}\widehat{B}_Y^{\inv}+\sum E_{i,Y}+\Mm'_Y+\Aa'_Y+l\Hh_Y\bigg).
\end{align*}
Let $\Delta_Y:=\widehat{B}_Y^{\ninv}+\frac{1}{1-t}\widehat{B}_Y^{\inv}+\sum E_{i,Y}$ and $\Nn:=\frac{t}{1-t}\Pp+\Mm'+l\Hh$. Then $(Y,\Delta_Y,\Nn)$ is $\Qq$-factorial lc. Since $\Aa'$ descends to $Y$, we may run a 
$$(K_Y+\Delta_Y+\Nn_Y+\Aa'_Y)\text{-MMP}/U$$
which terminates with a good minimal model$/U$. Moreover, this is also a 
$$\left(tK_{\Ff_Y}+(1-t)K_Y+\widehat{B}_Y+\sum \tau_iE_i+\Mm_Y+\Aa_Y+l\Hh_Y\right)\text{-MMP}/U.$$
By \cite[Lemma B.6]{LMX24b}, this MMP is $\Hh_Y$-trivial. Let $\psi: Y\dashrightarrow V$ be this MMP, $\Ff_V:=\psi_\ast \Ff_Y$, and $B_V,\widehat{B}_V,E_{i,V},A_V,Q_V,G_V,S_V$ the images of $B_Y,\widehat{B}_Y,E_{i,Y},A_Y,Q_Y,G_Y,S_Y$ on $V$ respectively. Since
$$\left(Y,\Ff_Y,\widehat{B}_Y^{\ninv}+\sum\epsilon_{\Ff}(E_i)E_{i,Y}+(1-t)G_Y,\Mm+\Aa+l\Hh,t\right)$$
is lc,
$$\left(Y,\Ff_Y,\widehat{B}_Y+\sum\tau_iE_{i,Y},\Mm+\Aa+l\Hh,t\right)$$
is lc, so
$$\left(V,\Ff_V,\widehat{B}_V+\sum\tau_iE_{i,V},\Mm+\Aa+l\Hh,t\right)$$
is lc.

\medskip

\noindent\textbf{Step 5}. In this step we show that the adjoint foliated structure we have constructed on $V$ is the good minimal model$/U$ of an adjoint foliated structure on $W$.

Let 
$$K^V:=tK_{\Ff_V}+(1-t)K_V+\widehat{B}_V+\sum\tau_iE_{i,V}+\Mm_V+\Aa_V$$
and
$$K^W:=tK_{\Ff_W}+(1-t)K_W+\widehat{B}_W+\sum\tau_iE_{i}+\Mm_W+\Aa_W.$$
Then $(\psi\circ\phi)_\ast K^W=K^V$. Let $p: \widehat X\rightarrow W$ and $q: \widehat X\rightarrow\bar X$ be a resolution of indeterminacy of $\psi\circ\phi$, and write
$$p^\ast K^W=q^\ast K^V+F$$
for some $\Rr$-divisor $F$. Then $F$ is exceptional$/V$. Since $\psi\circ\phi$ is $\Hh_W$-trivial, we have
$$p^\ast (K^W+l\Hh_W)=q^\ast (K^V+l\Hh_V)+F.$$
and $p^\ast \Hh_W=q^\ast \Hh_{V}$. 

Since $S$ is the divisorial part of $\Bb_+(H/U)$, by Lemma \ref{lem: sbl under pullback}, $S_W\cup\cup_{i=1}^mE_i$ is the divisorial part of $\Bb_+(\Hh_W/U)$. Since  $\psi\circ\phi$ is $\Hh_W$-trivial, any divisor contracted by  $\psi\circ\phi$ is either $E_i$ for some $i$, or a component of $S_W$. Thus $\Supp(p_\ast F)\subset S_W\cup\cup_{i=1}^mE_i$, so for any component $D$ of $p_\ast F$, we have
\begin{align*}
  -t\epsilon_{\Ff}(D)-(1-t)&=a(D,W,\widehat B_W+\sum\tau_iE_i,\Mm+\Aa,t)\\
  &=a(D,V,\widehat B_V+\sum\tau_iE_{i,V},\Mm+\Aa,t)-\mult_Dp_\ast F\\
  &\geq -t\epsilon_{\Ff}(D)-(1-t)-\mult_Dp_\ast F.
\end{align*}
So $\mult_Dp_\ast F\geq 0$. Since $K^V+l\Hh_V$ is nef$/U$, $-F$ is nef$/W$. By the negativity lemma, $F\geq 0$. Therefore,
$$\left(V,\Ff_V,\widehat{B}_V+\sum\tau_iE_i,\Mm+\Aa+l\Hh,t\right)/U$$
is a semi-ample model of 
$$\left(W,\Ff_W,\widehat{B}_W+\sum\tau_iE_i,\Mm+\Aa+l\Hh,t\right)/U.$$
\medskip

\noindent\textbf{Step 6}. In this step we show that $K+A$ is semi-ample$/U$.

Since
$$\tau_i-e_i=t\epsilon_{\Ff}(E_i)+(1-t)-e_i>s>\tau,$$
for each $i$, we have
$$\Supp\left(\sum_{i=1}^m(\tau_i-e_i)E_i-\tau Q\right)=\cup_{i=1}^mE_i.$$
Since
$$K^W+l\Hh_W=(l+1)g^\ast (K+A)+\sum(\tau_i-e_i)E_i-\tau Q+(\widehat{B}_W-B_W),$$
$K+A$ is big and nef, and the divisorial part of $\Bb_+(g^\ast (K+A)/U)$ is $\Supp(\widehat{B}_W-B_W)\cup\cup_{i=1}^mE_i$, by Lemmas \ref{lem: sbl of movable big divisor} and \ref{lem: sbl under pullback},
we have
$$N_{\sigma}(K^W+l\Hh_W/U)=\sum(\tau_i-e_i)E_i-\tau Q+(\widehat{B}_W-B_W).$$
By Lemma \ref{lem: nz for lc divisor}, all $E_i$ and all components of $S_W$ are contracted by $\psi\circ\phi$. Therefore, the induced birational map
$$\mu: X\dashrightarrow V$$
does not extract any divisor and contracts all the divisors that are contained in $S$. Thus 
$$\mu_\ast B=\widehat{B}_V=B_V,$$
so $\mu_\ast H=\mu_\ast (K+A)=K^V$. Since $\psi\circ\phi$ is $\Hh_W$-trivial and $\Hh_W=g^\ast H$, $\mu$ is $H$-trivial. Therefore,
$$(g\circ p)^\ast (K+A)=q^\ast K^V,$$
so
$$(l+1)(g\circ p)^\ast (K+A)=(g\circ p)^\ast (K+A+l\Hh_X)=q^\ast (K^V+l\Hh_V).$$
Since $K^V+l\Hh_V$ is semi-ample$/U$, $K+A$ is semi-ample$/U$.

\medskip

\noindent\textbf{Step 7.} In this step we prove (1-3) and conclude the proof of the theorem.

(1) Since $K+A$ is Cartier, $K^V+l\Hh_V$ is Cartier and 
$$K^V+l\Hh_V=(1-t)(K_V+\Delta_V+\Nn_V+\Aa'_V).$$
Thus by the usual base-point-freeness theorem, we have that 
$$\mathcal{O}_{V}(n(K^V+l\Hh_{V}))=\mathcal{O}_{V}(n(l+1)\Hh_{V})$$ is globally generated$/U$ for any integer $n\gg 0$. Since $\psi\circ\phi$ is $\Hh_W$-trivial, $(V,\Delta_V,\Nn+\frac{1}{1-t}\Hh+\Aa')$ is lc, $K^V+(l+1)\Hh_{V}$ is Cartier, and 
$$K^V+(l+1)\Hh_{V}=(1-t)\left(K_V+\Delta_V+\Nn_V+\frac{1}{1-t}\Hh_V+\Aa'_V\right).$$
Thus by the usual base-point-freeness theorem, we have that
$$\mathcal{O}_{V}(n(K^V+(l+1)\Hh_{V}))=\mathcal{O}_{V}(n(l+2)\Hh_{V})$$ is globally generated$/U$ for any integer $n\gg 0$. Since $l+1$ and $l+2$ are co-prime, 
$\mathcal{O}_{V}(n\Hh_{V})$
is globally generated$/U$ for any integer $n\gg 0$, hence $$\mathcal{O}_{X}(n\Hh_X)=\mathcal{O}_{X}(n(K+A))$$
is globally generated$/U$ for any integer $n\gg 0$. This implies (1).

(2) By \textbf{Step 6}, $K+A$ defined a contraction $c: X\rightarrow T$ and there exists a contraction $c_V: V\rightarrow T$ defined by $K^V+l\Hh_V$. Then $c_V$ is also defined by $K_V+\Delta_V+\Nn_V+\Aa'_V$. 

Since $V$ is $\Qq$-factorial, by \cite[Lemma 3.4]{HL22}, there exists a klt pair $(V,D_V)$ such that
$$K_V+D_V\sim_{\mathbb R,U}K_V+\Delta_V+\Nn_V+\Aa'_V,$$
and $c_V$ is also defined by $K_V+D_V$. Therefore, $(T,D_T:=(c_V)_\ast D_V)$ is klt, so $T$ is potentially klt.

(3.a) It holds by construction.

(3.b) By assumption, $L-K$ is ample$/T$. By (1), $\mathcal{O}_X(nL)$ is globally generated$/T$ for any integer $n\gg 0$. Thus $nL=\cont_R^\ast L_{n,T}$ and  $(n+1)L=\cont_R^\ast L_{n+1,T}$ for some line bundles $L_{n,T}$ and $L_{n+1,T}$. We may let $L_T:=L_{n+1,T}-L_{n,T}$.
\end{proof}

\begin{prop}\label{prop: q-factoriality preserved for divisorial contraction}
 Let $(X,\Ff,B,\Mm,t)/U$ be an lc algebraically integrable adjoint foliated structure such that $X$ is $\Qq$-factorial klt. Let $K:=K_{(X,\Ff,B,\Mm,t)}$ and let $\cont_R: X\rightarrow T$ be the contraction of a $K$-negative extremal ray $R$. Assume that $\cont_R$ is a divisorial contraction. Then:
 \begin{enumerate}
     \item the exceptional locus of $\cont_R$ is irreducible;
    \item $T$ is $\Qq$-factorial; and 
    \item $\rho(X/U)=\rho(T/U)+1$.
 \end{enumerate}
\end{prop}
\begin{proof}
By Theorem \ref{thm: cone}, $R$ is a $K$-negative exposed ray in $\overline{NE}(X/U)$. By \cite[Lemma 8.4.1]{CHLX23} there exists an ample$/U$ $\Rr$-divisor $A$ such that $K+A$ is the supporting function of $R$. By Theorem \ref{thm: eo flipping and divisorial contraction}, the exceptional locus of $\cont_R$ is covered by curves $C$ such that $[C]=R$. Since $\cont_R$ is a divisorial contraction, there exists a prime divisor $E$ on $X$ such that $\cont_R$ contracts $E$. 

(1) Since $\overline{NE}(X/T)=R$, $E$ is either ample$/T$, or trivial$/T$, or anti-ample$/T$. By the negativity lemma, $E$ is anti-ample$/T$. Thus for any curve $C$ that is contracted by $\cont_R$, by Theorem \ref{thm: eo flipping and divisorial contraction}, $E\cdot C<0$, so $C\subset E$. This implies (1). 

(2) Let $D_T$ be any Weil divisor on $T$ and $D$ the strict transform of $D_T$ on $X$. Then there exists a unique rational number $t$ such that $(D+tE)\cdot R=0$. By Theorem \ref{thm: eo flipping and divisorial contraction}, for $m>0$ sufficiently divisble, $m(D+tE)=\cont_R^\ast H_T$ for some Cartier divisor $H_T$ on $T$. Thus $H_T=(\cont_R)_\ast mD=mD_T$ is Cartier. Therefore, $T$ is $\Qq$-factorial.

(3) By Theorem \ref{thm: cone}, $R$ is spanned by a rational curve $C$. Consider the maps
$$0\rightarrow\Pic(T/U)\xrightarrow{D\rightarrow \cont_R^\ast D}\Pic(X/U)\xrightarrow{L\rightarrow (L\cdot C)}\mathbb Z.$$
Since $\cont_R$ is a contraction, $\Pic(T/U)\xrightarrow{D\rightarrow \cont_R^\ast D}\Pic(X/U)$ is an injection. By Theorem \ref{thm: eo flipping and divisorial contraction}, for any $L\in\Pic(X/U)$, if $L\cdot C=0$, then $L\cong\cont_R^\ast L_T$ for some line bundle $L_T$ in $T$. In particular, $L$ and $(\cont_R)^\ast L_t$ corresponds to the same element in $\Pic(X/U)$. Thus the sequence above is exact, and we get (3).
\end{proof}

\begin{rem}
    We note that in Theorem \ref{thm: eo flipping and divisorial contraction}, we do not need that $t<1$. Therefore, by letting $t=1$, we get a strengthened result compared to \cite[Theorem 1.2(1)]{LMX24b} for divisorial contractions and flipping contractions: we still need $X$ to be potentially klt but we do not require $B\geq\Delta$ for some klt pair $(X,\Delta)$. 
\end{rem}

\subsection{Proof of the general type base-point-freeness theorem}

\begin{lem}\label{lem: afs big and nef equivalent to +a}
Let $(X,\Ff,B,\Mm,t)/U$ be a klt algebraically integrable adjoint foliated structure such that $K:=K_{(X,\Ff,B,\Mm,t)}$ is nef$/U$ and big$/U$, and $X$ is potentially klt. Let $s>0$ be a real number. Then there exists a klt algebraically integrable adjoint foliated structure $(X,\Ff,B',\Mm,t)/U$ and an ample$/U$ $\Rr$-divisor $A$, such that
$$(1+s)K=tK_{\Ff}+(1-t)K_X+B'+A+\Mm_X.$$
\end{lem}
\begin{proof}
There exists $E\geq 0$ and an ample$/U$ $\Rr$-divisors $A_n$ such that 
$$K\sim_{\mathbb R,U}A_n+\frac{1}{n}E$$
for any positive integer $n$. Possibly replacing $A_n$ with $K-\frac{1}{n}E$ we may assume that $K=A_n+\frac{1}{n}E$. By Lemma \ref{lem: perturb qdlt afs}, there exists a positive integer $n$ such that $(X,\Ff,B+\frac{1}{n}E,\Mm,t)$ is klt. We may take $B':=B+\frac{1}{n}E$ and $A:=A_n$.
\end{proof}

\begin{proof}[Proof of Theorem \ref{thm: bpf intro}]
By Lemma \ref{lem: afs big and nef equivalent to +a} for any positive integer $m\geq 2$, there exists a klt algebraically integrable adjoint foliated structure $(X,\Ff,B_m,\Mm,t)$ and an ample$/U$ $\Rr$-divisor $A_n$ on $X$ such that 
$mK=K_m+A_m:=tK_{\Ff}+(1-t)K_X+B_m+\Mm_X+A_m$.

(1) By Theorem \ref{thm: eo flipping and divisorial contraction}, $2K$ is semi-ample$/U$, so $K$ is semi-ample$/U$.

(2) If $a\geq 2$, then by Theorem \ref{thm: eo flipping and divisorial contraction}(1), $$\mathcal{O}_X(n(aK))=\mathcal{O}_X(n(tK_{\Ff}+(1-t)K_X+B_a+\Mm_X+A_a))$$
is globally generated$/U$ and we are done. So we may assume that $a=1$. Then there exist integers $p_0,q_0$ such that $\mathcal{O}_X(2pK)$ and $\mathcal{O}_X(3qK)$ are globally generated$/U$ for any integers $p\geq p_0$ and $q\geq q_0$. Since any $n\geq 2p_0+3q_0+2$ can be written as $2p+3q$ for some $p\geq p_0$ and $q\geq q_0$, $\mathcal{O_X}(nK)$ is globally generated$/U$ for any $n\geq 2p_0+3q_0+2$.
\end{proof}

\subsection{Existence of Fano contractions}

At this point, we cannot prove the existence of Fano contractions for algebraically integrable adjoint foliated structures on potentially klt varieties due to non-$\Qq$-factoriality. Nevertheless, we can prove the existence of Fano contractions on $\Qq$-factorial klt varieties. Indeed, we have the following slightly stronger result:

\begin{thm}\label{thm: contraction not big}
Let  $(X,\Ff,B,\Mm,t)/U$ be an lc algebraically integrable adjoint foliated structure associated with $\pi: X\rightarrow U$. Let $K:=K_{(X,\Ff,B,\Mm,t)}$. Let $R$ be a $K$-negative extremal ray$/U$ and $H_R$ a supporting function$/U$ of $R$, such that $H_R$ is not big$/U$. Assume that there exist an $\Rr$-divisor $B_X$ and a nef$/U$ $\bb$-divisor $\Nn$ satisfying the following.
\begin{itemize}
    \item $K_X+B_X+\Nn_X$ is $\Rr$-Cartier.
    \item $\Mm-\Nn$ is nef$/U$.
    \item $B^{\ninv}\geq B_X^{\ninv}$ and $B^{\inv}\geq (1-t)B_X^{\inv}$.
    \item If $t=1$, then $(X,B,\Nn)$ is lc.
\end{itemize}
Then $R$ is also a $(K_X+B_X+\Nn_X)$-negative extremal ray$/U$. In particular,
there exists a contraction$/U$ $\cont_R: X\rightarrow T$ of $R$ satisfying the following:
\begin{enumerate}
    \item For any integral curve $C$ such that $\pi(C)$ is a point, $\cont_R(C)$ is a point if and only if $[C]\in R$.
    \item Let $L$ be a line bundle on $X$ such that $L\cdot R=0$. Then there exists a line bundle $L_T$ on $T$ such that $L\cong \cont_R^\ast L_T$.
    \item If $H_R$ is Cartier, then $\mathcal{O}_X(mH_R)$ is globally generated$/U$ for any integer $m\gg 0$.
    \item If $X$ is $\Qq$-factorial, then $T$ is $\Qq$-factorial.
\end{enumerate}
\end{thm}
\begin{proof}
First we prove (1-3). If $t=0$, then (1-3) follow from the contraction theorem for usual generalized pairs \cite[Theorems 2.2.1(4), 2.2.6]{CHLX23}. Thus we may assume that $t>0$.

By Proposition \ref{prop: adjoint lc implies X lc}, $(X,B,\Nn)$ is lc. First we show that $R$ is a $(K_X+B_X+\Nn_X)$-negative extremal ray$/U$.

Let $X\rightarrow U'\rightarrow U$ the Stein factorization of $\pi$. Possibly replacing $U$ with $U'$, we may assume that $\pi$ is a contraction. By Theorem \ref{thm: cone intro}, $R$ is an exposed ray in $\overline{NE}(X/U)$. By \cite[Lemma 8.4.1]{CHLX23}, 
we may assume that 
       $$H_R=tK_{\Ff}+(1-t)K_X+B+A+\Mm_X$$
for some ample$/U$ $\Rr$-Cartier $\Rr$-divisor $A$ on $X$. 

Let $B_{\Ff}$ be the unique $\Rr$-divisor on $X$ such that 
$$tB_{\Ff}+(1-t)B_X=B$$
and let $\Pp$ be the unique $\bb$-divisor on $X$ such that
$$t\Pp+(1-t)\Nn=\Mm.$$
Then $B_{\Ff}\geq B^{\ninv}\geq B_X^{\ninv}\geq 0$, $\Pp$ and $\Pp-\Nn$ are nef$/U$, and
$$t(K_{\Ff}+B_{\Ff}+\Pp_X)+(1-t)(K_X+B_X+\Nn_X)=K.$$
Therefore, we may assume that $R$ is $(K_{\Ff}+B_{\Ff}+\Pp_X)$-negative, otherwise $R$ is automatically $(K_X+B_X+\Nn_X)$-negative and we are done. By \cite[Lemma 8.4.1]{CHLX23} there exists a supporting function $H_R'$ of $R$ such that  $H_R'-(K_{\Ff}+B_{\Ff}+\Pp_X)$ is ample$/U$. Possibly replacing $H_R$ with $H_R+H_R'$ we may assume that $A_{\Ff}:=H_R-(K_{\Ff}+B_{\Ff}+\Pp_X)$ is ample$/U$.

Let $F$ be a general fiber of $\pi$. Then $H_F:=H_R|_F$ is nef, not big, and is not numerically trivial. Let $q:=\dim F$ and $A_F:=A_{\Ff}|_F$, then there exists an integer $1\leq k\leq q-1$ such that
$$H_F^k\cdot A_F^{q-k}>H_F^{k+1}\cdot A_F^{q-k-1}=0.$$

Let $D_i:=H_R$ for any $1\leq i\leq k+1$, and let $D_i:=A_{\Ff}$ for any $k+2\leq i\leq q$. Then
$$(D_1|_F)\cdot (D_2|_F)\cdots\dots\cdot (D_q|_F)=H_F^{k+1}\cdot A_F^{q-k-1}=0$$
and
$$-(K_{\Ff}+B_{\Ff}+\Pp_X)|_F\cdot (D_2|_F)\cdots\dots\cdot (D_q|_F)=(A_F-H_F)\cdot H_F^{k}\cdot A_F^{q-k-1}=H_F^{k}\cdot A_F^{q-k}>0.$$
Let $M:=H_R+A_{\Ff}=K_{\Ff}+B_{\Ff}+\Pp_X+2A_{\Ff}$. Then $M$ is nef$/U$. By \cite[Theorem 8.1.1]{CHLX23}, for any general closed point $x\in X$, there exists a rational curve $C_x$ such that $x\in C_x$, $\pi(C_x)$ is a closed point, $C_x$ is tangent to $\Ff$, and 
$$0=D_1\cdot C_x=H_R\cdot C_x.$$
In particular, $C_x$ spans $R$.

By Theorem \ref{thm: eo acss model}, there exists a $\Qq$-factorial proper  ACSS modification $$h\colon (X',\Ff',B_{\Ff'},\Pp;G)/Z\rightarrow (X,\Ff,B_{\Ff},\Pp).$$ Then $G$ contains any $\Ff'$-invariant prime divisor that is either $h$-exceptional or is a component of $\Supp h_\ast ^{-1}B_{\Ff}$, and $\Supp B_{\Ff'}$ contains any $h$-exceptional non-$\Ff'$-invariant prime divisor. We have
$$K_{\Ff'}+B_{\Ff'}+\Pp_{X'}=h^\ast (K_{\Ff}+B_{\Ff}+\Pp_X)-L'$$
for some $L'\geq 0$. Since $(X,B_X,\Nn)$ is lc and $B_{\Ff}\geq B_X^{\ninv}$, we have
$$K_{X'}+B_{\Ff'}+G+\Nn_{X'}\geq h^\ast (K_X+B_X+\Nn_X).$$

Let $x$ be a general closed point in $X$ and let $C_x'$ be the strict transform of $C_x$ on $X'$.  Since $x$ is a general closed point in $X$ and $C_x$ is tangent to $\Ff$, $C_x'$ is tangent to $\Ff'$.  Let $A_{\Ff'}:=h^\ast A_{\Ff}$. Then
\begin{align*}
  0&=H_R\cdot C_x=h^\ast H_R\cdot C_x'=h^\ast (K_{\Ff}+B_{\Ff}+\Pp_X+A)\cdot C_x'\geq (K_{\Ff'}+B_{\Ff'}+\Pp_{X'}+A')\cdot C_x'\\
  &=(K_{X'}+B_{\Ff'}+G+\Pp_{X'}+A')\cdot C_x'\geq (h^\ast (K_X+B_X+\Nn_X+A)+(\Pp_{X'}-\Nn_{X'}))\cdot C_x'\\
  &\geq (K_X+B_X+\Nn_X+A)\cdot C_x>(K_X+B_X+\Nn_X)\cdot C_x.
\end{align*}
Therefore, $R$ is a $(K_X+B_X+\Nn_X)$-negative extremal ray$/U$. Now (1-3) follow from \cite[Theorems 2.2.1(4), 2.2.6]{CHLX23}. 

Finally, (4) follows from the sames line of \cite[Proof of Corollary 3.18]{KM98} (or, since we may assume that $B_X=0$ and $\Nn=\bm{0}$ when $X$ is $\Qq$-factorial, (4) follows from the usual contraction theorem for $\Qq$-factorial lc pairs).
\end{proof}

\begin{rem}
    We note that in Theorem \ref{thm: contraction not big}, we do not need that $t<1$. Therefore, by letting $t=1$, we get a strengthened result compared to \cite[Proposition 7.1]{LMX24b} for divisorial contractions and flipping contractions: in the notation of \cite[Proposition 7.1]{LMX24b}, we only need $B\geq\Delta^{\ninv}\geq 0$, and do not need $B\geq\Delta$.
\end{rem}

Summarizing Theorem \ref{thm: contraction not big} and Theorem \ref{thm: eo flipping and divisorial contraction}, we obtain the following contraction theorem for algebraically integrable adjoint foliated structures.

\begin{thm}\label{thm: contraction theorem strong}
    Let $(X,\Ff,B,\Mm,t)/U$ be an lc algebraically integrable adjoint foliated structure such that $X$ is potentially klt. Let $K:=K_{(X,\Ff,B,\Mm,t)}$. Let $R$ be a $K$-negative extremal ray and $H_R$ a supporting function$/U$ of $R$. 
    
    Let $B_X$ be an $\Rr$-divisor on $X$ such that $B^{\ninv}\geq B_X^{\ninv}\geq 0$ and $B_X^{\inv}\geq (1-t)B^{\inv}\geq 0$. Let $\Nn$ be a $\bb$-divisor on $X$ such that $\Mm-\Nn$ and $\Nn$ are nef$/U$. Assume that one of the following conditions hold:

(i) $H_R$ is big$/U$. (ii) $X\rightarrow U$ is birational. (iii) $K$ is pseudo-effective$/U$. 

(iv) $(X,B_X,\Nn)$ is lc. (v) $t<1$ and $K_X+B_X+\Nn_X$ is $\Rr$-Cartier.  

(vi) $(X,\Ff,B,\Mm,t)$ is qdlt. (vii) $X$ is lc. (viii) $X$ is $\Qq$-factorial. 

\noindent Then there exists a contraction$/U$ $\cont_R: X\rightarrow T$ of $R$ satisfying the following:
\begin{enumerate}
    \item For any integral curve $C$ such that $\pi(C)$ is a point, $\cont_R(C)$ is a point if and only if $[C]\in R$.
    \item Let $L$ be a line bundle on $X$ such that $L\cdot R=0$. Then there exists a line bundle $L_T$ on $T$ such that $L\cong \cont_R^\ast L_T$.
    \item If $H_R$ is Cartier, then $\mathcal{O}_X(mH_R)$ is globally generated$/U$ for any integer $m\gg 0$.
    \end{enumerate}
\end{thm}
\begin{proof}
(i) The theorem follows from Theorem \ref{thm: eo flipping and divisorial contraction}. 

(ii) Since $\pi$ is birational, $H_R$ is big$/U$, so (ii) follows from (i). 

(iii) $R$ is a $(K+A)$-negative extremal ray$/U$ for some ample$/U$ $\Rr$-divisor $A$. Thus $H_R$ is big so (iii) follows from (i).

From now on we may assume that $H_R$ is not big$/U$.

(iv) The theorem follows from Theorem \ref{thm: contraction not big}.

(v) By Proposition \ref{prop: adjoint lc implies X lc}, $(X,B_X,\Nn)$ is lc, so (v) follows from (iv).

(vi) We may take $B_X=B^{\ninv}+\frac{1}{1-t}B^{\inv}$ and $\Mm=\Nn$ and (vi) follows from (iv).

(vii) We may take $B_X=0$ and $\Mm=\bm{0}$ and (vii) follows from (iv).

(viii) Since $X$ is potentially klt, $X$ is klt and (viii) follows from (vii).
\end{proof}

\begin{proof}[Proof of Theorem \ref{thm: cont intro}]
It is a special case of Theorem \ref{thm: contraction theorem strong}.
\end{proof}

\section{Existence of flips}\label{sec: eof}

\begin{thm}\label{thm: eof and small modification after divisorial contraction}
    Let $(X,\Ff,B,\Mm,t)/U$ be an lc algebraically integrable adjoint foliated structure such that $X$ is potentially klt. Let $K:=K_{(X,\Ff,B,\Mm,t)}$. Let $c: X\rightarrow T$ be a divisorial contraction$/U$ or a flipping contraction$/U$ of a $K$-negative extremal ray$/U$. Then $(X,\Ff,B,\Mm,t)/T$ has a $\Qq$-factorial good minimal model. 
    
    Moreover, let $X\dashrightarrow X^+$ be the ample model$/T$ of $K$, then $X^+$ is potentially klt, and if $c$ is a flipping contraction, then $X\dashrightarrow X^+$ is a $K$-flip.
\end{thm}
\begin{proof}
Let $A:=-K$, then $A$ is ample$/T$ and $c$ is the contraction$/T$ associated to $K+A$. By Theorem \ref{thm: eo flipping and divisorial contraction}, $T$ is potentially klt, so there exists a small $\Qq$-factorialization $h: T'\rightarrow T$. Let $\Ff_{T'}:=h^{-1}c_\ast \Ff$ and let $B_{T'}$ be the image of $B$ on $T'$. Since $T'$ is of Fano type over $T$, we may run a $(tK_{\Ff_{T'}}+(1-t)K_{X'}+B_{T'}+\Mm_{T'})$-MMP$/T$ which terminates with a good minimal model $(\bar X,\bar\Ff,\bar B,\Mm,t)/T$ of $(T',\Ff_{T'},B_{T'},\Mm,t)/T$. In particular, $\bar X$ is $\Qq$-factorial.  Since $K$ is anti-ample$/T$ and $\bar K:=tK_{\bar\Ff}+(1-t)K_{\bar X}+\bar B+\Mm_{\bar X}$ is ample$/T$, $(\bar X,\bar\Ff,\bar B,\Mm,t)/T$ is a $\Qq$-factorial good minimal model of $(X,\Ff,B,\Mm,t)/T$

Since $h$ is small, the induced morphism $\bar X\rightarrow T$ is small, so the induced morphism $X^+\rightarrow T$ is small. Since $T$ is potentially klt, $X^+$ is potentially klt. The moreover part follows.
\end{proof}

\begin{proof}[Proof of Theorem \ref{thm: eof intro}]

By Theorem \ref{thm: eof and small modification after divisorial contraction} we only need to prove the moreover part. By Theorem \ref{thm: eof and small modification after divisorial contraction} again, there exists a $\Qq$-factorial good minimal model$/T$ $(\bar X,\bar\Ff,\bar B,\Mm,t)/T$ of $(X,\Ff,B,\Mm,t)/T$. Since $X\dashrightarrow\bar X$ does not extract any divisor, we have
$$1=\rho(X/T)\geq\rho(\bar X/T).$$
Since $\bar X\not=T$, $\rho(\bar X/T)\geq 1$, so $\rho(\bar X/T)=1$. Thus $\bar X=X^+$, so $X^+$ is $\Qq$-factorial.
\end{proof}

\begin{rem}
    We note that in Theorem \ref{thm: eof intro}, we do not need that $t<1$. Therefore, by letting $t=1$, we get a strengthened result comparing to \cite[Theorem 1.2(2)]{LMX24b}: we still need $X$ to be potentially klt but we do not require $B\geq\Delta$ for some klt pair $(X,\Delta)$. 
\end{rem}

\section{Existence of MMP for adjoint foliated structures}\label{sec: run mmp}

In this section we prove the existence of the minimal model program, Theorem \ref{thm: can run mmp intro}. Moreover, we show the existence of minimal model program with scaling.

\begin{lem}\label{lem: properties preserved under mmp}
Let $(X,\Ff,B,\Mm,t)/U$ be an lc algebraically integrable adjoint foliated structure such that $X$ is potentially klt. Let $K:=K_{(X,\Ff,B,\Mm,t)}$. Then for any birational map $\phi: (X,\Ff,B,\Mm,t)\dashrightarrow (X',\Ff',B',\Mm,t)$ that is a sequence of steps of a $K$-MMP$/U$, we have the following:
\begin{enumerate}
    \item $X'$ is potentially klt.
    \item If $X$ is $\Qq$-factorial, then $X'$ is $\Qq$-factorial.
    \item If $D$ is an $\Rr$-Cartier $\Rr$-divisor on $X$, then $\phi_\ast D$ is an $\Rr$-Cartier $\Rr$-divisor.
    \item Let $B_X\geq 0$ be an $\Rr$-divisor such that $B^{\ninv}\geq B_X^{\ninv}$ and $B^{\inv}\geq (1-t)B_X^{\inv}$. Let $B_{X'}:=\phi_\ast B_X$. Let $\Nn$ be a nef$/U$ $\bb$-divisor such that $\Mm-\Nn$ is nef$/U$. Then:
    \begin{enumerate}
     \item If $(X,B_X,\Nn)$ is lc, then $(X',B_{X'},\Nn)$ is lc.
     \item If $(X,B_X,\Nn)$ is klt and $(X,\Ff,B,\Mm,t)$ is klt, then $(X',B_{X'},\Nn)$ is klt.
    \end{enumerate}
    \item If $(X,\Ff,B,\Mm,t)$ is qdlt, then $(X',\Ff',B',\Mm,t)$ is qdlt.
\end{enumerate}
\end{lem}
\begin{proof}
We may assume that $\phi$ is a single step of a $K$-MMP$/U$.

Then there are contractions $g: X\rightarrow T$ and $g': X'\rightarrow T$ such that $g'\circ\phi=g$ and $g$ is the contraction of a $K$-negative extremal ray$/U$. By Theorem \ref{thm: cone}, $R$ is exposed. By \cite[Lemma 8.4.1]{CHLX23}, there exists a supporting function $H$ of $R$ such that $H=K+A$ for some ample$/U$ $\Rr$-divisor $A$. Since $g$ is birational, $H$ is big$/U$.

(1) It is a part of Theorem \ref{thm: eof and small modification after divisorial contraction}.

(2) If $\Exc(g)$ contains any divisor, then by Proposition \ref{prop: q-factoriality preserved for divisorial contraction}, $T=X'$ is $\Qq$-factorial. If $\Exc(g)$ does not contain any divisor, then by Theorem \ref{thm: eof intro}, $X'$ is $\Qq$-factorial.

(3) We have $(D+\lambda K)\cdot R=0$ for some real number $\lambda$. By Theorem \ref{thm: eo flipping and divisorial contraction}(2.b), $D+\lambda K\sim_{\mathbb R}g^\ast L$ for some $\Rr$-Cartier $\Rr$-divisor $L_T$ on $T$. Thus $\phi_\ast D+K'=g'^\ast L$ is $\Rr$-Cartier. Since $K'$ is $\Rr$-Cartier, $\phi_\ast D$ is $\Rr$-Cartier.

(4) By (3), $K_{X'}+B_{X'}+\Nn_{X'}$ is $\Rr$-Cartier. (4) follows from Proposition \ref{prop: adjoint lc implies X lc}.

(5) It follows from (3) and Proposition \ref{prop: qdlt preserved under mmp}.
\end{proof}

\begin{thm}\label{thm: run mmp strong} Let $(X,\Ff,B,\Mm,t)/U$ be an lc algebraically integrable adjoint foliated structure such that $X$ is potentially klt. Let $K:=K_{(X,\Ff,B,\Mm,t)}$. Assume that one of the conditions (ii)-(viii) of Theorem \ref{thm: contraction theorem strong} hold. Then we may run a $K$-MMP$/U$.
\end{thm}
\begin{proof}
First we show that we can run a step of a $K$-MMP$/U$. Let $R$ be a $K$-negative extremal ray and $H_R$ a supporting function of $K$. The existence of the contraction of $R$ $g: X\rightarrow T$ follows from Theorem \ref{thm: contraction theorem strong}. If the contraction is a Fano contraction then we are done. Otherwise, by Theorem \ref{thm: eof and small modification after divisorial contraction}, $K$ has an ample model$/T$ $X'$ and $X\dashrightarrow X'$ is a step of a $K$-MMP$/U$. 

By Lemma \ref{lem: properties preserved under mmp}, conditions (ii)-(viii) of Theorem \ref{thm: contraction theorem strong} are preserved, so we may replace $X$ with $X'$ and continue this process. 
\end{proof}

\begin{prop}[$=$Theorem \ref{thm: lc afs ambient has nice singularities intro}(1)]\label{prop: klt afs imply potentially klt}
Let $(X,\Ff,B,\Mm,t)/U$ be a klt algebraically integrable adjoint foliated structure. Then $X$ is potentially klt.
\end{prop}
\begin{proof}
We may assume that $\Ff\not=T_X$, otherwise $(X,\Ff,B,\Mm,t)=(X,B,\Mm)$ and the proposition follows from \cite[Lemma 3.4]{HL22}. Since $\Ff$ is algebraically integrable, there exists an $\Ff$-invariant divisor $S$ on $X$. If $t=1$, then $S$ is an lc place of $(X,\Ff,B,\Mm,1)$, which contradicts the assumption that $(X,\Ff,B,\Mm,t)/U$ is klt. Thus we may assume that $t<1$.

Let $h: (X',\Ff',B',\Mm,t)\rightarrow (X,\Ff,B,\Mm,t)$ be a $\Qq$-factorial qdlt modification whose existence is guaranteed by Theorem \ref{thm: qdlt model intro}. Let $A$ be an ample divisor on $X$.
    Since $(X,\Ff,B,\Mm,t)$ is klt, $h$ is small. Thus there exists an $h$-anti-ample divisor $E\geq 0$. Then $A':=h^\ast A-eE$ is ample$/U$ for any $0<e\ll 1$. By Lemma \ref{lem: perturb qdlt afs}, $(X',\Ff',B'+eE,\Mm,t)$ is klt for any $0<e\ll 1$. Now the contraction $h: X'\rightarrow X$ is the contraction$/X$ associated to
    $$K_{(X',\Ff',B'+eE,\Mm,t)}+A'.$$
    By Theorem \ref{thm: eo flipping and divisorial contraction}(2), $X$ is potentially klt.
\end{proof}

\begin{proof}[Proof of Theorem \ref{thm: lc afs ambient has nice singularities intro}]
(1) follows from  Proposition \ref{prop: klt afs imply potentially klt} and (2) follows from Proposition \ref{prop: adjoint lc implies X lc}.
\end{proof}

\begin{proof}[Proof of Theorem \ref{thm: can run mmp intro}]
    By Proposition \ref{prop: klt afs imply potentially klt}, $X$ is potentially klt, so $X$ is klt, so Theorem \ref{thm: can run mmp intro} is a special case of Theorem \ref{thm: run mmp strong}.
\end{proof}

\begin{proof}[Proof of Theorem \ref{thm: simplified main theorem}]
    It is a special case of Theorem \ref{thm: can run mmp intro}.
\end{proof}

\begin{proof}[Proof of Theorem \ref{thm: small improvement LMX24b}]
    We consider $(X,\Ff,B)/U$ as the adjoint foliated structure $$(X,\Ff,B,\bm{0},1)/U.$$ Under condition (1), the existence of the $(K_{\Ff}+B)$-MMP$/U$ follows from Theorem \ref{thm: run mmp strong} (using condition (iii) of Theorem \ref{thm: contraction theorem strong}). Under condition (2), the existence of the $(K_{\Ff}+B)$-MMP$/U$ follows from Theorem \ref{thm: run mmp strong} (using condition (iv) of Theorem \ref{thm: contraction theorem strong}).

    Now we prove the moreover part. By \cite[Theorem 1.10]{LMX24b}, $(X,\Ff,B+A)/U$ has a minimal model or Mori fiber space in the sense of Birkar-Shokurov. The theorem follows from \cite[Theorem 1.11(1)]{LMX24b}.
\end{proof}

\appendix

\section{ACC type conjectures}\label{sec: acc}

\begin{rem}
We state the conjectures in this section for arbitrary foliations in arbitrary dimensions. 
However, due to the lack of the existence of suitable (partial) resolution statements for general foliations on varieties,
which in turn leads to the fact that the Minimal Model Program has yet to be estrablished in full generality for , 
we only expect the conjectures we will state in this section to be approachable either for algebraically integrable foliations or in dimension $\leq 3$ for now. 
Most of these conjectures are widely open in dimension $\geq 3$, even for in the algebraically integrable case, and some are open already in dimension $2$.

For simplicity, we only state the conjectures for adjoint foliated structures with moduli part $\Mm=\bm{0}$. But, of course, it is interesting also to consider the  versions with $\Mm\not=\bm{0}$.
\end{rem}

Instead of considering the usual lc thresholds, which are defined both for $K_X$ and $K_\mathcal{F}$ measuring the singularities with respect to effective $\Rr$-divisors, the adjoint lc threshold is computed just in terms of the parameter $t$ appearing in the definition of adjoint foliated structure. 
More precisely, given a $\Qq$-factorial normal variety $X$ and a foliation $\mathcal F$ on it, we define \emph{interpolated lc threshold} of $(X,\Ff)$ to be
\begin{align*}
\bar t:=\sup\{t\in [0,1]\mid (X,\Ff,t)\text{ is lc}\}.
\end{align*}
This definition was first introduced by M\textsuperscript{c}Kernan cf. \cite[25:00]{McK22}. 

Proposition \ref{prop: adjoint lc implies X lc} implies that either $\bar t=1$, or $(X,\Ff,t)$ is lc for any $t\leq\bar t$, at least when $\Ff$ is algebraically integrable. 
The \emph{interpolated lc threshold} is an interesting invariant for adjoint foliated structures given that it provides information about the singularities of the foliation and the underlying, at the same time.
We expect the interpolated lc threshold $s_0$ to play a crucial role in studying the behavior of the foliations via adjunction.

At the 2022 JAMI conference, M\textsuperscript{c}Kernan proposed the following conjecture on the ACC for interpolated lc thresholds indicating that the set of thresholds $\bar{t}$ such that $(1-\bar t)K_X+\bar t(K_{\Ff}+B)$ 
has an lc structure satisfies the ACC in any fixed dimension. 

\begin{conj}[M\textsuperscript{c}Kernan's ACC conjecture for interpolated lc thresholds, {\cite[28:22]{McK22}}]\label{conj: mckernan conjecture}
Let $d$ be a positive integer and let $\Ii\subset [0,+\infty)$ be a DCC set. 
Then there exists an ACC set $\Ii'=\Ii'(d, \Ii)$ satisfying the following property: 
\\
Let $(X,\Ff,B)$ be a foliated triple such that $K_{\Ff}+B$ and $K_X$ are $\Rr$-Cartier, $\dim X=d$, and $B\in\Ii$. Then
\begin{align*}
\sup\{t\in [0,1]\mid (X,\Ff,tB,t)\text{ is lc}\}
\end{align*}
belongs to $\Ii'$.
\end{conj}

We propose the following more general versions of the above conjecture on ACC conjecture for interpolated lc thresholds to hold.

\begin{conj}[M\textsuperscript{c}Kernan's ACC conjecture for interpolated lc thresholds, variation 1]\label{conj: strong mckernan conjecture}
 Let $d$ be a positive integer and let 
 $\Ii\subset [0,1]$ 
 be a DCC set. 
 Then there exists an ACC set $\Ii'$ satisfying the following property:
 \\
 Let $(X,B)$ be a pair of dimension $d$ such that $B\in\Ii$, and let $\Ff$ be a foliation on $X$ such that $(X,\Ff,B^{\ninv})$ is a foliated triple. Then
\begin{align*}
\sup\{t\in [0,1]\mid (X,\Ff,B^{\ninv}+(1-t)B^{\inv},t)\text{ is lc}\}
\end{align*}
belongs to $\Ii'$.
\end{conj}

\begin{conj}[ACC for lc thresholds for adjoint foliated structures]\label{conj: acc for lct afs}
 Let $d$ be a positive integer, $t\in [0,1]$ a real number, and let $\Ii\subset [0,+\infty)$ be a DCC set. 
 Then there exists an ACC set $\Ii'$ satisfying the following property:
 \\
 Let $(X,\Ff,B,t)$ be an adjoint foliated structure of dimension $d$ such that $B\in\Ii$ and $t\in\Ii$. Let $D\in\Ii$ be an $\Rr$-divisor. Then
\begin{align*}
\sup\{s\in [0,1]\mid (X,\Ff,B+s(D^{\ninv}+(1-t)D^{\inv}),t)\text{ is lc}\}
\end{align*}
belongs to $\Ii'$.
\end{conj}

We expect Conjectures \ref{conj: strong mckernan conjecture} and \ref{conj: acc for lct afs} to imply Conjecture \ref{conj: mckernan conjecture}. Although, due to technical reasons, we cannot prove this yet, we can at least prove the $\Qq$-factorial case:

\begin{prop}
Let $d$ be a positive integer. If Conjectures \ref{conj: strong mckernan conjecture} and \ref{conj: acc for lct afs} hold for $\Qq$-factorial varieties in dimension $d$, then Conjecture \ref{conj: mckernan conjecture} holds for $\Qq$-factorial varieties in dimension $d$.
\end{prop}

\begin{proof}
    Suppose Conjecture \ref{conj: mckernan conjecture} does not hold for $\Qq$-factorial varieties in dimension $d$, then we can assume that there exists a strictly increasing sequence
    $\{t_i \}_{i \in \mathbb N}$ 
    and  for all
    $i \in \mathbb N$ there exists a $\Qq$-factorial foliated triple $(X_i,\Ff_i,B_i)$ as in the the statement of  
    Conjecture \ref{conj: mckernan conjecture} such that
    $t_i=\sup\{t\in [0,1]\mid (X_i,\Ff_i,t_iB_i,t_i)\text{ is lc}\}$.
    Let 
    $\bar t:=\lim_{i\rightarrow+\infty}t_i$.
    In particular, we can assume that  $t_i<1$ for each $i$. 
    
    We write $B_i':=t_iB_i^{\ninv}+\frac{t_i}{1-t_i}B_i^{\inv}$. Let $$\Ii_1:=\left\{\bar t, t_i\gamma,\frac{t_i}{1-t_i},\frac{1}{1-t_i},\bar t\gamma,\bar t(1-\bar t)\gamma \middle\vert i\in\mathbb N^+,\gamma\in\Ii\right\}\cap [0,+\infty).$$ Then $\dim X_i=d$, $\Ii_1$ is a DCC set, $B_i'\in\Ii_1$, and
    $$(X_i,\Ff_i,t_iB_i=B_i'^{\ninv}+(1-t_i)B_i'^{\inv},t_i)$$
    is log canonical for any $i \in \mathbb N$. 
    Therefore, by Conjecture \ref{conj: strong mckernan conjecture} in dimension $d$, possibly passing to a subsequence, 
    $$(X_i,\Ff_i,B_i'^{\ninv}+(1-\bar t)B_i'^{\inv},\bar t)=\left(X_i,\Ff_i,t_i\left(B_i^{\ninv}+\frac{1-\bar t}{1-t_i}B_i^{\inv}\right),\bar t\right)$$
    is lc. By Conjecture \ref{conj: acc for lct afs} in dimension $d$, possibly passing to a subsequence, 
    $$\left(X_i,\Ff_i,\bar t\left(B_i^{\ninv}+\frac{1-\bar t}{1-t_i}B_i^{\inv}\right),\bar t\right)=\left(X_i,\Ff_i,\bar t B_i^{\ninv}+\frac{1}{1-t_i}(\bar t(1-\bar t)B_i^{\inv}),\bar t\right)$$
    is lc. By Conjecture \ref{conj: acc for lct afs} in dimension $d$ again, possibly passing to a subsequence, 
    $$\left(X_i,\Ff_i,\bar t B_i^{\ninv}+\frac{1}{1-\bar t}(\bar t(1-\bar t)B_i^{\inv}),\bar t\right)=\left(X_i,\Ff_i,\bar t B_i^{\ninv}+\bar tB_i^{\inv},\bar t\right)=(X_i,\Ff_i,\bar tB_i,\bar t)$$
    is lc. This is not possible as $\bar t>t_i$.
\end{proof}

All three conjectures above are widely open in dimension $\geq 3$, even for algebraically integrable foliations. M\textsuperscript{c}Kernan sketched a proof of Conjecture \ref{conj: mckernan conjecture} \cite{McK22} when $d=2$, while Conjectures \ref{conj: strong mckernan conjecture} and \ref{conj: acc for lct afs} remain open when $d=2$. As a very special case of these conjectures, we obtain the following $1$-gap conjecture for interpolated lc thresholds:

\begin{conj}[$1$-gap for interpolated lc thresholds]\label{conj: 1-gap lct}
Let $d$ be a positive integer. Then there exists a positive real number $\tau=\tau(d)$ satisfying the following property:

Let $(X,\Ff,t)$ be an adjoint foliated structure of dimension $d$.
Assume that $(X,\Ff,t)$ is lc for some $t>1-\tau$ and that $K_{\Ff}$ is $\Qq$-Cartier. 
Then $\Ff$ is lc.
\end{conj}

When $d=2$, Conjecture \ref{conj: 1-gap lct} was proven by the fifth and the sixth authors \cite[Lemma 2.19]{SS23} which shows that we may take $\tau(2)=\frac{1}{6}$, but the conjecture is widely open in dimension $\geq 3$, even for algebraically integrable foliations. From the point of view of explicit geometry it is also interesting to ask the following question:

\begin{ques}
Let $d$ be a positive integer. Find the optimal value of $\tau(d)$.
\end{ques}

Inspired by \cite{HMX14}, we also propose the following global ACC conjecture for adjoint foliated structures:

\begin{conj}[M\textsuperscript{c}Kernan's global ACC conjecture for adjoint foliated structures]\label{conj: interpolated global acc}
Let $d$ be a positive integer and $\Ii\subset [0,1]$ a DCC set. Then there exists a finite set $\Ii_0 = \Ii_0(d, I) \subset \Ii$ satisfying the following property:

Let $(X,\Ff,B^{\ninv}+(1-t)B^{\inv},t)$ be an lc adjoint foliated structure with $t\in\Ii$, $B\in\Ii$.
Assume that $X\rightarrow Z$ is a contraction such that 
\begin{align*}
& tK_{\Ff}+(1-t)K_X+B^{\ninv}+(1-t)B^{\inv}\equiv_Z 0,
\qquad
\text{and}
\\
& sK_{\Ff}+(1-s)K_X+B^{\ninv}+(1-s)B^{\inv}\not\equiv_Z 0
\end{align*}
for any $s \in [0, 1]$, $s \neq t$.
Then $t\in\Ii_0$ and also the coefficients of the horizontal$/Z$ part of $B$ belong to $\Ii_0$.
\end{conj}

Conjecture \ref{conj: interpolated global acc} remains widely open even in dimension 2, already in the case where $Z$ is a point. 
In fact, even when $t=1$ is fixed and $Z$ is point, 
Conjecture \ref{conj: interpolated global acc} is non-trivial, and it has only recently been proven in dimension $\leq 3$ by~\cite[Theorem 2.5]{Che22},~\cite[Theorem 1.1]{LMX24a}, and by~\cite[Theorem 1.2]{DLM23} in the algebraically integrable case.

We expect that Conjectures \ref{conj: strong mckernan conjecture}, \ref{conj: acc for lct afs}, and \ref{conj: interpolated global acc} may be solved via an approach based on induction on dimension. In future work, we will prove that Conjecture \ref{conj: interpolated global acc} for algebraically integrable foliations in dimension $d$ implies Conjectures \ref{conj: strong mckernan conjecture} and \ref{conj: acc for lct afs} for algebraically integrable foliations in dimension $d+1$. 
We also expect that Conjectures \ref{conj: strong mckernan conjecture} and \ref{conj: acc for lct afs} in dimension $d$ will imply Conjecture \ref{conj: interpolated global acc} in dimension $d$.

An important consequence of the ACC for lc thresholds is the existence of uniform lc rational polytopes \cite[Theorem 5.6]{HLS19}. Such result also hold for foliations in dimension $\leq 3$ \cite{LMX24a,LMX24c} or for algebraically integrable foliations \cite{DLM23}. Inspired by this, we also conjecture the existence of a uniform interpolated lc rational polytope for adjoint foliated structures:

\begin{conj}[Uniform interpolated lc rational polytope]
Let $d$ be a positive integer, $\Ii_0\subset [0,1]\cap\mathbb Q$ a finite set, and $t_0\in [0,1]$ an irrational real number. Then there exists a positive real number $\delta$ depending only on $d,\Ii_0$, and $t_0$ satisfying the following.

Assume that $(X,\Ff,B^{\ninv}+(1-t_0)B^{\inv},t_0)$ is an lc adjoint foliated structure of dimension $d$ such that $B\in\Ii_0$. Then $(X,\Ff,B^{\ninv}+(1-t)B^{\inv},t)$ is lc for any $t\in (t_0-\delta,t_0+\delta)$.
\end{conj}

\smallskip

We also remark that there is a more general version of the ACC conjecture for interpolated lc thresholds proposed by M\textsuperscript{c}Kernan. Instead of considering a normal variety $X$ and a foliation $\Ff$ on $X$, we can consider the more general setting where we are given two foliations $\Ff$ and $\Gg$ on $X$ such that $\Ff \subset \Gg$ and the adjoint foliated structure $(\Gg,\Ff,t)$, i.e., a structure such that $t \in [0,1]$ and $tK_{\Ff}+(1-t)K_{\Gg}$ is $\Rr$-Cartier. 
We may define the singularities of $(\Gg,\Ff,t)$ similarly to those of $(X,\Ff,t)$, by measuring them on all possible resolutions of $X$.
Similarly to the previous definitons, we can also define a new version of the interpolated lc thresholds in this new context simply by taking 
\begin{align*}
\sup\{t\in [0,1]\mid (\Gg,\Ff,t)\text{ is lc}\}.
\end{align*}
M\textsuperscript{c}Kernan conjectured that also in this framework the set of all possible interpolated lc threshold satisfies the ACC in any fixed dimension.

\begin{conj}[{\cite[50:55]{McK23}}]\label{conj: two foliation interpolated acc}
    Let $d$ be a positive integer. Then there exists an ACC set 
    $\Ii=\Ii(d) \subset [0, 1]$
    satisfying the following property:

    Let $X$ be a normal variety of dimension $d$ and let $\Ff\subset\Gg\subset T_X$ be foliations.
    Then,
    \begin{align*}
\sup\{t\in [0,1]\mid (\Gg,\Ff,t)\text{ is lc}\}
    \end{align*}
    belongs to $\Ii$.
\end{conj}

Of course, one may generalize Conjecture \ref{conj: two foliation interpolated acc} to the case with a boundary and also consider the corresponding global ACC conjecture. Due to the technicality of the statements, we omit the detailed formulations.

Finally, inspired by the ACC for log canonical threshold polytopes \cite[Theorem 1.1]{HLQ21}, it is also interesting to consider the case with more than two foliations, i.e., the ``ACC for interpolated lc threshold polytope". More precisely, instead of considering two foliations $\Ff \subset \Gg$ and the one-parameter combination $tK_{\Ff}+(1-t)K_{\Gg}$, we may consider a sequence of foliations $\Ff_1 \subset \Ff_2 \subset \dots \subset \Ff_n$ and the combination $\sum a_iK_{\Ff_i}$, with $a_i \in [0,1]$ and $\sum a_i = 1$, and study the ACC property of the polytope
$$\left\{(x_1,\dots,x_n)\Big| \sum x_i=1, x_i\in [0,1], \sum x_iK_{\Ff_i}\text{ has an lc adjoint foliated structure}\right\}.$$
Due to technicality we do not provide a more detailed conjecture here.

\end{document}